\documentclass[a4paper,leqno,11pt,twoside]{amsart}

\usepackage{amsmath,amsthm,amssymb,amsfonts} 
\usepackage[all]{xy}

\theoremstyle{definition} \newtheorem{defn}[subsection]{Definition}
			 \newtheorem{rmk}[subsection]{Remark}
			 \newtheorem{rmks}[subsection]{Remarks}
			 
			 \newtheorem{rules}[subsection]{Rules}
			 
			 \newtheorem{defn-prop}[subsection]{Definition-Proposition}

\theoremstyle{plain}      \newtheorem{thm}[subsection]{Theorem}

			 \newtheorem{lem}[subsection]{Lemma}
			 \newtheorem{cor}[subsection]{Corollary}
			 \newtheorem{prop}[subsection]{Proposition}

\numberwithin{equation}{subsection}

\newcommand{\reset}[1]{\setcounter{#1}{0}}

\newcounter{romain}[subsection]
\newcommand{\romain}
{\stepcounter{romain}
\noindent\makebox[1.5cm][r]{{\normalfont(\roman{romain})}\hspace{0.3cm}}}

\newcounter{numero}[subsection]
\newcommand{\numero}
{\stepcounter{numero}
\noindent\makebox[1.4cm][r]{\arabic{numero})\hspace{0.3cm}}}
\newcommand{\numerop}[1]
{\stepcounter{numero}
\noindent\makebox[1.4cm][r]{(\arabic{numero}#1)\hspace{0.3cm}}\hspace{-0.12cm}}

\newcounter{alphab}[subsection]
\newcommand{\alphab}
{\stepcounter{alphab}
\noindent\hspace{0.5cm}{\normalfont\alph{alphab})}\hspace{0.3cm}}

\newcommand{\ie}{i.e., }
\newcommand{\lp}{{\normalfont (}}
\newcommand{\rp}{{\normalfont )}}

\newenvironment{rem}{\noindent\textit{Remark.}}{}

\newcommand{\eq}[2]{\begin{equation}\label{#1}#2 \end{equation}}
\newcommand{\eqn}[1]{\begin{equation*}#1\end{equation*}}
\newcommand{\eqa}[2]{\begin{eqnarray}\label{#1}#2 \end{eqnarray}}
\newcommand{\eqna}[1]{\begin{eqnarray*}#1\end{eqnarray*}}
\newcommand{\ml}[2]{\begin{multline}\label{#1}#2 \end{multline}}
\newcommand{\mln}[1]{\begin{multline*}#1 \end{multline*}}
\newcommand{\ga}[2]{\begin{gather}\label{#1}#2 \end{gather}}
\newcommand{\gan}[1]{\begin{gather*}#1 \end{gather*}}
\newcommand{\surj}{\twoheadrightarrow}
\newcommand{\inj}{\hookrightarrow}
\newcommand{\lra}{\longrightarrow}
\newcommand{\xra}[1]{\xrightarrow{#1}}
\newcommand{\riso}{\xrightarrow{\ \sim\ }}

\newcommand{\liso}{\xleftarrow{\ \sim\ }}
\newcommand{\red}{{\rm red}}
\newcommand{\codim}{{\rm codim}}

\newcommand{\Sym}{{\rm Sym}}
\newcommand{\Pic}{{\rm Pic}}

\newcommand{\Tor}{{\rm Tor}}
\newcommand{\Spec}{{\rm Spec \,}}
\newcommand{\Proj}{{\rm Proj \,}}
\newcommand{\image}{{\rm Im}}

\newcommand{\Char}{{\rm char}}
\newcommand{\Tr}{{\rm Tr}}
\newcommand{\Trs}{{\rm Tr}^{\,\sharp}}
\newcommand{\Trf}{{\rm Trf}}
\newcommand{\Trp}{{\rm Trp}}

\newcommand{\Res}{{\rm Res}}
\newcommand{\Gal}{{\rm Gal}}
\newcommand{\ord}{{\rm ord}}

\newcommand{\Ker}{\mathrm{Ker}}
\renewcommand{\Im}{\mathrm{Im}}

\newcommand{\Id}{\mathrm{Id}}
\newcommand{\otimesl}{\otimes^\LL}
\newcommand{\otimesL}{\overset{\LL}{\otimes}}
\newcommand{\bd}{\bar{d}}
\newcommand{\bk}{\bar{k}}
\newcommand{\bt}{\bar{t}}
\newcommand{\del}{\partial}
\newcommand{\mbt}{\mathbf{t}}
\newcommand{\FV}{$F$-$\mspace{1mu}V\mspace{-2mu}$}
\newcommand{\restr}[1]{\raisebox{-0.8pt}{{\small$|$}}{}_{#1}}
\newcommand{\tF}{\widetilde{F}}
\newcommand{\tG}{\widetilde{G}}
\newcommand{\tu}{\tilde{u}}

\newcommand{\WC}[3]{W_{#1}\Omega^{#2}_{#3}}

\newcommand{\dlog}{\mathrm{dlog\,}}

\newcommand{\Fil}{\mathrm{Fil}}
\newcommand{\gr}{\mathrm{gr}}

\newcommand{\sHom}{\mathcal{H}{\mspace{-1mu}\it om}}
\newcommand{\sExt}{\mathcal{E}{\mspace{-2mu}\it xt}}
\DeclareMathOperator{\cl}{cl}
\DeclareMathOperator{\can}{can}
\newcommand{\Wedge}{\raisebox{.6mm}{$\bigwedge$}}
\newcommand{\Zar}{\mathrm{Zar}}
\newcommand{\et}{\mathrm{\acute{e}t}}
\newcommand{\crys}{\mathrm{crys}}
\newcommand{\Crys}{\mathrm{Crys}}
\newcommand{\st}{\mathrm{st}}
\newcommand{\dR}{\mathrm{dR}}
\newcommand{\dRW}{\mathrm{dRW}}
\newcommand{\sk}{\mathrm{sk}}
\newcommand{\cosk}{\mathrm{cosk}}
\newcommand{\Sets}{\underline{\mathrm{Sets}}}

\newcommand{\sA}{{\mathcal A}}
\newcommand{\sB}{{\mathcal B}}

\newcommand{\sE}{{\mathcal E}}
\newcommand{\sF}{{\mathcal F}}
\newcommand{\sG}{{\mathcal G}}
\newcommand{\sH}{{\mathcal H}}
\newcommand{\sI}{{\mathcal I}}
\newcommand{\sJ}{{\mathcal J}}

\newcommand{\sL}{{\mathcal L}}
\newcommand{\sM}{{\mathcal M}}
\newcommand{\sN}{{\mathcal N}}
\newcommand{\sO}{{\mathcal O}}
\newcommand{\sP}{{\mathcal P}}
\newcommand{\sQ}{{\mathcal Q}}

\newcommand{\sT}{{\mathcal T}}

\newcommand{\A}{{\mathbb A}}

\newcommand{\C}{{\mathbb C}}

\newcommand{\FF}{{\mathbb F}}

\newcommand{\LL}{\mathbb{L}}

\newcommand{\N}{{\mathbb N}}
\renewcommand{\P}{{\mathbb P}}
\newcommand{\Q}{{\mathbb Q}}
\newcommand{\RR}{\mathbb{R}}

\newcommand{\Z}{\mathbb{Z}}

\newcommand{\fU}{\mathfrak{U}}
\newcommand{\fV}{\mathfrak{V}}

\newcommand{\fa}{\mathfrak{a}}

\newcommand{\fm}{\mathfrak{m}}

\newcommand{\oline}[1]{\mspace{1mu}\overline{\mspace{-1mu}#1}}
\newcommand{\Oline}[1]{\mspace{2mu}\overline{\mspace{-2mu}#1}}
\newcommand{\bg}{\oline{g}}
\newcommand{\ok}{\oline{k}}
\newcommand{\oF}{\Oline{F}}
\newcommand{\oK}{\Oline{K}}
\newcommand{\Fpb}{\overline{\mathbb{F}}_p}
\newcommand{\Uline}[1]{\underline{#1\mspace{-2mu}}\mspace{2mu}}
\newcommand{\uP}{\Uline{P}}
\newcommand{\uS}{\Uline{S}}
\newcommand{\uT}{\Uline{T}}
\newcommand{\uX}{\Uline{X}}
\newcommand{\uY}{\Uline{Y}}
\newcommand{\uZ}{\Uline{Z}}
\newcommand{\usT}{\Uline{\mathcal{T}}}

\newcommand{\usZ}{\Uline{\mathcal{Z}}}

\newcommand{\ux}{\mspace{-1mu}\underline{\mspace{1mu}x\mspace{-1mu}}\mspace{1mu}}
\newcommand{\ualpha}{\underline{\alpha}{}}
\newcommand{\uGamma}{\underline{\Gamma}{}}
\newcommand{\uSigma}{\underline{\Sigma}{}}
\newcommand{\sbul}{\raisebox{.1mm}{\tiny$\bullet$}}
\newcommand{\hbul}{^{\sbul}}

\newcommand{\lbul}{_{\sbul}}
\newcommand{\lbbul}{_{\sbul,\sbul}}
\newcommand{\Frac}{\mathrm{Frac}}
\newcommand{\trace}{\mathrm{trace}}
\newcommand{\Qcoh}{\mathrm{Qcoh}}
\newcommand{\Db}{D^\mathrm{b}}

\newcommand{\Dpqc}{D^+_\mathrm{qc}}
\newcommand{\Dbqc}{D^\mathrm{b}_\mathrm{qc}}
\newcommand{\Dcoh}{D_\mathrm{coh}}
\newcommand{\Dpcoh}{D^+_\mathrm{coh}}
\newcommand{\Dbcoh}{D^\mathrm{b}_\mathrm{coh}}
\newcommand{\Dtdf}{D^\mathrm{b}_\mathrm{fTd}}
\newcommand{\Dqctdf}{D^\mathrm{b}_\mathrm{qc,fTd}}
\newcommand{\vC}{\check{\mathrm{C}}}

\newcommand{\half}{\frac{1}{2}}

\newcommand{\entry}[1]{\raisebox{0mm}[3mm][1mm]{$#1$}}

\begin{document}

\title[Rational points of regular models]{Rational points over finite fields for
regular models of algebraic varieties of Hodge type $\geq 1$}

\author{Pierre Berthelot}
\address{IRMAR, Universit\'e de Rennes 1,
Campus de Beaulieu,
35042 Rennes cedex, France}
\email{pierre.berthelot@univ-rennes1.fr}

\author{H\'el\`ene Esnault}
\address{Mathematik,
Universit\"at Duisburg-Essen, FB6, Mathematik, 45117 Essen, Germany}
\email{esnault@uni-due.de}

\author{Kay R\"ulling}
\address{Mathematik,
Universit\"at Duisburg-Essen, FB6, Mathematik, 45117 Essen, Germany}
\email{kay.ruelling@uni-due.de}



\thanks{Partially supported by the DFG Leibnitz Preis, the SFB/TR45, the ERC Advanced 
Grant 226257.}

\date{September 9, 2011}

\begin{abstract}
Let $R$ be a discrete valuation ring of mixed characteristics $(0,p)$, with finite
residue field $k$ and fraction field $K$, let $k'$ be a finite extension of $k$, and
let $X$ be a regular, proper and flat $R$-scheme, with generic fibre $X_K$ and special
fibre $X_k$. Assume that $X_K$ is geometrically connected and of Hodge type $\geq 1$ in
positive degrees. Then we show that the number of $k'$-rational points of $X$ satisfies
the congruence $|X(k')| \equiv 1$ mod $|k'|$. Thanks to \cite{BBE07}, we deduce such
congruences from a vanishing theorem for the Witt cohomology groups $H^q(X_k,
W\sO_{X_k,\Q})$, for $q > 0$. In our proof of this last result, a key step is the
construction of a trace morphism between the Witt cohomologies of the special fibres of
two flat regular $R$-schemes $X$ and $Y$ of the same dimension, defined by a surjective
projective morphism $f : Y \to X$.
\end{abstract}

\maketitle
\thispagestyle{empty}

\setcounter{tocdepth}{1}\tableofcontents

\section{Introduction and first reductions}\label{Intro}

Let $R$ be a discrete valuation ring of mixed characteristics $(0,p)$, with perfect
residue field $k$, and fraction field $K$. The main goal of this article is to prove
the following theorem.

\begin{thm}\label{Main}
Let $X$ be a proper and flat $R$-scheme, with generic fibre $X_K$, such that the
following conditions hold:

\alphab $X$ is a regular scheme.

\alphab $X_K$ is geometrically connected.

\alphab $H^q(X_K, \sO_{X_K}) = 0$ for all $q \geq 1$.

If $k$ is finite, then, for any finite extension $k'$ of $k$, the number of
$k'$-rational points of $X$ satisfies the congruence
\eq{maincong}{ |X(k')| \equiv 1 \mod |k'|.}
\end{thm}

Condition c) should be viewed as a Hodge theoretic property of $X_K$, which can be
stated by saying that $X_K$ has Hodge type $\geq 1$ in positive degrees. From this
point of view, this theorem fits in the general analogy between the vanishing of Hodge
numbers for varieties over a field of characteristic $0$, and congruences on the number
of rational points with values in finite extensions for varieties over a finite field.
This analogy came to light with the coincidence between the numerical values in
Deligne's theorem on smooth complete intersections in a projective space \cite[Expos\'e
XI, Th.~2.5]{SGA 7 II}, and in the Ax-Katz theorem on congruences on the number of
solutions of systems of algebraic equations \cite[Th.~1.0]{Kz71}. It has been made
effective by Katz's conjecture \cite[Conj.~2.9]{Kz71} relating the Newton and Hodge
polygons associated to the cohomology of a proper and smooth variety (and generalizing
earlier results of Dwork for hypersurfaces \cite{Dw64}). For varieties in
characteristic $p$, this conjecture was proved by Mazur (\cite{Ma72}, \cite{Ma73}) and
Ogus \cite[Th.~8.39]{BO78}. In the mixed characteristic case, where a stronger form can
be given using the Hodge polygon of the generic fibre, it is a consequence of the
fundamental results in $p$-adic Hodge theory. Our proof of Theorem \ref{Main} makes
essential use of the unequality between these two polygons, but the setup of the theorem
is actually more general, since the scheme $X$ is not supposed to be semi-stable over
$R$.

Let us also recall that a result similar to Theorem \ref{Main} has been proved by the
second author \cite[Th.~1.1]{Es06} by $\ell$-adic methods, with condition c) replaced
by a coniveau condition: for any $q \geq 1$, any cohomology class in
$H^q_{\et}(X_{\oK},\Q_{\ell})$ vanishes in $H^q_{\et}(U_{\oK},\Q_{\ell})$ for some non
empty open subset $U \subset X_K$. It is easy to see, using \cite{De71}, that this
coniveau condition implies that the Hodge level of $X_K$ is $\geq 1$ in degree $q \geq 
1$ (see \cite[4.4 (d)]{Il06} for a more general discussion). It
would actually follow from Grothendieck's generalized Hodge conjecture \cite{Gr69} that
the two conditions are equivalent. In this article, the use of $p$-adic methods, and in
particular of $p$-adic Hodge theory, allows us to derive congruence \eqref{maincong}
directly from Hodge theoretic hypotheses.

\subsection{}\label{Wittreduc}
As explained by Ax \cite{Ax64}, congruences such as \eqref{maincong} can be expressed
in terms of the zeta function of the special fibre $X_k$ of $X$. We recall that the 
rationality of the zeta function $Z(X_k,t)$ allows to define the slope $< 1$ part 
$Z^{<1}(X_k,t)$ of $Z(X_k,t)$ as follows \cite[6.1]{BBE07}. Let $|k| = p^a$, and write  
\eqn{ Z(X_k,t) = \prod_i(1-\alpha_i t)/\prod_j(1-\beta_j t), }
with $\alpha_i, \beta_j \in \overline{\Q}_p$ and $\alpha_i \neq \beta_j$ for all $i, 
j$. Normalizing the $p$-adic valuation $v$ of $\overline{\Q}_p$ by $v(p^a)=1$, one sets 
\eqn{ Z^{<1}(X_k,t) = \prod_{v(\alpha_i)<1}(1-\alpha_i t)/\prod_{v(\beta_j)<1}(1-\beta_j 
t).  }
Then the congruences \eqref{maincong} are equivalent to 
\eq{partialzeta}{ Z^{<1}(X_k,t)  =  \frac{1}{1-t} }
\cite[Prop.~6.3]{BBE07}.

On the other hand, let $W(\sO_{X_k})$ be the sheaf of Witt vectors with coefficients in
$\sO_{X_k}$, and $W\sO_{X_k,\Q} = W(\sO_{X_k})\otimes \Q$. Then the identification of
the slope $< 1$ part of rigid cohomology with Witt vector cohomology provides the
cohomological interpretation
\eq{wittzeta}{ Z^{<1}(X_k,t)  =  \prod_i \det(1-tF^a|H^i(X_k,W\sO_{X_k,\Q}))^{(-1)^{i+1}}, }
where $F$ is induced by the Frobenius endomorphism of $W(\sO_{X_k})$
\cite[Cor.~1.3]{BBE07}. Therefore, Theorem \ref{Main} is a consequence of the
following theorem, where $k$ is only assumed to be perfect:

\begin{thm}\label{Wittvanish}
Let $X$ be a regular, proper and flat $R$-scheme. Assume that \hfill\linebreak 
$H^q(X_K, \sO_{X_K}) = 0$ for some $q \geq 1$. Then:
\eq{wittvanish}{ H^q(X_k, W\sO_{X_k,\Q}) = 0.}
\end{thm}

\noindent\textit{Proof of Theorem \ref{Main}, assuming Theorem \ref{Wittvanish}}. Let
us prove here this implication, which is easy and does not use the regularity
assumption on $X$. Let $W = W(k)$, and $K_0 = \Frac(W)$. Thanks to \eqref{partialzeta}
and \eqref{wittzeta}, Theorem \ref{Wittvanish} implies that it suffices to prove that
the homomorphism $K_0 \to H^0(X_k, W\sO_{X_k,\Q})$ is an isomorphism.

Since $X$ is proper and flat over $R$, $H^0(X,\sO_X)$ is a free finitely generated
$R$-module. As the generic fibre $X_K$ is geometrically connected and geometrically
reduced, the rank of $H^0(X,\sO_X)$ is $1$. The homomorphism $R \to H^0(X,\sO_X)$
maps $1$ to $1$, hence Nakayama's lemma implies that it is an isomorphism. Applying
Zariski's connectedness theorem, it follows that $X_k$ is connected, and even
geometrically connected, since the same argument can be applied after any base change
from $R$ to $R'$, where $R'$ is the ring of integers of a finite extension of $K$.

On the other hand, let $\bk$ be an algebraic closure of $k$, and let $k'$ be a finite
extension of $k$ such that $X_{\bk\,\red}$ is defined over $k'$. As $k'$ is separable
over $k$, the homomorphisms $W_n(k)\to W_n(k')$ are finite \'etale liftings of $k \to
k'$, and the homomorphisms $W_n(k')\otimes_{W_n(k)} W_n(\sO_{X_k}) \to
W_n(\sO_{X_{k'}})$ are isomorphisms \cite[I, Prop.~1.5.8]{Il79}. It follows that the
homomorphism $W(k')\otimes_{W(k)} H^0(X_k, W(\sO_{X_k})) \to H^0(X_{k'},
W(\sO_{X_{k'}}))$ is an isomorphism, and that it suffices to prove the claim for
$X_{k'}$. Using the fact that
\eqn{ H^0(X_{k'}, W\sO_{X_{k'},\Q}) \riso H^0(X_{k'\,\red}, W\sO_{X_{k'\,\red},\Q}) }
by \cite[Prop.~2.1 (i)]{BBE07}, it suffices to check that, if $Z$ is a proper,
geometrically connected and geometrically reduced $k$-scheme, the homomorphism $W(k)
\to H^0(Z,W(\sO_{Z}))$ is an isomorphism.

Under these assumptions, the homomorphism $k \to H^0(Z, \sO_Z)$ is an isomorphism. As
the homomorphism $R : W_n(\sO_Z) \to W_{n-1}(\sO_Z)$ is the projection of a product
onto one of its factors, the homomorphisms $H^0(Z, W_n(\sO_Z)) \to H^0(Z,
W_{n-1}(\sO_Z))$ are surjective, and one gets by induction that the homomorphism
$W_n(k) \to H^0(Z, W_n(\sO_Z))$ is an isomorphism for all $n$. Taking inverse limits,
the claim follows.
\hfill$\Box$

\subsection{}\label{Trivial}
Theorem \ref{Wittvanish} is deeper, and most of our paper is devoted to developing the
techniques used in its proof. We may observe though that, in the context of Theorem
\ref{Main}, there is a case where \eqref{wittvanish} is trivial: namely, if we replace
the condition on the Hodge numbers of $X_K$, which is equivalent to requiring that the
modules $H^q(X,\sO_X)$ be $p$-torsion modules, by the stronger condition that
$H^q(X,\sO_X)$ vanishes for all $q \geq 1$. Indeed, the flatness of $X$ over $R$ allows
to apply the derived base change formula for coherent cohomology and to conclude that
$H^q(X_k,\sO_{X_k}) = 0$ for all $q \geq 1$. By induction on $n$, one gets that
$H^q(X_k,W_n(\sO_{X_k})) = 0$ for all $n, q \geq 1$, and \eqref{wittvanish} follows for
all $q \geq 1$ (even before tensoring with $\Q$).

In the general case, where the $H^q(X,\sO_X)$ are $p$-torsion modules, we do not know
any direct argument to derive the vanishing property stated in \eqref{wittvanish}. Our
strategy is then to use the results of $p$-adic Hodge theory relating the Hodge and
Newton polygons of certain filtered $F$-isocrystals on $k$, which allow to study
separately the cohomology groups for a given $q$ as in Theorem \ref{Wittvanish}. In
particular, when $X$ is semi-stable on $R$, a straightforward argument using the
fundamental comparison theorems of $p$-adic Hodge theory allows to deduce
\eqref{wittvanish} from the unequality between the two polygons defined by the log
crystalline cohomology of $X_k$. We explain this argument in Theorem \ref{Semistable}.

In the rest of Section \ref{pHodge}, we show that this argument can be modified to
prove the vanishing of $H^q(X_k, W\sO_{X_k,\Q})$ in the general case. For any finite
extension $K'$ of $K$, with ring of integers $R'$, let $X_{R'}$ be deduced from $X$ by
base change from $R$ to $R'$. After reducing to the case where $R$ is complete, 
the first step is to apply de Jong's alteration theorem
to construct for any $m$ an $m$-truncated simplicial scheme $Y\lbul$ over the ring of
integers $R'$ of a suitable extension $K'$ of $K$, endowed with an augmentation
morphism $Y_0 \to X_{R'}$, such that the $Y_i$'s are pullbacks of proper semi-stable
schemes, and $Y\lbul \to X_{R'}$ induces an $m$-truncated proper hypercovering of
$X_{K'}$ (see Lemma \ref{mtruncres} for a precise statement). Then, using Tsuji's
extension of the comparison theorems to truncated simplicial schemes \cite{Ts98}, we
show that, in this situation, the cohomology group $H^q(Y\lbul{}_k,
W\sO_{Y\lbul{}_k,\Q})$ vanishes. However, due to the possible presence of vertical 
components in the coskeletons, the special fibre $Y\lbul{}_k$ of the
$m$-truncated simplicial scheme $Y\lbul$ may not be a proper hypercovering of $X_k$,
and it is unclear how the groups $H^q(Y\lbul{}_k, W\sO_{Y\lbul{}_k,\Q})$ are related to
the groups $H^q(X_k, W\sO_{X_k,\Q})$. Therefore another ingredient will be necessary to
complete the proof. It will be provided by the following injectivity theorem, the proof
of which will be given in section \ref{Injectivity2}.

\begin{thm}\label{Tworeg}
Let $X$, $Y$ be two flat, regular $R$-schemes of finite type, of the same dimension,
and let $f : Y \to X$ be a projective and surjective $R$-morphism, with reduction 
$f_k$ over $\Spec k$. Then, for all $q \geq 0$, the functoriality homomorphism 
\eq{tworeg}{ f_k^* : H^q(X_k, W\sO_{X_k,\Q}) \lra H^q(Y_k, W\sO_{Y_k,\Q}) }
is injective.
\end{thm}

\subsection{}\label{Traceintro}
We will deduce Theorem \ref{Tworeg} from the existence of a trace morphism 
\eq{tauipi}{ \tau_{i,\pi} : \RR f_*(W\sO_{Y_k,\Q}) \lra W\sO_{X_k,\Q}, }
defined by means of a factorization $f = \pi \circ i$, where $\pi$ is the projection of
a projective space $\P^{\,d}_X$ on $X$, and $i$ is a closed immersion. The key fact
used in the construction of this trace morphism is that, under the assumptions of
Theorem \ref{Tworeg}, $i$ is a regular immersion of codimension $d$, or, said
otherwise, that $f$ is a complete intersection morphism of virtual relative dimension
$0$, in the sense of \cite[Expos\'e VIII]{SGA 6}. 

Sections \ref{Injectivity1} to \ref{HWclass} are devoted to the construction of 
$\tau_{i,\pi}$. In section \ref{Injectivity1}, we state a similar result for $\sO_X$, 
providing a canonical trace morphism 
\eqn{ \tau_f : \RR f_*(\sO_Y) \to \sO_X, }
whenever $X$ is a noetherian scheme with a relative dualizing complex, and $f : Y \to
X$ is a proper complete intersection morphism of virtual relative dimension $0$ (see
Theorem \ref{Thtau}). The existence of $\tau_f$ has been observed by El Zein as a
particular case of his construction of the relative fundamental class \cite[IV,
Prop.~6]{El78}. However, there does not seem to be in the literature a complete proof
of the properties listed in Theorem \ref{Thtau}. Due to the many corrections and
complements to \cite{Ha66} made by Conrad in \cite{Co00}, we have included in an
Appendix the details of a proof of Theorem \ref{Thtau} based on \cite{Co00}. So we
refer to \ref{Deftau} for the definition of $\tau_f$, and to \ref{Prooftau} for the
proof of Theorem \ref{Thtau}. When $Y$ is finite locally free of rank $r$ over $X$, the
composition of the functoriality morphism $\sO_X \to \RR f_*(\sO_Y)$ with $\tau_f$ is
multiplication by $r$ on $\sO_X$. This has striking consequences for the functoriality
maps induced by $f$ on coherent cohomology (see Theorem \ref{Injth}). For example, if
$r$ is invertible on $X$, one obtains an injectivity theorem which may be of
independent interest. An outline of the construction of $\tau_f$ is given in the
introduction to the Appendix.

To construct the trace morphism $\tau_{i,\pi}$, we consider more generally a projective
complete intersection morphism $f : Y \to X$ of virtual relative dimension $0$ between
two noetherian $\FF_p$-schemes with dualizing complexes. Under these assumptions, we
construct a compatible family of morphisms
\eqn{ \tau_{i,\pi,n} : \RR f_*(W_n(\sO_Y)) \to W_n(\sO_X) }
for $n \geq 1$, with $\tau_{i,\pi,1} = \tau_f$. Our main tool here is the theory of the
relative de Rham-Witt complex developped by Langer and Zink \cite{LZ04}. In Section
\ref{DRWprel}, we recall some basic facts about their construction, and we extend to
the relative case some structure theorems proved by Illusie \cite{Il79} when the base
scheme is perfect (see in particular Proposition \ref{KerR} and Theorem
\ref{Structgrn}). Then we define $\tau_{i,\pi,n}$ by combining two morphisms. On the
one hand, we consider a projective space $P := \P^{\,d}_X$ with projection $\pi$ on
$X$, and we define in Section \ref{HWtrp} a trace morphism
\eqn{ \Trp_{\pi,n} : \RR\pi_*(\WC{n}{d}{P/X}[d]) \to W_n(\sO_X), }
using the $d$-th power of the Chern class of the canonical bundle $\sO_P(1)$. On the
other hand, we consider a regularly embedded closed subscheme $Y$ of a smooth
$X$-scheme $P$, and we define in Section \ref{HWclass} a relative Hodge-Witt
fundamental class for $Y$ in $P$, which is a section of $\sH^d_Y(\WC{n}{d}{P/X})$ and
defines a morphism
\eqn{ \gamma_{i,\pi,n} : i_*W_n(\sO_Y) \to \WC{n}{d}{P/X}[d], }
with $i : Y \inj P$ and $d = \codim_P(Y)$. This allows to define the morphism
$\tau_{i,\pi,n}$ as being the composition $\Trp_{\pi,n}\circ
\RR\pi_*(\gamma_{i,\pi,n})$. The proof of Theorem \ref{Tworeg} is then completed in
Section \ref{Injectivity2} thanks to a theorem relating the morphisms $\tau_{i,\pi,n}$
defined by the reduction mod $p$ of a factorization of the given morphism $f : Y \to X$
over $R$, and the morphism $\tau_f$ defined by $f$.

It may be worth pointing out here that these results seem to indicate that
Grothendieck's relative duality theory for coherent $\sO$-modules can be generalized to
some extent to the Hodge-Witt sheaves, as was already apparent from \cite{Ek84} when
the base scheme is a perfect field. We do not try to develop such a generalization in
this article, and we limit ourselves to the properties needed for the proof of
Theorem \ref{Main}. For example, it is very likely that the morphisms $\tau_{i,\pi,n}$
only depend on $f$, and not on the chosen factorization $f = \pi\circ i$, but this is
not needed here, and we do not prove it in this article. A natural context one might
think of for developing our results is the theory of the trace map for projectively
embeddable morphisms outlined in \cite[III, 10.5 and \S 11]{Ha66}. However, as
discussed by Conrad in \cite[p.~103-104]{Co00}, the foundational work needed for the
definition of such a theory has not really been done even for coherent $\sO$-modules.
We hope to return to these questions in another article.

Finally, we conclude in Section \ref{TheExample} by giving a family of examples to
which Theorem \ref{Main} can be applied, but which are not covered by earlier results,
nor by cases where Theorem \ref{Wittvanish} can be proved directly, such as the trivial
case where $H^i(X, \sO_X) = 0 $ for all $i \geq 1$, or the semi-stable case. These
examples are obtained for $p \geq 7$, and are quotients of an hypersurface of degree $p$
in a projective space $\P^{\,p-2}_R$ by a free $(\Z/p\Z)$-action. Their generic fibre
is a smooth variety of general type, and their special fibre has isolated
singularities, at least when $p$ is not a Fermat number.

\medskip
\noindent\textbf{Acknowledgements}

The authors thank the referee for his careful reading of the manuscript, and for his 
very useful comments, questions and suggestions. 

Part of this work was done in March 2008, when the second and third authors enjoyed the
hospitality of the Department of Mathematics at Rennes.

\medskip
\noindent\textbf{General conventions}  

\numero All schemes under consideration are supposed to be separated. By a projective
morphism $f : Y \to X$, we always mean a morphism which can be factorized as $f = \pi
\circ i$, where $i$ is a closed immersion in some projective space $\P^n_X$, and $\pi$
is the natural projection $\P^n_X \to X$.

\numero In this paper, we use the terminology of \cite{SGA 6} for complete intersection
morphisms: a morphism of schemes $f : Y \to X$ is said to be a \textit{complete
intersection morphism} if, for any $y \in Y$, there exists an open neighbourhood $U$ of
$y$ in $Y$ such that the restriction of $f$ to $U$ can be factorized as 
$f\restr{U} = \pi
\circ i$, where $\pi$ is a smooth morphism and $i$ a regular immersion \cite[VIII,
1.1]{SGA 6}. Note that this notion of complete intersection morphism is more general
than the notion of ``local complete intersection map'' used in \cite{Ha66} and
\cite{Co00}, where ``lci map'' is only used for regular immersions.

If $d$ is the codimension of $i$ at $y$, and $n$ the relative dimension of $\pi$ at
$i(y)$, the integer $m = n-d$ does not depend upon the local factorization $f\restr{U} =
\pi \circ i$, and is called the \textit{virtual relative dimension} of $f$ at $y$
\cite[VIII, 1.9]{SGA 6}. One says that $f$ has \textit{constant virtual relative
dimension} $m$ if the integer $m$ does not depend upon $y$. We will always assume in
this paper that the virtual relative dimension of the morphisms under consideration is
constant (however, the dimension of the fibres of such morphisms can vary).

\numero Apart from the previous remark, we will use the definitions and sign
conventions from Conrad's book \cite{Co00}. In particular, when $i : Y \inj P$ is a
regular immersion of codimension $d$ defined by an ideal $\sI \subset \sO_P$, we define
$\omega_{Y/P}$ by
\eqn{\omega_{Y/P} = \wedge^d((\sI/\sI^2)^\vee)}
rather than $(\wedge^d(\sI/\sI^2))^\vee$ as in \cite[III, p.\ 141]{Ha66} (see \cite[p.\
7]{Co00}). The canonical identification between both definitions is given by \cite[III,
\S 11, Prop.\ 7]{Bo70}.

\numero If $R$, $S$ are commutative rings, $R \to S$ a ring homomorphism, and $X$ an
$R$-scheme, we denote by $X_S$ the $S$-scheme $\Spec S \times_{\Spec R} X$.

\numero If $\sE\hbul$ is a complex, we denote by $(\sigma_{\geq i}\sE\hbul)_{i\in\Z}$
the naive filtration on $\sE\hbul$, i.e., the filtration defined by $\sigma_{\geq i}
\sE^n = 0$ if $n < i$, $\sigma_{\geq i}\sE^n = \sE^n$ if $n \geq i$.

\numero If $X$ is a scheme (resp.~locally noetherian scheme), we denote by
$\Dbqc(\sO_X)$ (resp.~$\Dbcoh(\sO_X)$) the full subcategory of the derived category
$D(\sO_X)$ which has as objects the bounded complexes with $\sO_X$-quasi-coherent
(resp.~$\sO_X$-coherent) cohomology sheaves. We denote by $\Dtdf(\sO_X) 
\subset D(\sO_X)$ the full subcategory of complexes which are isomorphic to a bounded
complex of flat $\sO_X$-modules. Adding several of the indices to $\Db(\sO_X)$, as in
$\Dqctdf(\sO_X)$, means taking the intersection of the corresponding subcategories.

When relevant, we will use similar notations for the analogous subcategories of
$D(\sO_X)$ and $D^+(\sO_X)$.

\section{Application of $p$-adic Hodge theory}\label{pHodge}

We explain in this section how the fundamental results of $p$-adic Hodge theory can be 
used to prove Theorem \ref{Wittvanish}. We begin with the semi-stable case, where 
$p$-adic Hodge theory suffices to conclude, and which will serve as a model for the 
general case. We use the notations $R$, $K$, $k$ as in the introduction.

\begin{thm}\label{Semistable}
Let $X$ be a proper and semi-stable $R$-scheme, with generic fibre $X_K$ and special 
fibre $X_k$, and let $q \geq 0$ be an integer. If $H^q(X_K, \sO_{X_K}) = 0$, then 
$H^q(X_k,W\sO_{X_k,\Q}) = 0$. 
\end{thm}

\begin{proof}
We may assume that $R$ is a complete discrete valuation ring. Indeed, if $\widehat{R}$ is the
completion of $R$, $\widehat{K} = \Frac(\widehat{R})$ and $\widetilde{X} = X_{\widehat{R}}$, then 
$\widetilde{X}$ is proper and
semi-stable over $\widehat{R}$, $H^q(\widetilde{X}_{\widehat{K}}, \sO_{\widetilde{X}_{\widehat{K}}}) 
= \widehat{K}\otimes_K H^q(X_K, \sO_{X_K}) =
0$, and $X$ and $\widetilde{X}$ have isomorphic special fibres. So the theorem for 
$\widetilde{X}$ implies the
theorem for $X$. 

We endow $S = \Spec R$ with the log structure defined by the divisor $\Spec k \subset
S$, $S_0 = \Spec k$ with the induced log structure, and we denote by $\uS$, $\uS_0$ the
corresponding log schemes. Similarly, we endow $X$ with the log structure defined by
the special fibre $X_k$, $X_k$ with the induced log structure, and we denote by $\uX$,
$\uX_k$ the corresponding log schemes. Then $\uX$ is smooth over $\uS$, and $\uX_k$ is
smooth of Cartier type \cite[(4.8)]{Ka89} over $\uS_0$.

Let $W_n = W_n(k)$ (resp.~$W = W(k)$), and let $\uSigma_n$ (resp.~$\uSigma$) be the log
scheme obtained by endowing $\Sigma_n = \Spec W_n$ (resp.~$\Sigma = \Spec W$) with the
log structure associated to the pre-log structure defined by the morphism $M_{S_0} \to
\sO_{S_0}=\sO_{\Sigma_1} \to \sO_{\Sigma_n}$ (resp.~$\sO_{\Sigma}$) provided by
composition with the Teichm\"uller representative map. We can then consider the log
crystalline cohomology groups $H^q_{\crys}(\uX/\uSigma_n)$, which are finitely
generated $W_n$-modules endowed with a Frobenius action $\varphi$ and a monodromy
operator $N$. The log scheme $\uX_k$ also carries a logarithmic de Rham-Witt complex
$\WC{}{\sbul}{\uX_k} = \varprojlim_n \WC{n}{\sbul}{\uX_k}$, constructed by Hyodo
\cite{Hy91} in the semi-stable case, and generalized by Hyodo and Kato
\cite[(4.1)]{HK94} to the case of smooth $\uS_0$-log schemes of Cartier type. In degree
$0$, we have
\eq{dRW0}{ \WC{n}{0}{\uX_k}  =  W_n(\sO_{X_k}), }
by \cite[Prop.~(4.6)]{HK94}. 

It follows from \cite[Th.~(4.19)]{HK94} that, for all $q$, there are canonical
isomorphisms
\eq{logdRWn}{ H^q_{\crys}(\uX_k/\uSigma_n) \riso H^q(\uX_k, \WC{n}{\sbul}{\uX_k}), }
which are compatible when $n$ varies, and commute with the Frobenius actions. As $X_k$
is proper over $S_0$, these cohomology groups are artinian $W$-modules. Therefore one
can apply the Mittag-Leffler criterium to get canonical isomorphisms 
\eq{logdRW}{ H^q_{\crys}(\uX_k/\uSigma) \riso \varprojlim_n H^q_{\crys}(\uX_k/\uSigma_n)
\riso \varprojlim_n H^q(\uX_k, \WC{n}{\sbul}{\uX_k}) \liso H^q(\uX_k, \WC{}{\sbul}{\uX_k}) }
compatible with the Frobenius actions. Using the naive filtration of
$\WC{}{\sbul}{\uX_k}$ and tensoring by $K_0$, one obtains a spectral sequence
\eq{slopeseq}{ E^{i,j}_1 = H^j(\uX_k, \WC{}{i}{\uX_k})\otimes K_0 \ \Longrightarrow\ 
H^{i+j}_{\crys}(\uX_k/\uSigma)\otimes K_0  }
endowed by functoriality with a Frobenius action $F^*$. The operators $d$, $F$ and $V$
on the logarithmic de Rham-Witt complex satisfy the same relations than on the usual de
Rham-Witt complex \cite[(4.1)]{HK94}, and the structure theorems of \cite{Il79} remain
valid in the logarithmic case \cite[Th.~(4.4) and Cor.~(4.5)]{HK94}. It follows that, for all $i$,
$j$, the $K_0$-vector space $H^j(\uX_k, \WC{}{i}{\uX_k})\otimes K_0$ is finite
dimensional, and the action of $F^*$ on this space has slopes in $[i,i+1[$ 
\cite[3.1]{Lo02}.
Therefore, the spectral sequence \eqref{slopeseq} degenerates at $E_1$, and yields in
particular an isomorphism $(H^q_{\crys}(\uX_k/\uSigma)\otimes K_0)^{<1} \xra{\sim}
H^q(\uX_k, \WC{}{0}{\uX_k})\otimes K_0$, the source being the part of
$H^q_{\crys}(\uX_k/\uSigma)\otimes K_0$ where Frobenius acts with slope $< 1$. Thanks 
to \eqref{dRW0}, we finally get a canonical isomorphism
\eq{slopepart}{ (H^q_{\crys}(\uX_k/\uSigma)\otimes K_0)^{<1} \riso H^q(X_k,
W\sO_{X_k,\Q}). }

On the other hand, the choice of an uniformizer of $R$ determines a Hyodo-Kato 
isomorphism \cite[Th.~(5.1)]{HK94}
\eq{HKiso}{ \rho : H^q_{\crys}(\uX_k/\uSigma)\otimes_W K \riso H^q(X_K, 
\Omega\hbul_{X_K/K}). }
This allows to endow $H^q_{\crys}(\uX_k/\uSigma)\otimes_W K$ with the filtration
deduced via $\rho$ from the Hodge filtration of $H^q(X_K, \Omega\hbul_{X_K/K})$.
Together with its Frobenius action and monodromy operator,
$H^q_{\crys}(\uX_k/\uSigma)\otimes_W K$ is then a filtered $(\varphi,N)$-module as
defined by Fontaine \cite[4.3.2 and 4.4.8]{Fo94}. As such, it has both a Newton
polygon, built as usual from the slopes of the Frobenius action, and a Hodge polygon,
built as usual from the Hodge numbers of $H^q(X_K, \Omega\hbul_{X_K/K})$. Now, let
$\oK$ be an algebraic closure of $K$, and let $B_{\st}$, $B_{\dR}$ be the 
Fontaine $p$-adic period rings. Then Tsuji's comparison theorem \cite[Th.~0.2]{Ts99}
provides a $B_{\st}$-linear isomorphism
\eq{tsuji1}{ B_{\st} \otimes_{K_0} H^q_{\crys}(\uX_k/\uSigma) \riso B_{\st} \otimes_K 
H^q_{\et}(X_{\oK}, \Q_p), }
compatible with the natural Galois, Frobenius and monodromy actions on both sides, and
with the natural Hodge filtrations defined on both sides after scalar extension from
$B_{\st}$ to $B_{\dR}$. Thus $H^q_{\crys}(\uX_k/\uSigma)\otimes K_0$ is an admissible
filtered $(\varphi,N)$-module \cite[5.3.3]{Fo94}, and therefore is weakly admissible
\cite[5.4.2]{Fo94}. This implies that its Newton polygon lies above its Hodge polygon
\cite[4.4.6]{Fo94}. In particular, either $H^q_{\crys}(\uX_k/\uSigma)\otimes K_0 = 0$,
or the smallest slope of its Newton polygon is bigger than the smallest 
slope of its Hodge polygon. By assumption, the latter is at least 1, which forces the
part of slope $< 1$ of $H^q_{\crys}(\uX_k/\uSigma)\otimes K_0$ to vanish. Thanks to
\eqref{slopepart}, this implies the theorem.
\end{proof}

In the general case, we will use truncated simplicial log schemes satisfying the
conditions of the next lemma. We will assume that all the log schemes under
consideration are fine log schemes \cite[(2.3)]{Ka89}, and all constructions involving
log schemes will be done in the category of fine log schemes. For any finite extension
$K'$ of $K$, with ring of integers $R'$, we will endow $\Spec R'$ with the log
structure defined by its closed point, and pullbacks of log schemes to $\Spec R'$ will
mean pullbacks in the category of log schemes. Note that, because of \cite[(4.4) (ii) 
and (4.3.1)]{Ka89}, the underlying scheme of such a pullback is the usual pullback in 
the category of schemes. We will denote log schemes by underlined letters, and drop 
the underlining to denote the underlying schemes.

\begin{lem}\label{mtruncres}
Assume that $R$ is complete, and that $X$ is an integral, flat $R$-scheme of finite
type. Let $m \geq 0$ be an integer. Then there exists a finite extension $K'$ of $K$,
with ring of integers $R'$, a split $m$-truncated simplicial $R'$-log scheme $\uY\lbul
= (Y\lbul,M_{Y\lbul})$ \cite[Vbis, 5.1.1]{SGA 4 II}, and an augmentation morphism $u :
Y_0 \to X_{R'}$ over $R'$, such that the following conditions hold:

\alphab For each $r$, $Y_r$ is projective over $X_{R'}$, and $\uY_r$ is a disjoint
union of pullbacks to $R'$ of semi-stable schemes over the integers of
sub-$K$-extensions of $K'$ endowed with the log structure defined by their special
fibre;

\alphab Via the augmentation morphism induced by $u$, $Y\lbul{}_{,K'}$ is an
$m$-truncated proper hypercovering of $X_{K'}$;

\alphab There exists a projective $R$-alteration $f : Y \to X$, where $Y$ is
semi-stable over the ring of integers $R_1$ of a sub-$K$-extension $K_1 $ of
$K'$, and there exists finitely many $R$-embeddings $\sigma_i : R_1 \inj R'$, such
that, if $u_1 : Y \to X_{R_1}$ denotes the $R_1$-morphism defined by $f$, and if 
$Y_{\sigma_i}$
\lp resp.\ $u_{\sigma_i} : Y_{\sigma_i} \to X_{R'}$\rp\ denotes the $R'$-scheme \lp resp.\
$R'$-morphism\rp\ deduced by base change via $\sigma_i$ from $Y$ \lp resp.\ $u_1$\rp,
then $Y_0 = \coprod_i Y_{\sigma_i}$ and $u\restr{Y_{\sigma_i}}=u_{\sigma_i}$.

\end{lem}

Therefore, we obtain the following commutative diagram:

\eq{}{ \xymatrix{
Y_0 = \coprod_i Y_{\sigma_i} \ar[rd]^{u} & \,Y_{\sigma_i} \ar@{_{(}->}[l] \ar[r] 
\ar[d]^{u_{\sigma_i}} & Y \ar[d]^{u_1} \ar[dr]^{f} & \\
& X_{R'} \ar[r] & X_{R_1} \ar[r]  &  X.
} }

\begin{proof}
This is a well known consequence of de Jong's alteration theorem \cite[Th.~6.5]{dJ96}.
For the sake of completeness, we briefly recall how to construct such a simplicial log
scheme. For $r \geq 0$, we denote by $[r]$ the ordered set $\{0,\ldots, r\}$, and by
$\Delta$ (resp.~$\Delta[m]$) the category which has the sets $[r]$ (resp.~with $r \leq
m$) as objects, the set of morphisms from $[r]$ to $[s]$ being the set of
non-decreasing maps $[r] \to [s]$.

One proceeds by induction on $m$. Assume first that $m = 0$. De Jong's theorem provides
a finite extension $K_1$ of $K$, an integral semi-stable scheme $Y$ over the ring of
integers $R_1$ of $K_1$, and an $R$-morphism $f : Y \to X$ which is a projective
alteration. Let $u_1 : Y \to X_{R_1}$ be the morphism defined by $f$. Let $K'$ be a
finite extension of $K_1$ such that $K'/K$ is Galois, and let $R'$ be its ring of
integers. For any $g \in \Gal(K'/K)$, let $\sigma_g$ be the composition $K_1 \to K'
\xra{g} K'$, and let $\uY_g$ (resp.~$u_g : Y_g \to X_{R'}$) be the $R'$-log scheme
(resp.~$R'$-morphism) deduced from $\uY$ (resp.~$u_1$) by base change via $\sigma_g :
R_1 \to R'$. Then one defines $\uY_0$ and $u$ by setting
\eqn{ \uY_0 = \coprod_{g \in \Gal(K'/K)}\uY_g, \quad\quad u\restr{Y_g} = u_g. }
One easily checks by Galois descent that $Y_{0,K'} \to X_{K'}$ is surjective, and 
conditions a) - c) are then satisfied. 

Assume now that the lemma has been proved for $m-1$. Over the ring of integers $R''$ of
some finite extension $K''$ of $K$, this provides a split $(m-1)$-truncated simplicial
log scheme $\uY''\lbul$, together with an augmentation morphism $u'' : Y''_0 \to
X_{R''}$, so as to satisfy conditions a) - c). Note that these conditions remain
satisfied after a base change to the ring of integers of any finite extension of $K''$.
Let $\cosk_{m-1}(\uY''\lbul)$ be the coskeleton of $\uY''\lbul$ in the category of
simplicial fine $R''$-log schemes, and $\uZ = \cosk_{m-1}(\uY''\lbul)_m$ its
component of index $m$. Denote by $Z_1,\ldots,Z_c$ those irreducible components of $Z$
which are flat over $R''$, and endow each $Z_j$ with the log structure
induced by the log structure of $\uZ$. As a consequence of condition a), this log
structure induces the trivial log structure on the generic fibre $\uZ_{j,K''}$. Applying
de Jong's theorem to $Z_j$, one can find a finite extension $K'_j$ of $K''$, with ring
of integers $R'_j$, an integral semi-stable scheme $T_j$ over $R'_j$ and a projective
alteration $f_j : T_j \to Z_j$. One endows $T_j$ with the log structure defined by its
special fibre. Because the log structure of the generic fibre $\uZ_{j,K''}$ is trivial,
the morphism $f_j$ extends uniquely to a log morphism $f_j : \uT_j \to \uZ_j$. Let $K'$
be a Galois extension of $K''$ containing $K'_j$ for all $j$, $1 \leq j \leq c$, and
let $R'$ be its ring of integers. Arguing as in the case $m = 0$ above, one can deduce
from the alterations $f_j$ an $R'$-morphism
\eq{cosk}{ \uT \lra \coprod_{j=1}^c \uZ_{j,R'} \lra \uZ_{R'} \riso 
\cosk_{m-1}(\uY''\lbul{}_{,R'})_m }
where $\uT$ satisfies condition a), and $T_{K'} \to
\cosk_{m-1}(Y''\lbul{}_{\!\!,K'})_m$ is projective and surjective (note that, since
all log structures are trivial on the generic fibres, the generic fibre of the
coskeleton computed in the category of fine log schemes is the coskeleton of the
generic fibres computed in the category of schemes). One can then follow the method of 
Saint-Donat \cite[Vbis, 5.1.3]{SGA 4 II} and Deligne 
\cite[(6.2.5)]{De74} to extend $\uY''\lbul{}_{,R'}$ as a split
$m$-truncated simplicial log scheme $\uY\lbul$ over $R'$. The $R'$-log
scheme $\uY_m$ is defined by 
\eqn{\uY_m\ =\ \uT \mbox{\ $\coprod$\ } \coprod_{[m]\surj[l],\; l<m} 
N(\uY_l),}
where $N(\uY_l)$ is the complement of the union of the images of the degeneracy
morphisms with target $\uY_l$. It satisfies condition a) because $\uT$ does and
$\uY\lbul$ is split. Similarly, the morphism $Y_{m,K'} \to
\cosk_{m-1}(Y\lbul{}_{,K'})_m$ is proper and surjective because the morphism $T_{K'}
\to \cosk_{m-1}(Y''\lbul{}_{\!\!,K'})_m$ is proper and surjective. Thus the
$m$-truncated simplicial scheme $Y\lbul{}_{,K'}$ is an $m$-truncated proper
hypercovering of $X_{K'}$. Finally, condition c) is satisfied thanks to the induction
hypothesis.
\end{proof}

\subsection{}\label{Cohsimpl}
We recall how to associate cohomological invariants to simplicial schemes and truncated
simplicial schemes (see \cite[Vbis, 2.3]{SGA 4 II}, \cite[5.2]{De74}, \cite[(6.2)]{Ts98}). 

If $\sT$ is a topos, we denote by $\sT^{\Delta}$
(resp.~$\sT^{\Delta[m]}$) the topos of cosimplicial objects (resp.~$m$-truncated
cosimplicial objects) in $\sT$. Let $\sA$ be a ring in $\sT$, and $\sA\hbul$ the 
constant cosimplicial ring defined by $\sA$. If $\sE\hbul$ is an $\sA\hbul$-module of 
$\sT^{\Delta}$ (resp.~$\sT^{\Delta[m]}$), one associates to $\sE\hbul$ the complex 
\gan{ \varepsilon_*\sE\hbul\ =\ \sE^0 \to \sE^1 \to \cdots \to \sE^r 
\xra{\sum_j (-1)^j\del^j } \sE^{r+1} \to \cdots\\
\text{(resp.}\quad \varepsilon^m_*\sE\hbul\ =\ \sE^0 \to \sE^1 \to \cdots \to \sE^m
\to 0 \to \cdots \quad ). }
One views $\varepsilon_*\sE\hbul$ (resp.~$\varepsilon^m_*\sE\hbul$) as a filtered
complex of $\sA$-modules using the naive filtration. The functors $\varepsilon_*$ and
$\varepsilon^m_*$ are exact functors from the category of $\sA\hbul$-modules to the
category of filtered complexes of $\sA$-modules (which means that they transform a
short exact sequence of $\sA\hbul$-modules into a short exact sequence of filtered
complexes, i.e., such that the sequence of $\Fil^i$'s is exact for all $i$). Hence,
they factorize so as to define exact functors $\RR\varepsilon_*$ and
$\RR\varepsilon^m_*$ from $D^+(\sT^{\Delta}, \sA\hbul)$ (resp.~$D^+(\sT^{\Delta[m]},\sA
\hbul)$) to $D^+F(\sT,\sA)$. For any complex $\sE^{\sbul,\sbul} \in
D^+(\sA\hbul)$, they provide functorial spectral sequences
\eq{simplspecseq}{ E_1^{r,q} = \sH^q(\sE^{r,\sbul}) \Rightarrow 
\sH^{r+q}(\RR\varepsilon_*(\sE^{\sbul,\sbul})) }
and similarly for $\RR\varepsilon^m_*$ with $E_1^{r,q} = 0$ for $r>m$ (we use here the
first index to denote the simplicial degree). Note that the truncation functor induces
a functorial morphism
\eq{cohtrunc}{ \RR\varepsilon_*(\sE^{\sbul,\sbul}) \lra 
\RR\varepsilon^m_*(\sk_m(\sE^{\sbul,\sbul})), }
and therefore a morphism between the corresponding spectral sequences
\eqref{simplspecseq}. It follows that, if $\sH^q(\sE^{r,\sbul}) = 0$ for $q < 0$ and all
$r$, then the morphism \eqref{cohtrunc} is a quasi-isomorphism in degrees $< m$. 

Let $Y\lbul$ be a simplicial scheme (resp. $m$-truncated simplicial scheme), and
$\Sets$ the topos of sets. If $R$ is a commutative ring, and $\sE\hbul$ a (Zariski,
\'etale, \ldots) sheaf of $R$-modules on $Y\lbul$, one can associate to $\sE\hbul$ a
cosimplicial $R\hbul$-module $\Gamma\hbul(Y\lbul, \sE\hbul) \in \Sets^{\Delta}$
(resp.~$\Sets^{\Delta[m]}$) by setting for all $r \geq 0$
\eqn{ \Gamma^r(Y\lbul, \sE\hbul) = \Gamma(Y_r, \sE^r). }
The functor $\Gamma\hbul$ can be derived, and its right derived functor
$\RR\Gamma\hbul$ can be computed using resolutions by complexes $\sI^{\sbul,\sbul}$
such that, for each $r,q$, the sheaf $\sI^{r,q}$ is acyclic on $Y_r$. The cohomology 
of $Y\lbul$ with coefficients in a complex $\sE^{\sbul,\sbul}$ is then by definition 
\gan{ \RR\Gamma(Y\lbul,\sE^{\sbul,\sbul}) = \RR\varepsilon_* 
\RR\Gamma\hbul(Y\lbul,\sE^{\sbul,\sbul}) \quad\quad 
\text{(resp.\ \ }\RR\varepsilon^m_*),\\
H^q(Y\lbul,\sE^{\sbul,\sbul}) = H^q(\RR\Gamma(Y\lbul,\sE^{\sbul,\sbul})). }
If $Y\lbul$ is a smooth simplicial (resp.\ $m$-truncated simplicial) $R$-scheme, this
can be applied to the complex $\Omega\hbul_{Y\lbul/R}$ and to its sub-complexes
$\sigma_{\geq i}\Omega\hbul_{Y\lbul/R}$, defining the naive filtration. This provides
the definition of the de Rham cohomology of $Y\lbul$, and of its Hodge filtration.

\begin{prop}\label{DRsimpl}
Let $K$ be a field of charactristic $0$, $X$ a proper and smooth $K$-scheme, $Y\lbul
\to X$ an $m$-truncated proper hypercovering of $X$ over $K$ such that $Y_r$ is proper
and smooth for all $r$. Then, for all $q < m$, the canonical homomorphism
\eq{dRsimpl}{ H^q(X,\Omega\hbul_{X/K}) \lra H^q(Y\lbul, 
\Omega\hbul_{Y\lbul/K}) }
is an isomorphism of filtered $K$-vector spaces for the Hodge filtrations.
\end{prop}

\begin{proof}
Since algebraic de Rham cohomology (endowed with the Hodge filtration) commutes with
base field extensions, standard limit arguments allow to assume that $K$ is of finite
type over $\Q$. Choosing an embedding $\iota : K \inj \C$, we are reduced to the case
where $K = \C$. Using resolution of singularities, we can find a proper and smooth
hypercovering $Z\lbul$ of $X$ such that $\sk_m(Z\lbul) = Y\lbul$. As the morphism
\eqref{cohtrunc} for $\sigma_{\geq i}\Omega\hbul_{Z\lbul/\C}$ is a quasi-isomorphism in
degrees $< m$ for all $i$, it suffices to prove the proposition with $Y\lbul$ replaced
by $Z\lbul$. This now follows from \cite[Prop.~(8.2.2)]{De74}.
\end{proof}

\begin{cor}\label{Hodgesimpl}
Under the assumptions a) and b) of Lemma \ref{mtruncres}, assume in addition that $X_K$
is proper and smooth, and that $H^q(X_K, \sO_{X_K}) = 0$ for some $q < m$. Then the
smallest Hodge slope of $H^q(Y\lbul{}_{K'}, \Omega\hbul_{Y\lbul{}_{K'}})$ is at
least $1$.
\end{cor}

\begin{proof}
Assumption a) and b) imply that the hypotheses of the proposition are satisfied by 
$Y\lbul{}_{K'} \to X_{K'}$, and the corollary is then clear.
\end{proof}

\subsection{}\label{Cryssimpl}
Let $\uSigma_n, \uSigma$ be as in the proof of Theorem \ref{Semistable}. We now denote
by $\uY\lbul = (Y\lbul, M_{Y\lbul})$ an $m$-truncated simplicial log scheme over
$\uSigma_1$. We assume that each $\uY_r$ is smooth of Cartier type over
$\uSigma_1$, so that, for all $n \geq 1$, its de Rham-Witt complex
$\WC{n}{\sbul}{\uY_r}$ is defined \cite[(4.1)]{HK94}. When $r$ varies, the
functoriality of the de Rham-Witt complex turns the family of complexes
$(\WC{n}{\sbul}{\uY_r})_{0\leq r \leq m}$ into a complex $\WC{n}{\sbul}{\uY\lbul}$ on
$Y\lbul$. One defines its cohomology as in \ref{Cohsimpl}, and one has similar
definitions for the de Rham-Witt complex $\WC{}{\sbul}{\uY\lbul} = \varprojlim_n
\WC{n}{\sbul}{\uY\lbul}$.

For a morphism $\alpha : [r] \to [s]$ in $\Delta[m]$, let $\alpha_{\crys} :
(\uY_s/\uSigma_n)_\crys \to (\uY_r/\uSigma_n)_\crys$ be the morphism between the log
crystalline topos induced by the corresponding morphism $\uY_s \to \uY_r$. One defines
the log crystalline topos $(\uY\lbul/\uSigma_n)_{\crys}$ as being the topos of families
of sheaves $(E^r)_{0 \leq r \leq m}$, where $E^r$ is a sheaf on the log crystalline
site $\Crys(\uY_r/\uSigma_n)$, endowed with a transitive family of morphisms
$\alpha_{\crys}^{-1}E^r \to E^s$ for morphisms $\alpha$ in $\Delta[m]$. In particular,
the family of sheaves $\sO_{\uY_r/\uSigma_n}$ defines the structural sheaf of
$(\uY\lbul/\uSigma_n)_{\crys}$, denoted by $\sO_{\uY\lbul/\uSigma_n}$. There is a
canonical morphism $u_{\uY\lbul/\uSigma_n} : (\uY\lbul/\uSigma_n)_{\crys} \to
Y\lbul{_{\Zar}}$, such that $u_{\uY\lbul/\uSigma_n*}(E\hbul)^r =
u_{\uY_r/\uSigma_n*}(E^r)$ for all $r$. If $E^{\sbul,\sbul}$ is a complex of abelian
sheaves in $(\uY\lbul/\uSigma_n)_{\crys}$, one proceeds as in \ref{Cohsimpl} to define
its log crystalline cohomology $\RR\Gamma_{\crys}(\uY\lbul/\uSigma_n,E^{\sbul,\sbul})$
and its projection on the Zariski topos $\RR u_{\uY\lbul/\uSigma_n*}(E^{\sbul,\sbul})$.
One gives similar definitions for the log crystalline topos
$(\uY\lbul/\uSigma)_{\crys}$ relative to $\uSigma$. By construction, there are
canonical isomorphisms
\eqa{logLerayn}{ \RR\Gamma(Y\lbul,\RR u_{\uY\lbul/\uSigma_n*}(E^{\sbul,\sbul})) & \riso 
& \RR\Gamma_{\crys}(\uY\lbul/\uSigma_n,E^{\sbul,\sbul}),\\
\RR\Gamma(Y\lbul,\RR u_{\uY\lbul/\uSigma\,*}(E^{\sbul,\sbul})) & \riso 
& \RR\Gamma_{\crys}(\uY\lbul/\uSigma,E^{\sbul,\sbul}). \label{logLeray}}

If $\uY\lbul \inj \uP\lbul$ is a closed immersion of the
$m$-truncated simplicial log scheme $\uY\lbul$ into a smooth $m$-truncated simplicial
$\uSigma_n$-log scheme $\uP\lbul$ (resp.~$\uSigma$-formal log scheme), the family of
PD-envelopes $\sP^{\log}_{\uY_r}(\uP_r)$ (resp.~completed PD-envelopes)
\cite[(5.4)]{Ka89} defines a sheaf $\sP^{\log}_{\uY\lbul}(\uP\lbul)$ on $\uY\lbul$, and
one can form the de Rham complex $\sP^{\log}_{\uY\lbul}(\uP\lbul)
\otimes_{\sO_{P\lbul}} \Omega\hbul_{\uP\lbul/\uSigma_n}$
(resp.~$\sP^{\log}_{\uY\lbul}(\uP\lbul) \otimes_{\sO_{P\lbul}}
\Omega\hbul_{\uP\lbul/\uSigma}$), which is supported in $\uY\lbul$. Because the
linearization functor $L$ used in the proof of the comparison theorem between
crystalline and de Rham cohomologies \cite[(6.9)]{Ka89} makes sense simplicially, this
theorem extends to the simplicial case and there is a canonical isomorphism in
$D^+(Y\lbul, W_n)$ (resp.~$D^+(Y\lbul, W)$)
\eqa{logcrysdRn}{ \RR u_{\uY\lbul/\uSigma_n*}(\sO_{\uY\lbul/\uSigma_n}) & \riso & 
\sP^{\log}_{\uY\lbul}(\uP\lbul) \otimes_{O_{\uP\lbul}} \Omega\hbul_{\uP\lbul/\uSigma_n} \\
\text{(resp.}\quad \RR u_{\uY\lbul/\uSigma\,*}(\sO_{\uY\lbul/\uSigma}) &  \riso & 
\sP^{\log}_{\uY\lbul}(\uP\lbul) \otimes_{\sO_{P\lbul}} 
\Omega\hbul_{\uP\lbul/\uSigma}\quad). 
\label{logcrysdR} }

\begin{prop}\label{DRWcrys}
With the hypotheses of \ref{Cryssimpl}, assume that $\uY\lbul$ is split. Then there
exists in $D^+(Y\lbul,W_n)$ \lp resp.~$D^+(Y\lbul,W)$\rp\ canonical isomorphisms
compatible with the transition morphisms and the Frobenius actions
\eqa{dRWncrys}{ \RR u_{\uY\lbul/\uSigma_n*}(\sO_{\uY\lbul/\uSigma_n}) 
& \riso & \WC{n}{\sbul}{\uY\lbul} \\
(\text{resp.}\quad \RR u_{\uY\lbul/\uSigma\,*}(\sO_{\uY\lbul/\uSigma}) 
& \riso & \WC{}{\sbul}{\uY\lbul}\quad).\label{dRWcrys}}
\end{prop}

The proof will use the following lemma, due to Nakkajima \cite[Lemma 6.1]{Na09}.

\begin{lem}\label{Nakklem}
Under the assumptions of \ref{DRWcrys}, there exists an $m$-truncated simplicial log
scheme $\uZ\lbul$ and a morphism of $m$-truncated simplicial log schemes $\uZ\lbul \to
\uY\lbul$ such that, for $0 \leq r \leq m$, $Z_r$ is a disjoint union of affine open
subsets of $Y_r$ covering $Y_r$, and the morphism $Z_r \to Y_r$ induces the natural
inclusion on each of these subsets.
\end{lem}

\begin{defn}\label{Wittscheme}
Let $X$ be a scheme on which $p$ is locally nilpotent, and $n \geq 1$ an integer. We
denote by $|X|$ the topological space underlying $X$, and by $W_n(X)$ the ringed space
$(|X|,W_n(\sO_X))$, which is a scheme (\cite[0, 1.5]{Il79} and \cite[1.10]{LZ04}). The
ideal $VW_{n-1}(\sO_X)$ carries a canonical PD-structure (\cite[0, 1.4]{Il79} and
\cite[1.1]{LZ04}), which turns the nilpotent immersion $u : X \inj W_n(X)$ into a
PD-thickening of $X$.

If $\uX = (X, M_X)$ is a log scheme, we denote by $W_n(\uX) = (W_n(X), M_{W_n(X)})$ 
the log scheme obtained by sending $M_X$ to $W_n(\sO_X)$ by the Teichm\"uller 
representative map, and taking the associated log structure \cite[Def.~(3.1)]{HK94}. 
The immersion $u$ is then in a natural way an exact closed immersion $u : \uX \inj 
W_n(\uX)$, functorial with respect to $\uX$. 
\end{defn}

\begin{lem}\label{Bisimpl}
Under the assumptions of \ref{DRWcrys}, there exists a bisimplicial log scheme 
$\uZ\lbbul$, $m$-truncated with respect to the first index and augmented towards 
$\uY\lbul$ with respect to the second index, a bisimplicial formal log scheme 
$\usT\lbbul$ over $\uSigma$, $m$-truncated with respect to the first index, and a 
closed immersion of bisimplicial formal log schemes $i\lbbul : \uZ\lbbul \inj 
\usT\lbbul$, such that the following conditions are satisfied:

\alphab For $0 \leq r \leq m$, $Z_{r,0}$ is a disjoint union of affine open subsets of
$Y_r$ covering $Y_r$, the augmentation morphism $Z_{r,0} \to Y_r$ induces the natural
inclusion on each of these subsets, and the canonical morphism $\uZ_{r,\sbul} \to
\cosk^{\uY_r}_0(\sk^{\uY_r}_0(\uZ_{r,\sbul}))$ is an isomorphism.

\alphab For $0 \leq r \leq m$ and $t \geq 0$, the formal log scheme $\usT_{r,t}$ is
smooth over $\uSigma$ \lp i.e., its reduction mod $p^n$ is smooth over $\uSigma_n$ for
all $n$\rp, and the canonical morphism $\usT_{r,\sbul} \to
\cosk^{\uSigma}_0(\sk^{\uSigma}_0(\usT_{r,\sbul}))$ is an isomorphism.

\alphab Let $i\lbbul{}_{;n} : \uZ\lbbul \inj \uT\lbbul{}_{;n}$ be the reduction mod
$p^n$ of $i\lbbul$, and let $u\lbbul{}_{;n} : \uZ\lbbul \inj W_n(\uZ\lbbul)$ denote
the morphism of bisimplicial log schemes defined by the canonical immersions. For
variable $n$, there exists a compatible family of $\uSigma_n$-morphisms of
bisimplicial schemes $h\lbbul{}_{;n} : W_n(\uZ\lbbul) \to \uT\lbbul{}_{;n}$ such that
$h\lbbul{}_{;n} \circ u\lbbul{}_{;n} = i\lbbul{}_{;n}$.
\end{lem}

\begin{proof}
Let $j\lbul : \uZ\lbul \to \uY\lbul$ be a morphism of $m$-truncated simplicial log
schemes satisfying the conclusions of Lemma \ref{Nakklem}. One chooses a decomposition
$\uZ_r = \coprod_\alpha \uZ_r^\alpha$, with $Z_r^\alpha \subset Y_r$ open affine such
that $j_r\restr{Z_r^\alpha}$ is the natural inclusion.

Let $\uZ_{r;1}^\alpha = \uZ_r^\alpha$. Since $\uZ_r^\alpha$ is affine and smooth
over $\uSigma_1$, and $\uSigma_{n-1} \inj \uSigma_n$ is a nilpotent exact closed
immersion, there exists for each $r, \alpha$ and each $n \geq 2$ a smooth log scheme
$\uZ_{r;n}^\alpha$ over $\uSigma_n$ endowed with an isomorphism $\uZ_{r;n-1}^\alpha
\xra{\sim} \uSigma_{n-1} \times_{\uSigma_n} \uZ_{r;n}^\alpha$ \cite[Prop.~(3.14)
(1)]{Ka89}. Taking limits when $n \to \infty$, we obtain a smooth formal log scheme
$\usZ_r^\alpha$ over $\uSigma$ and an isomorphism $\uZ_r^\alpha \xra{\sim} \uSigma_1
\times_{\uSigma} \usZ_r^\alpha$. Moreover, the smoothness of $\uZ_{r;n}^\alpha$ over 
$\uSigma_n$ for all $n$ implies that we can find inductively a compatible family of
$\uSigma_n$-morphisms $g_{r;n}^\alpha : W_n(\uZ_r^\alpha) \to \uZ_{r;n}^\alpha$ such
that the composition $\uZ_r^\alpha \inj W_n(\uZ_r^\alpha) \to 
\uZ_{r;n}^\alpha$ is the chosen immersion $\uZ_r^\alpha \inj \uZ_{r;n}^\alpha$. 

Let $\uZ_{r;n} = \coprod_\alpha\uZ_{r;n}^\alpha$, $\usZ_r =
\coprod_\alpha\usZ_r^\alpha$, let $v_{r;n} : \uZ_r \inj \uZ_{r;n}$, $v_r : \uZ_r \inj
\usZ_r$ be defined by the immersions $\uZ_r^\alpha \inj \uZ_{r;n}^\alpha$ and
$\uZ_r^\alpha \inj \usZ_r^\alpha$, and let $g_{r;n} : W_n(\uZ_r) \to \uZ_{r;n}$ be
defined by the morphisms $g_{r;n}^\alpha$. We now use the method of Chiarellotto and
Tsuzuki (\cite[11.2]{CT03}, \cite[7.3]{Tz04}) to deduce from these data a closed
immersion $i\lbul$ of $\uZ\lbul$ into an $m$-truncated simplicial formal log scheme
$\usT\lbul$, smooth over $\uSigma$, with reduction $\uT_{\sbul;n}$ over $\uSigma_n$,
and a compatible family of $\uSigma_n$-morphisms of $m$-truncated simplicial log
schemes $h_{\sbul;n} : W_n(\uZ\lbul) \to \uT_{\sbul;n}$ such that $h_{\sbul;n} \circ
u_{\sbul;n} = i_{\sbul;n}$, where $u_{\sbul;n} : \uZ_{\sbul;n} \inj
W_n(\uZ_{\sbul;n})$ is the canonical morphism, and $i_{\sbul;n}$ is the reduction mod
$p^n$ of $i\lbul$. First, we set for $0 \leq s \leq m$ 
\eqn{ \Gamma_s(\usZ_r) = \prod_{\gamma : [r] \to [s]} \usZ_{r,\gamma}, }
where the product is taken over $\uSigma$ and indexed by the set of morphisms $\gamma :
[r] \to [s]$ in $\Delta[m]$, and $\usZ_{r,\gamma} = \usZ_r$ for all $\gamma$. Then any
morphism $\eta : [s'] \to [s]$ in $\Delta[m]$ defines a morphism $\Gamma_s(\usZ_r) \to
\Gamma_{s'}(\usZ_r)$ having as component of index $\gamma'$ the projection of
$\Gamma_s(\usZ_r)$ to the factor of index $\eta\circ\gamma'$. One obtains in this way
an $m$-truncated simplicial formal log scheme $\Gamma\lbul(\usZ_r)$ over $\uSigma$, the
terms of which are smooth over $\uSigma$.

For each $\gamma : [r] \to [s]$, there is a commutative diagram 
\eqn{ \xymatrix@C=40pt@M=6pt{ 
W_n(\uZ_s) \ar[r]^{W_n(\gamma)} & W_n(\uZ_r) \ar[dr]^{g_{r;n}} \\
\uZ_s \ar@{^{(}->}[u]_{u_{s;n}} \ar[r]^{\gamma} & \uZ_r \ar@{^{(}->}[u]_{u_{r;n}}
\ar@{^{(}->}[r]^{v_{r;n}} & \uZ_{r;n} \ar@{^{(}->}[r] & \usZ_r.
} }
For fixed $r$ and variable $s$, the family of morphisms $\uZ_s \to \Gamma_s(\usZ_r)$
having the composition $\uZ_s \xra{\gamma} \uZ_r \inj \usZ_r$ as component of index
$\gamma$ defines a morphism of $m$-truncated simplicial formal log schemes $\uZ\lbul
\to \Gamma\lbul(\usZ_r)$. We set
\eqn{ \usT\lbul = \prod_{0 \leq r \leq m} \Gamma\lbul(\usZ_r), }
and we define $i\lbul : \uZ\lbul \to \usT\lbul$ as having the previous morphism as
component of index $r$, for $0 \leq r \leq m$. For each $r$, the morphism $\uZ_r \to
\Gamma_r(\usZ_r)$ has the closed immersion $v_r : \uZ_r \inj \usZ_r$ as component of
index $\Id_{[r]}$. It follows that $\uZ_r \to \usT_r$ is a closed immersion for all
$r$.

Similarly, the family of morphisms $W_n(\uZ_s) \to \Gamma_s(\usZ_r)$ having the
composition $W_n(\uZ_s) \xra{W_n(\gamma)} W_n(\uZ_r) \xra{g_{r;n}} \uZ_{r;n} \inj
\usZ_r$ as component of index $\gamma$ defines a morphism of $m$-truncated simplicial
log schemes $W_n(\uZ\lbul) \to \Gamma\lbul(\usZ_r)$. We define $h\lbul : W_n(\uZ\lbul)
\to \usT\lbul$ as having the previous morphism as component of index $r$ for $0 \leq r
\leq m$, and $h_{\sbul;n} : W_n(\uZ\lbul) \to \uT_{\sbul;n}$ as being the reduction of
$h\lbul$ mod $p^n$. It is clear that $h_{\sbul;n} \circ u_{\sbul;n} = i_{\sbul;n}$, and
that the morphisms $h_{\lbul;n}$ form a compatible family when $n$ varies.

We now set $\uZ_{\sbul,0} = \uZ\lbul$, $\usT_{\sbul,0} = \usT\lbul$, and we define 
\eqn{ \uZ\lbbul = \cosk^{\uY\lbul}_0(\uZ_{\sbul,0}), \quad\quad 
\usT\lbbul = \cosk^{\uSigma}_0(\usT_{\sbul,0}), }
the coskeletons being taken respectively in the category of simplicial $m$-truncated
simplicial log schemes over $\uY\lbul$ and of simplicial $m$-truncated simplicial
formal log schemes over $\uSigma$. The augmentation morphism $\uZ_{\sbul,0} \to
\uY\lbul$ is given by $j\lbul$, and the morphism $i\lbbul$ is defined by setting
$i_{\sbul,0} = i\lbul : \uZ_{\sbul,0} \inj \usT_{\sbul,0}$, and extending
$i_{\sbul,0}$ by functoriality to the coskeletons. As seen above, $i_{\sbul,0}$ is a
closed immersion, and it follows from the construction of coskeletons that $i_{\sbul,
t}$ is a closed immersion for all $t$. Since $\cosk^{\uSigma}_0(\usT_{r,0})_{t} =
\usT_r \times_{\uSigma} \times \cdots \times_{\uSigma} \usT_r$ ($t+1$ times),
$\usT_{r,t}$ is smooth over $\uSigma$ for all $r, t$. Finally, we define $h\lbbul{}_{;n} :
W_n(\uZ\lbbul) \to \uT\lbbul{}_{;n}$ as being the composition
\eqn{W_n(\cosk^{\uY\lbul}_0(\uZ_{\sbul,0})) \to
\cosk^{W_n(\uY\lbul)}_0(W_n(\uZ_{\sbul,0})) \to 
\cosk^{\uSigma_n}_0(\uT_{\sbul,0;n}) \simeq 
\uSigma_n \times_{\uSigma} \cosk^{\uSigma}_0(\usT_{\sbul,0}), }
where the first map is defined by the universal property of the coskeleton (and is 
actually an isomorphism), the second one is defined by functoriality by the 
morphism $h_{\sbul;n} : W_n(\uZ_{\sbul,0}) \to \uT_{\sbul;n} = \uT_{\sbul,0;n}$, and the 
last one is the base change isomorphism for coskeletons. The relations $h\lbbul{}_{;n} 
\circ u\lbbul{}_{;n} = i\lbbul{}_{;n}$ and the compatibility for variable $n$ follow 
from the similar properties for the morphisms $h_{\sbul;n}$. Properties a) - c) of the
Lemma are then satisfied.
\end{proof}

\subsection{}\label{proofdRWcrys}\textit{Proof of Proposition \ref{DRWcrys}}. Let 
\eqn{ \xymatrix@M=6pt{ \uZ\lbbul \ar@{^{(}->}[r]^{i\lbbul} \ar[d]_{j\lbbul} & 
\usT\lbbul \ar[d] \\
\uY\lbul \ar[r] & \uSigma } }
be a commutative diagram satisfying the properties of Lemma \ref{Bisimpl}. Since, for
all $r \leq m$, the morphism $j_{r,0}$ is locally an open immersion, the scheme
underlying $\uZ_{r,t}$ is the usual fibred product $Z_{r,0} \times_{Y_r} \cdots
\times_{Y_r} Z_{r,0}$ ($t+1$ times). Keeping the notations of the proof of Lemma
\ref{Bisimpl}, let $\fU_r = (Z_r^\alpha)_\alpha$ be an affine covering of $Y_r$ such
that $Z_{r,0} = \coprod_\alpha Z_r^\alpha$ and $j_{r,0}\restr{Z_r^\alpha}$ is the
natural inclusion. Then, for any abelian sheaf $\sE$ on $Y_r$, the complex
\eqn{ \varepsilon_{r\,*}(j_{r,\sbul\,*}j_{r,\sbul}^{-1}\sE) = 
\Big[j_{r,0\,*}j_{r,0}^{-1}\sE \to 
\cdots \to j_{r,t\,*}j_{r,t}^{-1}\sE \xra{\sum_k (-1)^k\del_k} 
j_{r,t+1\,*}j_{r,t+1}^{-1}\sE \to \cdots \Big] }
is the \v{C}ech resolution of $\sE$ defined by the covering $\fU_r$. If $\sE\hbul$ is
an abelian sheaf on $Y\lbul$, the fact that $j_{\sbul,0}$ is an augmentation morphism
in the category of $m$-truncated simplicial schemes implies that the complex
$\varepsilon_{r\,*}(j_{r,\sbul\,*}j_{r,\sbul}^{-1}\sE^r)$ is functorial with respect to
$[r] \in \Delta[m]$, and we obtain a resolution
$\varepsilon_{\sbul\,*}(j_{\sbul,\sbul\,*}j_{\sbul,\sbul}^{-1}\sE\hbul)$ of $\sE\hbul$
in the category of abelian sheaves on $Y\lbul$. In particular, taking into account that
each $j_{q,q'}$ is locally an open immersion, we obtain for all $n$ a resolution of the
de Rham-Witt complex of $\uY\lbul$ given by
\eq{CechdRW}{ \WC{n}{\sbul}{\uY\lbul} \xra{\text{ qis }} 
\varepsilon_{\sbul\,*}(j\lbbul{}_*\WC{n}{\sbul}{\uZ\lbbul}). }

On the other hand, one can also define for all $r$ a complex on 
$\Crys(\uY_r/\uSigma_n)$ by setting 
\mln{ \varepsilon_{r\,*}(j_{r,\sbul\,\crys\,*}(\sO_{\uZ_{r,\sbul}/\uSigma_n})) = \\
\Big[j_{r,0\,\crys\,*}(\sO_{\uZ_{r,0}/\uSigma_n}) \to \cdots \to 
j_{r,t\,\crys\,*}(\sO_{\uZ_{r,t}/\uSigma_n}) 
\xra{\sum_k (-1)^k\del_k} \cdots \Big]. }
Since $Z_{r,\sbul} \to Y_r$ is the \v{C}ech simplicial scheme defined by an affine open
covering of $Y_r$, this complex is a resolution of $\sO_{\uY_r/\uSigma_n}$ \cite[III,
Prop.~3.1.2 and V, Prop.~3.1.2]{Be74}. Since $\uZ\lbbul$ is a bisimplicial scheme,
these resolutions are functorial with respect to $[r]$ and yield a resolution
$\varepsilon_{\sbul\,*}(j_{\sbul,\sbul\,\crys\,*}(\sO_{\uZ\lbbul/\uSigma_n}))$ of
$\sO_{\uY\lbul/\uSigma_n}$. Let $\uT\lbbul{}_{;n}$ be the reduction mod $p^n$ of
$\usT\lbbul$. The linearization functor $L$ \cite[(6.9)]{Ka89} is functorial with
respect to embeddings, hence it provides a complex
$L(\Omega\hbul_{\uT\lbbul{}_{;n}/\uSigma_n})$ on $\Crys(\uZ\lbbul/\uSigma_n)$. This
complex is a resolution of $\sO_{\uZ\lbbul/\uSigma_n}$ thanks to the log Poincar\'e
lemma which follows from \cite[Prop.~(6.5)]{Ka89}. For each $(r,t)$ and each $i$, one
checks easily that the term $j_{r,t\,\crys\,*}(L(\Omega^i_{\uT_{r,t;n}/\uSigma_n}))$ is
acyclic with respect to $u_{\uY_r/\uSigma_n*}$ (use \cite[V, (2.2.3)]{Be74} and the
equality $u_{\uY_r/\uSigma_n*} \circ j_{r,t\,\crys\,*} = j_{r,t\,*} \circ
u_{\uZ_{r,t}/\uSigma_n*}$). Hence, the complex
$\varepsilon_{\sbul\,*}(j_{\sbul,\sbul\,\crys\,*}(L(\Omega\hbul_{\uT\lbbul{}_{;n}/\uSigma_n})))$
is an $u_{\uY\lbul/\uSigma_n*}$-acyclic resolution of $\sO_{\uY\lbul/\uSigma_n}$.
Moreover, the closed immersion of bisimplicial schemes $i\lbbul$ defines a family of
PD-envelopes $\sP^{\log}_{\uZ\lbbul}(\uT\lbbul{}_{;n})$, supported in $Z\lbbul$. They
provide a de Rham complex $\sP^{\log}_{\uZ\lbbul}(\uT\lbbul{}_{;n}) \otimes
\Omega\hbul_{\uT\lbbul{}_{;n}/\uSigma_n}$, which can be viewed as a complex of abelian
sheaves on $Z\lbbul$, and it follows from \cite[V, (2.2.3)]{Be74} that
\eqn{u_{\uY\lbul/\uSigma_n*}(j_{\sbul,\sbul\,\crys\,*}(L(\Omega\hbul_{\uT\lbbul{}_{;n}/\uSigma_n})))
= j_{\sbul,\sbul\,*}(\sP^{\log}_{\uZ\lbbul}(\uT\lbbul{}_{;n}) \otimes
\Omega\hbul_{\uT\lbbul{}_{;n}/\uSigma_n}). }
As the $j_{r,t}$'s are affine morphisms (as a consequence of 1) in our general
conventions), we finally get in $D^+(Z\lbul, W_n)$ an isomorphism
\eq{Cechcrys}{ \RR u_{\uY\lbul/\uSigma_n*}(\sO_{\uY\lbul/\uSigma_n}) \riso 
\varepsilon_{\sbul\,*}(j_{\sbul,\sbul\,*}(\sP^{\log}_{\uZ\lbbul}(\uT\lbbul{}_{;n}) \otimes
\Omega\hbul_{\uT\lbbul{}_{;n}/\uSigma_n})). }

To prove Proposition \ref{DRWcrys}, it suffices to define a quasi-isomorphism between
the right hand sides of \eqref{CechdRW} and \eqref{Cechcrys}. Note that, for each $r,
t,i$, the sheaves $\WC{n}{i}{\uZ_{r,t}}$ and $\sP^{\log}_{\uZ_{r,t}}(\uT_{r,t;n})
\otimes \Omega^i_{\uT_{r,t;n}/\uSigma_n}$ are $j_{r,t\,*}$-acyclic. Indeed, $Z_{r,t}$
is a disjoint union of affine open subsets of $Y_r$, and on the one hand
$\WC{n}{i}{\uZ_{r,t}}$ has a finite filtration with subquotients which are
coherent over suitable Frobenius pullbacks of $Z_{r,t}$ \cite[Th.~(4.4)]{HK94},
on the other hand $\sP^{\log}_{\uZ_{r,t}}(\uT_{r,t;n}) \otimes
\Omega^i_{\uT_{r,t;n}/\uSigma_n}$ is a quasi-coherent $\sO_{T_{r,t;n}}$-module with
support in $Z_{r,t}$, hence is a direct limit of submodules which have a finite
filtration with subquotients which are coherent over $Z_{r,t}$. Therefore, it
suffices to construct a quasi-isomorphism
\eq{simpldRWcrys}{ \sP^{\log}_{\uZ\lbbul}(\uT\lbbul{}_{;n}) \otimes
\Omega\hbul_{\uT\lbbul{}_{;n}/\uSigma_n} \lra \WC{n}{\sbul}{\uZ\lbbul} }
in the category of complexes of $W_n$-modules over $Z\lbbul$. 

We can now argue as in the proof of \cite[Th.~(4.19)]{HK94}. Since the PD-immersion
$u_{r,t;n} : \uZ_{r,t} \inj W_n(\uZ_{r,t})$ is an exact closed immersion for all $r,
t$, the morphism $h\lbbul{}_{;n} : W_n(\uZ\lbbul) \to \uT\lbbul{}_{;n}$ defines
uniquely a PD-morphism $\sP^{\log}_{\uZ\lbbul}(\uT\lbbul{}_{;n}) \to
W_n(\sO_{\uZ\lbbul})$ in the category of sheaves of $W$-modules on the bisimplicial
scheme $T\lbbul{}_{;n}$. As $h\lbbul{}_{;n}$ is a morphism of bisimplicial log schemes,
it defines by functoriality a morphism of complexes
$\Omega\hbul_{\uT\lbbul{}_{;n}/\uSigma_n} \to \Omega\hbul_{W_n(\uZ\lbbul)/\uSigma_n}$
on $T\lbbul{}_{;n}$. This morphism extends as a morphism of complexes with support in
$Z\lbbul$
\eqn{ \sP^{\log}_{\uZ\lbbul}(\uT\lbbul{}_{;n}) \otimes
\Omega\hbul_{\uT\lbbul{}_{;n}/\uSigma_n} \lra
\Omega\hbul_{W_n(\uZ\lbbul)/\uSigma_n}/\sN\hbul\lbbul, }
where $\sN\hbul\lbbul \subset \Omega\hbul_{W_n(\uZ\lbbul)/\uSigma_n}$ denotes the
graded ideal generated by the sections $d(a^{[i]})-a^{[i-1]}da$ for all sections $a$ of
$VW_{n-1}(\sO_{Z\lbbul})$ and all $i \geq 1$. The differential graded algebra
$\WC{n}{\sbul}{\uZ\lbbul}$ is a quotient of $\Omega\hbul_{W_n(\uZ\lbbul)/\uSigma_n}$
\cite[Prop.~(4.7)]{HK94}, and the generators of $\sN\hbul\lbbul$ vanish in
$\WC{n}{\sbul}{\uZ\lbbul}$ (because $\WC{}{\sbul}{\uZ\lbbul}$ is $p$-torsion free), so we
finally get the morphism \eqref{simpldRWcrys}. To check that it is a quasi-isomorphism,
it suffice to do so on each $Z_{r,t}$, and this follows from \cite[Th.~(4.19)]{HK94}. 
We obtain in this way the isomorphism \eqref{dRWncrys}. 

To construct the isomorphism \eqref{dRWcrys}, it suffices to observe that the 
compatibility of the previous constructions when $n$ varies implies that they make 
sense in the category of inverse systems indexed by $n\in\N$. Then one can apply the 
functor $\RR\varprojlim_n$ to the isomorphism \eqref{dRWncrys} viewed an an 
isomorphism in the derived category of inverse systems of sheaves of $W$-modules on 
$Y\lbul$, and this provides the isomorphism \eqref{dRWcrys}, since the local structure of 
the $\WC{n}{i}{}$ recalled above implies that they form a $\varprojlim$-acyclic inverse system.

The isomorphisms \eqref{dRWncrys} and \eqref{dRWcrys} do not depend upon the choices
made in their construction. If $(\uZ\lbbul,\usT\lbbul,j\lbbul,i\lbbul,h\lbbul{}_{;n})$
and $(\uZ'\lbbul,\usT'\lbbul,j'\lbbul,i'\lbbul,h'\lbbul{}_{;n})$ are two sets of data
provided by Lemma \ref{Bisimpl}, one can construct a third set of data
$(\uZ''\lbbul,\usT''\lbbul,j''\lbbul,i''\lbbul,h''\lbbul{}_{;n})$ mapping to the two
previous ones by setting
\eqn{ \uZ''\lbbul = \uZ\lbbul \times_{\uY\lbul} \uZ'\lbbul, \quad\quad
\usT''\lbbul = \usT\lbbul \times_{\uSigma} \usT'\lbbul,  }
and defining $j''\lbbul$, $i''\lbbul$ and $h''\lbbul{}_{;n}$ by functoriality. Then the
independence property of \eqref{dRWncrys} and \eqref{dRWcrys} follows from the
functoriality of the canonical isomorphisms used in their construction with respect to
the projections from $(\uZ''\lbbul, \usT''\lbbul)$ to $(\uZ\lbbul,\usT\lbbul)$ and
$(\uZ'\lbbul,\usT'\lbbul)$. Moreover, one can also prove the functoriality of
\eqref{dRWncrys} and \eqref{dRWcrys} with respect to $\uY\lbul$ by similar arguments
using the graph construction: for a morphism $\varphi\lbul : \uY'\lbul \to \uY\lbul$
between two $m$-truncated simplicial log schemes satisfying the assumptions of Lemma
\ref{DRWcrys}, one can find sets of data
$(\uZ\lbbul,\usT\lbbul,j\lbbul,i\lbbul,h\lbbul{}_{;n})$ and
$(\uZ'\lbbul,\usT'\lbbul,j'\lbbul,i'\lbbul,h'\lbbul{}_{;n})$ satisfying the conditions
of Lemma \ref{Bisimpl} relatively to $\uY\lbul$ and $\uY'\lbul$, and such that there
exists morphisms of bisimplicial log schemes $\psi\lbbul : \uZ'\lbbul \to \uZ\lbbul$,
$\theta\lbbul : \usT'\lbbul \to \usT\lbbul$ satisfying the obvious compatibilities.
Then the functoriality of \eqref{dRWncrys} and \eqref{dRWcrys} with respect to
$\varphi\lbul$ follows from the functoriality of the canonical isomorphisms used in
their construction with respect to $\varphi\lbul$, $\psi\lbbul$ and $\theta\lbbul$. In
particular, one obtains in this way that the isomorphisms \eqref{dRWncrys} and
\eqref{dRWcrys} are compatible with the Frobenius actions.
\hfill$\Box$
\medskip

\subsection{}\label{ProofWittvanish}
\textit{Proof of Theorem \ref{Wittvanish}, assuming Theorem \ref{Tworeg}}. 
To conclude this section, we prove that Theorem \ref{Tworeg} implies Theorem
\ref{Wittvanish}. We keep the notations of \ref{Main}, and we first observe that if
Theorem \ref{Wittvanish} holds when $R$ is complete, then it holds in general. Indeed,
let $\widehat{R}$ be the completion of $R$, and $\widetilde{X} = X_{\widehat{R}}$. Then
$\widetilde{X}$ is a regular scheme: on the one hand, its generic fibre is smooth over
$\widehat{K} = \Frac(\widehat{R})$; on the other hand, its special fibre is isomorphic
to $X_k$, and the completions of the local rings of $X$ and $\widetilde{X}$ are
isomorphic at any corresponding points of their special fibres. It follows that
$\widetilde{X}$ satisfies the assumptions of Theorem \ref{Wittvanish} relatively to
$\widehat{R}$, and the theorem for $\widetilde{X}$ implies the theorem for $X$.

Therefore, we assume in the rest of the proof that $R$ is complete. We fix an integer
$m > q$. Let $K'$ be a finite extension of $K$, with ring of integers $R'$ and residue
field $k'$, such that there exists an $m$-truncated simplicial log scheme $\uY\lbul$
over $R'$, with an augmentation morphism $u : Y_0 \to X_{R'}$, such that properties a)
- c) of Lemma \ref{mtruncres} are satisfied. Let $W'_n=W_n(k')$, $W'=W(k')$, $K'_0 =
\Frac(W')$, and let $\uSigma'_n$, $\uSigma'$ be the log schemes defined by $W'_n$, $W'$
as in \ref{Semistable}.

Thanks to property a) of Lemma \ref{mtruncres}, the log schemes $(\uY_r)_{k'}$ are
smooth of Cartier type over $\uSigma'_1$. Therefore, we can consider the log crystalline
cohomology of $\uY_{\sbul\,k'}$ 
\eqn{ \RR\Gamma_{\crys}(\uY_{\sbul\,k'}/\uSigma', \sO_{\uY_{\sbul\,k'}/\uSigma'}) :=
\RR\varepsilon^m_*\RR\Gamma\hbul_{\crys}(\uY_{\sbul\,k'}/\uSigma',
\sO_{\uY_{\sbul\,k'}/\uSigma'}), }
as defined in \ref{Cryssimpl}. Using the naive filtration on the functor
$\RR\varepsilon^m_*$ (see \ref{Cohsimpl}), its basic properties follow from those of
the log crystalline cohomology of the proper and smooth log schemes $(Y_r)_{k'}$. In
particular, since $Y_r$ is proper over $\uSigma'_1$ for all $r$, the complex
$\RR\Gamma_{\crys}(\uY_{\sbul\,k'}/\uSigma', \sO_{\uY_{\sbul\,k'}/\uSigma'})$ is a
perfect complex of $W'$-modules, and the cohomology space
$H^q_{\crys}(\uY_{\sbul\,k'}/\uSigma', \sO_{\uY_{\sbul\,k'}/\uSigma'})\otimes K'_0$ is
a finite dimensional $K'_0$-vector space. By functoriality, it is endowed with the 
semi-linear Frobenius action defined by the absolute Frobenius endomorphism of 
$\uY_{\sbul\,k'}$.

From \eqref{logLeray} and \eqref{dRWcrys}, we deduce an isomorphism 
\eqn{ H^q_{\crys}(\uY_{\sbul\,k'}/\uSigma', \sO_{\uY_{\sbul\,k'}/\uSigma'})\otimes 
K'_0 \riso H^q(\uY_{\sbul\,k'}, \WC{}{\sbul}{\uY_{\sbul\,k'}}) \otimes K'_0, }
which is compatible with the Frobenius actions thanks to Proposition \ref{DRWcrys}. 
The filtration of the complex $\WC{}{\sbul}{\uY_{\sbul\,k'}}$ by the subcomplexes 
$\sigma_{\geq i}\WC{}{\sbul}{\uY_{\sbul\,k'}}$ provides a spectral sequence 
\eqn{ E_1^{i,j} = H^j(\uY_{\sbul\,k'}, \WC{}{i}{\uY_{\sbul\,k'}}) \otimes K'_0 
\Longrightarrow H^{i+j}_{\crys}(\uY_{\sbul\,k'}/\uSigma',
\sO_{\uY_{\sbul\,k'}/\uSigma'}) \otimes K'_0, }
which is endowed with a Frobenius action. Using the naive filtration on
$\RR\varepsilon^m_*$, we deduce from the case of a single log scheme that each term
$E_1^{i,j}$ is a finite dimensional $K_0$-vector space on which the Frobenius action 
is bijective with slopes in $[i,i+1[$. Therefore the spectral sequence degenerates at 
$E_1$, and, taking \eqref{dRW0} into account, we get in particular an isomorphism 
\eq{simplslope}{ (H^q_{\crys}(\uY_{\sbul\,k'}/\uSigma', \sO_{\uY_{\sbul\,k'}/\uSigma'})
\otimes K'_0)^{<1} \riso H^q(Y_{\sbul\,k'}, W\sO_{Y_{\sbul\,k'},\Q}). }

Since $\uY\lbul$ satisfies property a) of \ref{mtruncres}, the construction of the
monodromy operator $N$ on log crystalline cohomology can be extended to the case of
$\uY_{\sbul\,k'}$ \cite[(6.3)]{Ts98}. Moreover, the Hyodo-Kato isomorphism $\rho$
can also be extended to the case of $\uY_{\sbul\,k'}$ \cite[(6.3.2)]{Ts98}, providing
an isomorphism
\eq{HKsimpl}{ \rho : H^q_{\crys}(\uY_{\sbul\,k'}/\uSigma',
\sO_{\uY_{\sbul\,k'}/\uSigma'}) \otimes K' \riso H^q(Y_{\sbul\,K'},
\Omega\hbul_{Y_{\sbul\,K'}}). }
Thus, $H^q_{\crys}(\uY_{\sbul\,k'}/\uSigma', \sO_{\uY_{\sbul\,k'}/\uSigma'}) \otimes 
K'$ inherits a filtered $(\varphi,N)$-module structure. 

It follows from \cite[Th.~7.1.1]{Ts98} (generalizing \cite[Th.~0.2]{Ts99}) that,
endowed with this structure, $H^q_{\crys}(\uY_{\sbul\,k'}/\uSigma',
\sO_{\uY_{\sbul\,k'}/\uSigma'}) \otimes K'_0$ is an admissible filtered
$(\varphi,N)$-module, corresponding to the Galois representation $H^q_{\et}(Y_{\sbul\,\oK},
\Q_p)$. Therefore it is weakly admissible. In particular, either
$H^q_{\crys}(\uY_{\sbul\,k'}/\uSigma', \sO_{\uY_{\sbul\,k'}/\uSigma'}) \otimes K'_0 =
0$, or its smallest Newton slope is greater or equal to its smallest Hodge
slope. Since $H^q(X_K, \sO_{X_K}) = 0$, Corollary \ref{Hodgesimpl} implies that the 
smallest Hodge slope is at least $1$. Therefore the part of
Newton slope $< 1$ vanishes. By \eqref{simplslope}, we obtain 
\eq{simplWittvanish}{ H^q(Y_{\sbul\,k'}, W\sO_{Y_{\sbul\,k'},\Q}) = 0. }

As $Y\lbul \to X_{R'}$ satisfies property \ref{mtruncres} c), there exists a
sub-$K$-extension $K_1 \subset K'$, with ring of integers $R_1$ and residue field
$k_1$, a semi-stable scheme $Y$ over $R_1$, a projective $R$-alteration $f : Y \to X$,
and finitely many $R$-embeddings $\sigma_i : R_1 \inj K'$ such that, if $u_1 : Y \to
X_{R_1}$ denotes the $R$-morphism defined by $f$, and if $Y_{\sigma_i}$ (resp.\ $u_{\sigma_i} : Y_{\sigma_i} \to
X_{R'}$) denotes the $R'$-scheme (resp.\ $R'$-morphism) deduced by base change via
$\sigma_i$ from $Y$ (resp.\ $u_1$), then $Y_0 = \coprod_i Y_{\sigma_i}$, and the augmentation
morphism $u : Y_0 \to X_{R'}$ is defined by $u\restr{Y_{\sigma_i}} = u_{\sigma_i}$. This provides a
commutative diagram
\eq{schemediag}{ \xymatrix@R=5pt{
& Y_{\sbul\,k'} \ar[r]^{u_{\sbul\,k'}} & X_{k'} \ar[dr] & \\
Y_{0,k'} = \coprod_i Y_{\sigma_i,k'} \ar[ur]^{s_0} \ar[dr] & & & X_k \\
& Y_{k_1} \ar@{^{(}->}[r] & Y_k \ar[ur]_{f_k} &
} }
in which we identify schemes with their Zariski topos, $Y_k := \Spec k \times_{\Spec 
R} Y$, and: 

\romain the morphism $u_{\sbul\,k'}$ is such that, for any sheaf $E$ on $X_{k'}$,
$u_{\sbul\,k'}^{-1}E$ is the family of sheaves $(u_r)_{k'}^{-1}E$, with $u_r : Y_r \to 
X_{R'}$ defined by the augmentation morphism,

\romain the morphism $s_0$ is such that, for any sheaf $F\hbul$ on $Y_{\sbul\,k'}$, 
$s_0^{-1}F\hbul = F^0$,

\romain the morphism $Y_{\sigma_i,k'} \to Y_{k_1}$ is the projection corresponding to 
$\sigma_i$. 

By functoriality, we obtain a commutative diagram for the corresponding Witt
cohomology spaces 
\eq{Wittdiag}{ \hspace{-0.9cm}\xymatrix@R=15pt@C=5pt@M=5pt{
& H^q(X_{k'}, W\sO_{X_{k'},\Q}) \ar[rr] & & H^q(Y_{\sbul\,k'}, 
W\sO_{Y_{\sbul\,k'},\Q}) \ar[dr] & \\
\hspace{0.8cm} H^q(X_k, W\sO_{X_k,\Q})\hspace{-0.8cm}  \ar[ur] \ar[dr]_{f_k^*} & & & &
\hspace{-1.2cm}\bigoplus_i H^q(Y_{\sigma_i,k'}, W\sO_{Y_{\sigma_i,k'},\Q}).\hspace{1.2cm} \\
& H^q(Y_k, W\sO_{Y_k,\Q}) \ar[rr]^-{\sim} & & H^q(Y_{k_1}, W\sO_{Y_{k_1},\Q})
\ar@{^{(}->}[ur] &
} }
In this diagram, the lower horizontal arrow is an isomorphism because $Y_{k_1} \inj
Y_k$ is a nilpotent immersion \cite[Prop.~2.1 (i)]{BBE07}. The lower right arrow is
injective on each summand, because each $\sigma_i$ turns $k'$ into a finite separable
extension of $k_1$, hence it follows from \cite[0, Prop.~1.5.8]{Il79} that
\eqn{ W(k') \otimes_{W(k_1)} \Gamma(U, W\sO_{Y_{k_1}}) \riso \Gamma(U_{\sigma_i}, 
W\sO_{Y_{\sigma_i,k'}}) }
for any affine open subset $U \subset Y_{k_1}$ with inverse image $U_{\sigma_i} \subset
Y_{\sigma_i,k'}$ ; as one can compute Witt cohomology using \v{C}ech cohomology, this
implies that
\eqn{ W(k') \otimes_{W(k_1)} H^q(Y_{k_1}, W\sO_{Y_{k_1}}) \riso H^q(Y_{\sigma_i,k'}, 
W\sO_{Y_{\sigma_i,k'}}). }
Finally, $f : Y \to X$ is a projective alteration between two flat regular schemes of
finite type over $R$, so Theorem \ref{Tworeg} implies that $f_k^*$ is injective.
Therefore, the functoriality map $H^q(X_k, W\sO_{X_k,\Q}) \to \bigoplus_i 
H^q(Y_{\sigma_i,k'},
W\sO_{Y_{\sigma_i,k'},\Q})$ is injective. But \eqref{simplWittvanish} implies that the
composition of the upper path in the diagram is $0$. It follows that $H^q(X_k,
W\sO_{X_k,\Q}) = 0$.
\hfill$\Box$

\section{An injectivity theorem for coherent cohomology}\label{Injectivity1}

We now begin our preliminary work in view of the proof of Theorem \ref{Tworeg}. 

One of the key ingredients in this proof is a theorem which bounds the order of
elements in the kernel of the functoriality map induced on coherent cohomology by a
proper surjective complete intersection morphism $f : Y \to X$ of virtual relative
dimension $0$. Such a result is a consequence of the existence of a ``trace morphism''
$\tau_f : \RR f_*\sO_Y \to \sO_X$ which satisfies the properties stated in the
following theorem:

\begin{thm}\label{Thtau}
Let $X$ be a noetherian scheme with a dualizing complex, and let $f : Y \to X$ be a
proper complete intersection morphism of virtual relative dimension $0$. There exists a
morphism $\tau_f : \RR f_*\sO_Y \to \sO_X$ which satisfies the following properties:

\romain If $g : Z \to Y$ is a second proper complete intersection morphism of 
virtual relative dimension $0$, then the composed morphism
\eq{transtau}{ \RR (f\circ g)_*\sO_Z \cong \RR f_*\RR g_*\sO_Z
\xra{\RR f_*(\tau_g)} \RR f_*\sO_Y \xra{\tau_f} \sO_X }
is equal to $\tau_{fg}$. 

\romain Let $X'$ be another noetherian scheme with a dualizing complex, $u : X' \to X$
a morphism such that $X'$ and $Y$ are Tor-independent over $X$, and $f' : Y' \to X'$
the pull-back of $f$ by $u$. If $f$ is projective, or if either $f$ is flat, or $u$ is 
residually stable \cite[p.~132]{Co00}, then the morphism
\eq{bctau}{\RR f'_*\sO_{Y'} \cong \LL u^* \RR f_*\sO_Y \xra{\LL u^*(\tau_f)} \sO_{X'},}
defined by the base change isomorphism \eqref{basechange}, is equal to $\tau_{f'}$. 

\romain If $f$ is finite and flat, then, for any section $b \in f_*\sO_Y$, 
\eq{tautrace}{\tau_f(b) = \trace_{f_*\sO_Y/\sO_X}(b).} 
\end{thm}

As explained in the introduction, we refer to \ref{Deftau} for the definition of
$\tau_f$, and to \ref{Prooftau} for the proof of the theorem.

It may be worth recalling a few examples of complete intersection morphisms of virtual 
relative dimension $0$ (in short: ci0):

\numero If $X$ and $Y$ are two regular schemes with the same Krull dimension, any 
morphism $f : Y \to X$ which is locally of finite type is ci0. This is the situation 
where we will use Theorem \ref{Thtau} in this article.

\numero If $X$ and $Y$ are smooth over a third scheme $S$, with the same relative 
dimension, any $S$-morphism $Y \to X$ is ci0.

\numero If $X$ is a scheme, $Z \inj X$ a regularly embedded closed subscheme, and $f :
Y \to X$ the blowing up of $X$ along $Z$, then $f$ is ci0 \cite[VII, Proposition
1.8]{SGA 6}.
\medskip

The existence of $\tau_f$ has a remarkable consequence for the functoriality maps 
induced on coherent cohomology.

\begin{thm}\label{Injth}
Let $X$ be a noetherian scheme with a dualizing complex, and $f : Y \to X$ a proper
complete intersection morphism of virtual relative dimension $0$. Assume
that there exists a scheme-theoretically dense open subset $U \subset X$ such that
$f^{-1}(U) \to U$ is finite locally free of constant rank $r \geq 1$. Then, for any
complex $\sE\hbul \in \Dbqc(\sO_X)$ and any $q \geq 0$, the kernel of the functoriality map
\eq{cohcomfunct}{ H^q(X, \sE\hbul) \to H^q(Y, \LL f^*\sE\hbul) }
is annihilated by $r$. In particular, when $r$ is invertible on $X$, the functoriality 
maps are injective.
\end{thm}

\begin{proof}
By \ref{Thtau} (iii), the composition $\sO_X \to \RR f_*\sO_Y \xra{\tau_f} \sO_X$ is
multiplication by $r$ over $U$. Since $U$ is scheme-theoratically dense in $X$, it is
multiplication by $r$ over $X$.

The complete intersection hypothesis implies that $f$ has finite Tor-dimension, hence
$\LL f^*\sE\hbul$ belongs to $\Dbqc(\sO_Y)$. Moreover, we can apply the projection formula
\cite[III, 3.7]{SGA 6} to obtain a commutative diagram
\eqn{ \xymatrix{ 
\entry{\sE\hbul} \ar[r] \ar[rd]_-{\times\; r} & 
\entry{\RR f_*\sO_Y \otimesL_{\sO_X} \sE\hbul}  \ar[r]^-{\sim} \ar[d]^{\tau_f\otimes\Id} &
\entry{\RR f_*\LL f^*\sE\hbul} \ar[ld] \\
& \entry{\sE\hbul} & \hspace{1.5cm},
}}
in which the upper composed morphism is the adjunction morphism. Applying the
functors $H^q(X,-)$ to the diagram, the theorem follows.
\end{proof}

\section{Koszul resolutions and local description of the trace morphism $\tau_f$}

We recall here some well-known explicit constructions based on the Koszul complex which
enter in the definition of the trace morphism $\tau_f$. Later on, this will allow us to
define generalizations of $\tau_f$ for sheaves of Witt vectors. As in the whole 
article, we follow Conrad's constructions and conventions \cite{Co00}.

\subsection{}\label{KoszulToExt}
Let $P$ be a scheme, and let $\mbt = (t_1,\ldots,t_d)$ be a regular sequence of
sections of $\sO_P$, defining an ideal $\sI \subset \sO_P$. We denote by $Y \subset P$
the closed subscheme defined by $\sI$, and by $i : Y \inj P$ the corresponding closed
immersion. Classically, the Koszul complex $K\lbul(\mbt)$ defined by the sequence
$(t_1,\ldots,t_d)$ is the chain complex concentrated in homological degrees $[0,d]$,
such that $\sE := K_1(\mbt)$ is a free $\sO_P$-module of rank $d$ with basis
$e_1,\ldots,e_d$, $K_k(\mbt) = \Wedge^k \sE$ for all $k$, and such that the
differential is given in degree $k$ by
\eqn{ d_k(e_{i_1}\wedge \ldots \wedge e_{i_k}) = \sum_{j=1}^k
(-1)^{j-1}t_{i_j}e_{i_1}\wedge \ldots \wedge \widehat{e_{i_j}} \wedge \ldots \wedge
e_{i_k}. }
It is often more convenient to consider $K\lbul(\mbt)$ as a cochain complex
concentrated in cohomological degrees $[-d,0]$, by setting $(K\lbul(\mbt))^k =
K_{-k}(\mbt)$ and leaving the differential unchanged \cite[p.~17]{Co00}.

Since $\mbt$ is a regular sequence, $K\lbul(\mbt)$ is a free resolution of $\sO_Y$ 
over $\sO_P$. For any $\sO_P$-module $\sM$, this resolution provides an isomorphism
\eq{psit}{ \sExt^d_{\sO_P}(\sO_Y,\sM) := H^d(\sHom\hbul_{\sO_P}(K\lbul(\mbt),\sM))
\xrightarrow[\sim]{\psi_{\mbt,\sM}} \frac{\sHom_{\sO_P}(\Wedge^d \sE,\sM )}
{\sI\sHom_{\sO_P}(\Wedge^d\sE,\sM)}, }
where $\psi_{\mbt,\sM}$ is the tautological isomorphism multiplied by $(-1)^{d(d+1)/2}$ 
(see \cite[definition of (1.3.28) and (2.5.2)]{Co00}). For any section $m$ of $\sM$, 
we will denote by
\eq{defSymbol}{ \left[\begin{array}{c} m \\ t_1,\ldots,t_d \end{array}\right] \in 
\sExt^d_{\sO_P}(\sO_Y,\sM)
 }
the section corresponding by \eqref{psit} to the class of the homomorphism $u_{\mbt,m}$
which sends $e_1\wedge\ldots\wedge e_d$ to $(-1)^d m$ (the $(-1)^d$ sign being 
needed to obtain relation \eqref{deltaFLI} later). Note that this section is linear with
respect to $m$, only depends on the class of $m$ mod $\sI\sM$, and is functorial with
respect to $\sM$. Its dependence on the regular sequence $\mbt$ is given by the
following lemma.

\begin{lem}\label{Changet}
Let $\mbt' = (t'_1,\ldots,t'_d)$ be another regular sequence of sections of $\sO_P$, 
generating an ideal $\sI'$ such that $\sI' \subset \sI$. Let $C = (c_{i,j})_{1\leq i,j 
\leq d}$ be a matrix with entries in $\sO_P$ such that $t'_i = \sum_{j=1}^d 
c_{i,j}t_j$ for all $i$. If $\alpha : \sExt^d_{\sO_P}(\sO_P/\sI,\sM) \to 
\sExt^d_{\sO_P}(\sO_P/\sI',\sM)$ is the functoriality homomorphism, then
\eq{changet}{ \alpha(\left[\begin{array}{c} m \\ t_1,\ldots,t_d\end{array}\right]) = 
\left[\begin{array}{c} \det(C)m \\ t'_1,\ldots,t'_d \end{array}\right]. }
\end{lem}

\begin{proof}
Let $K\lbul(\mbt')$ be the Koszul resolution of $\sO_P/\sI'$, and $\sE' = K_1(\mbt')$, with 
basis $e'_1,\ldots,e'_d$. One defines a morphism of resolutions $\phi : K\lbul(\mbt') 
\to K\lbul(\mbt)$ by setting $\phi_1(e'_i) = \sum_j c_{i,j}e_j$, and $\phi_k = 
\wedge^k\phi_1$ for $0 \leq k \leq d$. Then $\phi$ provides a commutative diagram 
\eqn{
\xymatrix{
\sHom_{\sO_P}(\wedge^d \sE, \sM) \ar[r] \ar[d]_{\phi_d=\det(C)} & 
\sExt^d_{\sO_P}(\sO_P/\sI, \sM) \ar[d]^{\alpha}  \\
\sHom_{\sO_P}(\wedge^d \sE', \sM) \ar[r] & \sExt^d_{\sO_P}(\sO_P/\sI', \sM)  .}
}
The lemma follows.
\end{proof}

\subsection{}\label{Glueing}
Under the assumptions of \ref{KoszulToExt}, the morphism $d_1 : \sE \surj \sI$ defines 
an isomorphism $\sE/\sI\sE \riso \sI/\sI^2$. Using the canonical isomorphisms, this 
provides 
\eqa{dualwedge}{ \frac{\sHom_{\sO_P}(\Wedge^d \sE,\sM )}
{\sI\sHom_{\sO_P}(\Wedge^d \sE,\sM)} & \riso & 
(\Wedge^d \sE)^{\vee} / \sI(\Wedge^d \sE)^{\vee} \otimes_{\sO_Y}\sM/\sI\sM \\
 & \riso & \Wedge^d((\sE / \sI\sE)^{\vee}) \otimes_{\sO_Y}\sM/\sI\sM \notag \\
& \riso & \omega_{Y/P} \otimes_{\sO_Y} i^*\sM. \notag
}
Note that, due to the commutation between dual and exterior power, the composition
\eqref{dualwedge} maps the class of $u_{\mbt,m}$ to $(-1)^d(\bt_d^{\,\vee}\wedge\ldots\wedge
\bt_1^{\,\vee})\otimes i^*(m)$, where $\bt_k$ denotes the class of $t_k$ mod $\sI^2$.

Composing \eqref{psit} and \eqref{dualwedge}, one obtains the \textit{fundamental local 
isomorphism} \cite[III, 7.2]{Ha66} as defined by Conrad \cite[(2.5.2)]{Co00} in the 
local case: 
\eq{fundlociso}{ \eta_{Y/P} : \sExt^d_{\sO_P}(\sO_Y, \sM) \riso 
\omega_{Y/P}\otimes_{\sO_Y}i^*\sM. }
Applying Lemma \ref{Changet} to the case of two regular sequences of generators of the
ideal $\sI$, one sees that the isomorphism $\eta_{Y/P}$ does not depend on the sequence
$\mbt$, so that local constructions can be glued to define $\eta_{Y/P}$ for any regular
immersion $i : Y \inj P$, without assuming that $\sI$ is defined globally by a regular
sequence. One obtains in this way the fundamental local isomorphism in the general 
case \cite[(2.5.1)]{Co00}. 

Let us now recall from \cite[2.5]{Co00} how the isomorphism \eqref{fundlociso} allows 
to define functorially for any $\sM\hbul \in D(\sO_P)$ the isomorphism \cite[(2.5.3)]{Co00}
\eq{derfundlociso}{ \eta_i : \RR\sHom_{\sO_P}(\sO_Y, \sM\hbul) \riso \omega_{Y/P}[-d] 
\otimesL_{\sO_P} \LL i^*(\sM\hbul). }
Applying \cite[Lemma 2.1.1]{Co00} and using the isomorphism of
functors defined by $\eta_{Y/P}$, one gets the isomorphism
\eqn{ \RR\sHom_{\sO_P}(\sO_Y, \sM\hbul) \riso (\omega_{Y/P} \otimesL_{\sO_P} 
\LL i^*(\sM\hbul))[-d]. }
The isomorphism $\eta_i$ is then obtained by composition with the canonical 
isomorphism 
\eqn{ (\omega_{Y/P} \otimesL_{\sO_P} \LL i^*(\sM\hbul))[-d] \riso \omega_{Y/P}[-d] 
\otimesL_{\sO_P} \LL i^*(\sM\hbul) }
defined by the general convention \cite[p.~11]{Co00}. From the discussion p.~53 of 
\cite{Co00}, it follows that $\eta_i$ satisfies the following properties:

\alphab If $\sM\hbul = \sM[0]$ for an $\sO_P$-module $\sM$, then the homomorphism 
induced by $\eta_i$ between the cohomology sheaves in degree $d$ is the isomorphism 
$\eta_{Y/P}$. 

\alphab The isomorphism $\eta_i$ commutes with translations in $D(\sO_P)$ (using the 
general convention \cite[(1.3.6)]{Co00} for the right hand side of 
\eqref{derfundlociso}).

\subsection{}\label{Defgammaf}
Let $\pi : P \to X$ be a smooth morphism of relative dimension $d$, $i : Y \inj P$ a
regular immersion of codimension $d$, and $f = \pi \circ i$. Let 
\eqn{ \zeta'_{i,\pi} :
\omega_{Y/X}\riso \omega_{Y/P}[-d] \otimesL_{\sO_Y} \LL i^*(\omega_{P/X}[d]) }
be the canonical isomorphism \eqref{defzeta}, which induces in degree $0$ the
tautological isomorphism $\zeta'_{i,\pi}$ provided in \eqref{transomega1} by the
construction of $\omega_{Y/X}$ in \ref{Defomega}. Let $\delta_f$ be the canonical
section of $\omega_{Y/X}$ (defined by \eqref{locdeltaf}), and $\varphi_f : \sO_Y \to 
\omega_{Y/X}$ the morphism sending $1$ to $\delta_f$. We define the morphism
\eq{defgammaf}{ \gamma_{f} : \sO_Y \to \omega_{P/X}[d] }
as being the composition 
\eqn{ \xymatrix{
\entry{\sO_Y} \ar[r]^-{\varphi_f} \ar@/_1pc/[rrrd]_-{\gamma_f} & 
\entry{\omega_{Y/X}} \ar[r]^-{\zeta'_{i,\pi}}_-{\sim} &
\entry{\omega_{Y/P}[-d] \otimesL_{\sO_Y}\LL i^*(\omega_{P/X}[d])} 
\ar[r]^-{\eta_i^{-1}}_-{\sim} & 
\entry{\RR\sHom_{\sO_P}(\sO_Y, \omega_{P/X}[d])} \ar[d]_-{\can} \\
 & & & \entry{\omega_{P/X}[d].} 
} }

\begin{prop}\label{DeltaFLI}
Under the assumptions of \ref{Defgammaf}, the isomorphism $\eta_i^{-1} \circ 
\zeta'_{i,\pi}$ entering in the definition of $\gamma_f$ induces in degree $0$ an 
isomorphism
\eqn{ \eta_i^{-1} \circ \zeta'_{i,\pi} : \omega_{Y/X} \riso \sExt^d_{\sO_P}(\sO_Y, 
\omega_{P/X}), }
which is such that 
\eq{deltaFLI}{ \eta_i^{-1} \circ \zeta'_{i,\pi}(\delta_f) = 
\left[\begin{array}{c} dt_1\wedge\cdots\wedge dt_d \\ t_1,\ldots,t_d \end{array}\right]. }
\end{prop}

\begin{proof}
Let us consider first the isomorphism 
\eqn{ \eta_{Y/P}^{-1} \circ \zeta'_{i,\pi} : \omega_{Y/X} \riso \sExt^d_{\sO_P}(\sO_Y, 
\omega_{P/X}), }
where $\eta_{Y/P}$ is defined by \eqref{fundlociso}. By definition,
$\left[\begin{array}{c} dt_1\wedge\cdots\wedge dt_d \\ t_1,\ldots,t_d
\end{array}\right]$ is mapped to $u_{\mbt,dt_1\wedge\ldots\wedge dt_d}$ by
\eqref{psit}, and we observed in \ref{Glueing} that $u_{\mbt,dt_1\wedge\ldots\wedge
dt_d}$ is mapped to $(-1)^d(\bt_d^{\,\vee}\wedge\ldots\wedge \bt_1^{\,\vee})\otimes
i^*(dt_1\wedge\ldots\wedge dt_d)$ by \eqref{dualwedge}. Since $\zeta'_{i,\pi}(\delta_f)
= (\bt_1^{\,\vee}\wedge\ldots\wedge \bt_d^{\,\vee})\otimes i^*(dt_d\wedge\ldots\wedge
dt_1)$ by construction, we get the relation
\eq{deltaFLI0}{ \eta_{Y/P}^{-1} \circ \zeta'_{i,\pi}(\delta_f) = (-1)^d
\left[\begin{array}{c} dt_1\wedge\cdots\wedge dt_d \\ t_1,\ldots,t_d \end{array}\right]. }

Let $\eta_{i,\omega_{P/X}}$ and $\eta_{i,\omega_{P/X}[d]}$ be the isomorphisms 
\eqref{derfundlociso} relative to the complexes $\omega_{P/X}[0]$ and 
$\omega_{P/X}[d]$. By \ref{Glueing} a), $\eta_{i,\omega_{P/X}}$ induces in degree $d$ 
the isomorphism $\eta_{Y/P}$. On the other hand, \ref{Glueing} b) shows that 
$\eta_{i,\omega_{P/X}[d]}$ is identified with $\eta_{i,\omega_{P/X}}[d]$ when using the
canonical isomorphisms $\RR \sHom(\sO_Y, \omega_{P/X}[d]) \riso \RR
\sHom(\sO_Y, \omega_{P/X})[d]$ and $\omega_{Y/P}[-d]\otimesL \LL i^*(\omega_{P/X}[d]) 
\riso (\omega_{Y/P}[-d]\otimesL \LL i^*(\omega_{P/X}))[d]$. The first one involves no 
sign, and, as $\omega_{Y/P}[-d]$ is concentrated in degree $d$, the second one is given
by multiplication by $(-1)^{d^2} = (-1)^d$ on $\omega_{Y/P} \otimes 
i^*(\omega_{P/X})$. Thus relation \eqref{deltaFLI0} implies relation \eqref{deltaFLI}. 
\end{proof}

\begin{prop}\label{Taugamma}
Let $X$ be a separated noetherian scheme with a dualizing complex, $P = \P^{\,d}_X$ a
projective space over $X$, $\pi : P \to X$ the structural morphism, $i : Y \inj P$ a
regular immersion of codimension $d$, and $f = \pi \circ i$. Then the trace morphism
$\tau_f : \RR f_*\sO_Y \to \sO_X$ of Theorem \ref{Thtau} is equal to the composition
\eq{taugamma}{ \RR f_*(\sO_Y) \xra{\RR \pi_*(\gamma_f)} \RR \pi_*(\omega_{P/X}[d]) 
\xra{\Trp_{\pi}} \sO_X, }
where $\Trp_{\pi}$ is the trace morphism for the projective space defined in 
\cite[(2.3.1)-(2.3.5)]{Co00}.
\end{prop}

\begin{proof}
By construction (see \ref{Deftau}), $\tau_f$ is the composition $\Tr_f \circ \RR
f_*(\lambda_f) \circ \RR f_*(\varphi_f)$ in the commutative diagram
\eqn{\xymatrix@C=40pt{
\RR f_*(\sO_Y) \ar[d]^-{\RR f_*(\varphi_f)} \\ 
\RR f_*(\omega_{Y/X}) 
\ar[d]_-{\RR f_*(\zeta'_{i,\pi})}^-{\wr} \ar[r]^-{\RR f_*(\lambda_f)}_-{\sim} & 
\RR f_*(f^!(\sO_X)) \ar[r]^-{\Tr_f} & \sO_X \\
\entry{\RR f_*(\omega_{Y/P}[-d]\otimesl \LL i^*(\omega_{P/X}[d]))} 
\ar[d]_-{\RR f_*(\eta_i^{-1})}^-{\wr} & 
\entry{\RR f_*(i^!\pi^!(\sO_X))} \ar[u]_-{\wr}^-{\RR f_*(c_{i,\pi}^{-1})} 
\ar[r]^-{\RR\pi_*(\Tr_i)} & \entry{\RR\pi_*(\pi^!(\sO_X))} \ar[u]^{\Tr_\pi} \\
\RR f_*(\RR\sHom_{\sO_P}(\sO_Y,\omega_{P/X}[d])) \ar[r]_-{\sim}^-{\RR f_*(d_i)} & 
\RR f_*(i^!(\omega_{P/X}[d])) \ar[u]_-{\wr}^-{\RR f_*(i^!(e_\pi))} \ar[r]^-{\RR\pi_*(\Tr_i)}
 & \RR\pi_*(\omega_{P/X}[d]) \ar[u]^-{\RR\pi_*(e_{\pi})}_-{\wr},
} }
in which the isomorphism $\lambda_f$ is defined by the commutativity of the left
rectangle before applying $\RR f_*$ (cf.~\ref{Deflambda}), and $d_i$, $e_\pi$,
$c_{i,\pi}$ are defined as follows:

\alphab $d_i$ is the canonical isomorphism of functors $i^\flat := \RR\sHom_{\sO_P}(\sO_Y,-) 
\riso i^!$, defined by \cite[(3.3.19)]{Co00};

\alphab $e_\pi$ is the canonical isomorphism of functors 
$\pi^\sharp := \omega_{P/X}[d]\otimesl\pi^*(-) \riso \pi^!$, defined by \cite[(3.3.21)]{Co00}.

\alphab $c_{i,\pi}$ is the transitivity isomorphism $f^! \riso i^!\pi^!$, defined by 
\cite[(3.3.14)]{Co00}.

\noindent Moreover, the upper right square commutes because of the transitivity of the trace 
morphism \cite[3.4.3, (TRA1)]{Co00}, and the lower right square commutes by 
functoriality of the trace morphism $\Tr_i$ with respect to $e_{\pi}$. 

In this diagram, the composition of the right vertical arrows is the projective trace 
morphism $\Trp_{\pi}$ \cite[3.4.3, (TRA3)]{Co00}, and the isomorphism $d_i$ on the 
bottom row identifies $\Tr_i$ with the trace morphism $\Trf_i$ for finite morphisms 
\cite[3.4.3, (TRA2)]{Co00}. As the latter is the canonical morphism 
$i_*\RR\sHom_{\sO_P}(\sO_Y,-) \to \Id$ defined by $\sO_P \surj \sO_Y$, it follows that 
the composition of the left column and the bottom row of the diagram is equal to 
$\RR\pi_*(\gamma_f)$, which proves the proposition.
\end{proof}

\section{Preliminaries on the relative de Rham-Witt complex}\label{DRWprel}

We extend here to the relative de Rham-Witt complex constructed by Langer and Zink
\cite{LZ04} structure theorems which are classical when the base is a perfect scheme of
characteristic $p$ (\cite{Il79}, \cite{IR83}). We begin by recalling some basic facts
from their construction.

From now on, we fix a prime number $p$. We denote by $\Z_{(p)}$ the localization of
$\Z$ at the prime ideal $(p)$. Although many results of \cite{LZ04} are valid for
$\Z_{(p)}$-schemes, we limit our exposition to the case of schemes on which $p$ is
locally nilpotent, which will suffice for our applications.

\subsection{}\label{defDRW}
Let $S$ be a scheme on which $p$ is locally nilpotent, and let $f:X\to S$ be a morphism
of schemes. An {\em \FV-pro-complex of $X/S$} as defined in \cite{LZ04} is a
pro-complex $\{R: E\hbul_{n+1}\to E\hbul_n \}_{n \geq 1}$ of sheaves on $X$, where
$E\hbul_n$ is a differential graded $W_n(\sO_X)/f^{-1}W_n(\sO_S)$-algebra (i.e.,
$E\hbul_n$ is a graded $W_n(\sO_X)$-algebra together with an $f^{-1}W_n(\sO_S)$-linear
map $d: E\hbul_n \to E\hbul_n(1)$ such that $d^2=0$, satisfying $\eta\omega =
(-1)^{\deg\omega\deg\eta}\omega\eta$ and $d(\omega\eta)= (d\omega)\eta+
(-1)^{\deg\omega}\omega d\eta$ for any homogeneous sections $\omega, \eta \in E\hbul_n$),
which is equipped with a map of graded pro-rings
\eqn{ F: E\hbul\lbul{}_{+1} \to E\hbul\lbul, }
called the Frobenius morphism, and with a map of graded abelian groups
\eqn{ V: E\hbul\lbul \to E\hbul\lbul{}_{+1}, }
called the Verschiebung morphism, such that the following properties hold: 

\romain The structure map $W\lbul(\sO_X)\to E^0\lbul$ is compatible with $F$ and $V$.

\romain The following relations hold:
\ga{}{ FV=p, \quad FdV=d,\label{FV} \\
V(\omega F(\eta))=V(\omega)\eta, \quad \text{for all } \omega\in E\hbul_n, \eta\in
E\hbul_{n+1}, n \geq 1,\label{VFlin} \\
F(d[a])= [a]^{p-1}d[a], \quad \text{for all } a\in \sO_X,\label{FTeich}}
where $[a]$ denotes the Teichm\"uller lift of $a$ to $W_{n}(\sO_X)$, for any $n$.

A morphism between two \FV-pro-complexes of $X/S$ is a map of pro-differential graded
$W\lbul(\sO_X)/f^{-1}W\lbul(\sO_S)$-algebras compatible with $F$ and $V$. By
\cite[Prop. 1.6, Rem.~1.10]{LZ04} there exists an initial object in the category of
\FV-pro-complexes of $X/S$, which is called the {\em relative de Rham-Witt complex of
$X/S$} and is denoted by $\{R: \WC{n+1}{\sbul}{X/S}\to \WC{n}{\sbul}{X/S}\}_{n\geq 1}$.
Each sheaf $\WC{n}{q}{X/S}$ is a quasi-coherent sheaf on the scheme $W_n(X) := (|X|,
W_n(\sO_X))$ defined in \ref{Wittscheme}, and the transition morphisms $R$ are
epimorphisms. When $S$ is a perfect scheme of characteristic $p$, the relative de
Rham-Witt complex coincides with the one defined in \cite{Il79}. Notice that we have
the following properties:
\eqn{ \WC{n}{0}{X/S}= W_n(\sO_X), \quad \WC{1}{\sbul}{X/S}=\Omega^{\sbul}_{X/S}, }
and that, by \cite[(1.16), (1.17) and (1.19)]{LZ04}, relations \eqref{FV} and \eqref{VFlin}
imply that
\ga{}{ V(\omega d\eta)=V(\omega)dV(\eta), \quad \text{for all } \omega, \eta \in 
\WC{n}{\sbul}{X/S}, 
n \geq 1,\label{Vxdy} \\
Vd = pdV, \quad dF = pFd.\label{Vd} }
In addition, when $S$ is an $\FF_p$-scheme, the operators $F$ and $V$ satisfy the
relation $VF = p$.

We also recall the behaviour of the de Rham-Witt complex with respect to \'etale 
pull-backs. Let 
\eqn{ \xymatrix{X'\ar[r]^h\ar[d] & X\ar[d] \\ S'\ar[r]^{g} & S} }
be a commutative diagram in which $h$ is \'etale and $g$ unramified. Then, for all
$q\ge 0$ and $r\ge n\ge 1$, $W_n(X')$ is \'etale over $W_n(X)$ and we have the 
$W_r(\sO_{X'})$-linear isomorphisms
\ga{etlocDRW1}{W_{r}(\sO_{X'}) \otimes_{h^{-1}W_r(\sO_{X})} h^{-1}\WC{n}{q}{X/S} \riso 
\WC{n}{q}{X'/S'},\\
W_{r}(\sO_{X'})\otimes_{h^{-1}W_r(\sO_{X})}
h^{-1}(F^{r-n}_*\WC{n}{q}{X/S}) \riso F^{r-n}_*\WC{n}{q}{X'/S'}, \quad
a\otimes\omega\mapsto F^{r-n}(a)\omega, \label{etlocDRW2}}
where, for any $W_n$-module $M$, $F^{r-n}_*M$ denotes $M$ viewed as a $W_r$-module via
$F^{r-n}: W_r\to W_{n}$ \cite[Prop.~1.11, Prop.~A.8 and Cor.~A.11]{LZ04}.

Finally, the completed relative de Rham-Witt complex is defined by $\WC{}{\sbul}{X/S} := 
\varprojlim_n \WC{n}{\sbul}{X/S}$; the canonical morphisms $\WC{}{\sbul}{X/S} \to 
\WC{n}{\sbul}{X/S}$ are still epimorphisms.

\subsection{}\label{affDRW}
Let $S=\Spec A$ be affine. We want to recall the calculation of
$\WC{}{q}{A[x_1,\ldots, x_d]/A} := \Gamma(\A^d_S,\WC{}{q}{\A^d_S/S})$. We need some
notations for this.

A {\em weight} is a function $k: [1,d]=\{1,2,\ldots, d\}\to \Z[\frac{1}{p}]_{\ge 0}$.
We write $k_i:=k(i)$, for $i\in [1,d]$. The support of $k$, ${\rm supp\,} k$, consists
of those $i\in [1,d]$ with $k_i\neq 0$. For any weight $k$ we choose once and for all a
total ordering on the elements of the support of $k$,
\eq{supp}{{\rm supp\,}k=\{i_1,\ldots,i_r\},}
such that: 

\romain $\ord_p\,k_{i_1}\le\ord_p\,k_{i_2}\le\cdots\le \ord_p\,k_{i_r}$.

\romain The ordering on ${\rm supp\,}k$ and on ${\rm supp\,} p^a k$ agree, for any $a\in \Z$.  

We say $k$ is {\em integral} if $k_i\in \Z$, for all $i\in [1,d]$. We say $k$ is {\em
primitive} if it is integral and not all $k_i$ are divisible by $p$. We set
\eq{deftofk}{t(k_i):=-\ord_p\,k_i \quad \text{and}\quad t(k):=
\begin{cases}\max\,\{\,t(k_i)\,|\,i\in {\rm supp\,}k\,\} & 
\text{if } {\rm supp\,}k\neq\emptyset, \\
0 & \text{if } k = 0.
\end{cases}} 
If $k\neq 0$, $t(k)$ is the smallest integer such that $p^{t(k)}k$ is primitive, and we have
\eqn{ t(k)=t(k_{i_1})\ge t(k_{i_2})\ge\cdots\ge t(k_{i_r}). }
We denote by $u(k)$ the smallest non-negative integer such that $p^{u(k)}k$ is
integral, i.e., $u(k)=\max\,\{0, t(k)\}$. Notice that $k$ is integral iff $u(k)=0$ iff
$t(k)\le 0$, and $k$ is primitive iff $t(k)=0$. An {\em interval of the support of $k$}
is by definition a subset $I\subset {\rm supp\,}k$ of the form
\eqn{ I=\{i_s,i_{s+1},\ldots, i_{s+m}\}. }
We denote by $k_I$ the weight which equals $k$ on $I$ and is zero on $[1,d]\setminus
I$. If $k$ is fixed and $I$ is an interval of the support of $k$, we write
$u(I):=u(k_I)$ and $t(I):=t(k_I)$. An {\em admissible partition $\sP$ of length $q$ of
${\rm supp\,} k$} (or just {\em of} $k$) is a tuple of intervals of ${\rm supp\,}k$,
$\sP=(I_0,I_1,\ldots, I_q)$, such that: 

\reset{romain}\romain ${\rm supp\,}k= I_0\sqcup I_1\sqcup \ldots\sqcup I_q.$

\romain The elements in $I_j$ are smaller than the elements in $I_{j+1}$ (with respect
to the ordering \eqref{supp}) for all $j=0,\ldots, q-1$.

\romain The intervals $I_1,\ldots, I_q$ are non-empty (but $I_0$ may be).

\noindent Notice that $u(k)=u(I_0)$ if $I_0\neq\emptyset$ and $u(k)=u(I_1)$ if
$I_0=\emptyset$.

For any $n \leq \infty$, we write $X_i:=[x_i]\in W_n(A[x_1,\ldots,x_d])$. If $k$ is an
integral weight as above, we write $X^k= X_{i_1}^{k_{i_1}}\cdots X_{i_r}^{k_{i_r}}\in
W_n(A[x_1,\ldots,x_d])$.

Let $k$ be any weight and $\eta\in W(A)$. We define 
\eq{defe0}{e^0(\eta,k):= V^{u(k)}(\eta X^{p^{u(k)}k})\in W(A[x_1,\ldots x_d]) }
and 
\eq{defe1}{e^1(\eta,k):=\begin{cases}
		      dV^{u(k)}(\eta X^{p^{u(k)}k}) & \text{if $k$ is not integral}\\
		      \eta F^{-t(k)}d X^{p^{t(k)}k} & \text{if $k$ is integral}
		      \end{cases}\ \ \in \WC{}{1}{A[x_1,\ldots,x_d]/A}.}

\begin{defn}[{Basic Witt differentials \cite[2.2]{LZ04}}]\label{basicWdif}
Let $k$ be a weight, $\sP=(I_0,I_1,\ldots, I_q)$ an admissible partition of $k$, and
$\xi=V^{u(k)}(\eta)\in W(A)$. The basic Witt differential $e(\xi, k, \sP)\in
\WC{}{q}{A[x_1,\ldots,x_d]/A}$ is defined as follows:
\eqn{e(\xi,k,\sP):=\begin{cases} e^0(\eta, k_{I_0})e^1(1,k_{I_1})\cdots e^1(1, k_{I_q})
& \text{if $I_0\neq \emptyset$,}\\
e^1(\eta, k_{I_1})e^1(1,k_{I_2})\cdots e^1(1,k_{I_q})  & \text{if $I_0=\emptyset$.} 
\end{cases} }
\end{defn}

\begin{rules}[{\cite[Prop.~2.5, Prop.~2.6]{LZ04}}]\label{rules}
Let $k$ be a weight, $\sP=(I_0,I_1,\ldots, I_q)$ a partition of $k$ and
$\xi = V^{u(k)}(\eta)\in W(A)$. Note that $u(k)\geq 1$ when $k$ is not integral, so 
that one can then define $V^{-1}\xi := V^{u(k)-1}\eta$. Then:
\vspace{2mm}

\reset{romain}\romain $\rho e(\xi, k,\sP)=e(\rho\xi, k, \sP)$ for all $\rho\in W(A)$.
\vspace{2mm}

\romain $F e(\xi, k, \sP)=\begin{cases} 
e(F\xi, pk, \sP) & \text{if } I_0\neq\emptyset \text{ or } k \text{ integral,}\\
e(V^{-1}\xi, pk,\sP) & \text{if } I_0=\emptyset \text{ and } k \text{ not integral}
			 \end{cases}$\\ 
\vspace{2mm}

\romain $V e(\xi, k, \sP)=\begin{cases} 
e(V\xi, \frac{1}{p}k, \sP) & \text{if } I_0\neq \emptyset \text{ or } \frac{1}{p}k
\text{ integral,}\\
e(pV\xi, \frac{1}{p}k,\sP) &\text{if } I_0=\emptyset \text{ and } \frac{1}{p}k \text{
not integral.}
			 \end{cases}$
\vspace{2mm}

\romain $d e(\xi,k, \sP)=\begin{cases} 
0 & \text{if } I_0=\emptyset,\\
e(\xi, k, (\emptyset,\sP)) & \text{if } I_0\neq\emptyset\text{ and } k \text{ not integral,}\\
p^{-t(k)} e(\xi, k, (\emptyset,\sP)) & \text{if } I_0\neq\emptyset\text{ and } k \text{
integral.}
		       \end{cases}$
\end{rules}

\begin{thm}[{\cite[Thm.~2.8]{LZ04}}]\label{candecomp}
Every $\omega\in \WC{}{q}{A[x_1,\ldots, x_d]/A}$ can uniquely be written as 
\eqn{ \omega=\sum_{k,\sP} e(\xi_{k,\sP},k,\sP), }
where the sum is over all weights $k$ with $|{\rm supp\,}k|\ge q$ and over all
admissible partitions of length $q$ of $k$, and the sum converges in the sense that, for
any $m\ge 0$, we have $\xi_{k,\sP}\in V^m W(A)$ for all but finitely many
$\xi_{k,\sP}$.
\end{thm}

For a weight $k$, $n\ge 1$ and $\eta\in W_{n-u(k)}(A)$ we define $e^0_n(\eta, k)\in
W_n(A[x_1,\ldots,x_d])$ and $e^1_n(\eta, k)\in \WC{n}{1}{A[x_1,\ldots,x_d]/A}$ by the
same formulas as in \eqref{defe0} and \eqref{defe1}. For $\sP$ an admissible partition
of length $q$ of $k$ and $\xi=V^{u(k)}(\eta)\in W_n(A)$, we then define
$e_n(\xi,k,\sP)\in \WC{n}{q}{A[x_1,\ldots,x_d]/A}$ by the same formula as in Definition
\ref{basicWdif} but with $e^i$ replaced by $e^i_n$, $i=0,1$.

\begin{cor}[{\cite[Prop.~2.17]{LZ04}}]\label{candecompn}
Every $\omega\in \WC{n}{q}{A[x_1,\ldots,x_d]/A}$ may uniquely be written as a finite sum
\eqn{ \omega=\sum_{k,\sP} e_n(\xi_{k,\sP}, k, \sP), \quad 
\xi_{k,\sP}\in V^{u(k)}W_{n-u(k)}(A), }
where the sum is over all weights $k$ with $|{\rm supp\,}k|\ge q$ and such that
$p^{n-1}k$ is integral and over all admissible partitions $\sP$ of $k$ of length $q$.
\end{cor}

We now assume that $S$ is an $\FF_p$-scheme; the absolute Frobenius endomorphism of $S$
will then be denoted $F_S$, or $F$ when no confusion can arise. The following
proposition is known if $S$ is perfect (see \cite[II, (1.2.2)]{IR83}); a similar 
result has been proved by M.~Olsson when, \'etale locally, $S$ has a flat lifting over 
$\Z_p$ to which $F_S$ can be lifted \cite[Th.~4.2.15]{Ol07}.

\begin{prop}\label{KerR}
Let $S$ be a locally noetherian $\FF_p$-scheme and $X$ a smooth $S$-scheme. Then
the sequence 
\begin{multline*} 
F^n_* \sO_S\otimes_{W_{n+1}(\sO_S)}\WC{n+1}{q-1}{X/S} 
\xra{(1\otimes F^n, -1\otimes F^n d)} 
F^n_*\Omega^{q-1}_{X/S}\oplus F^n_*\Omega^q_{X/S}\\
\xra{dV^n+V^n} \WC{n+1}{q}{X/S}\lra R_*\WC{n}{q}{X/S}\lra 0
\end{multline*}
is an exact sequence of $W_{n+1}(\sO_S)$-modules.
\end{prop}

\begin{proof}
The question is local, we thus assume $S=\Spec A$, $X=\Spec B$ and $B$ is \'etale over
$B_1=A[x_1,\ldots, x_d]$. As $\WC{}{\sbul}{X/S} \to \WC{n+1}{\sbul}{X/S}$ is an
epimorphism, \cite[Prop.~2.19]{LZ04} provides the exactness of the second line, and we
only have to show that
\eqn{ (*_{B/A}):\ \ F^n_*A\otimes\WC{n+1}{q-1}{B/A}\xra{(1\otimes F^n,
-1\otimes F^n d)} F^n_*\Omega^{q-1}_{B/A}\oplus
F^n_*\Omega^q_{B/A} \xra{dV^n+V^n} \WC{n+1}{q}{B/A} }
is exact. Notice that it is a complex, as for $a\in A$ and $\omega\in
\WC{n+1}{q-1}{B/A}$ we have
\[dV^n(aF^n\omega)-V^n(aF^nd\omega)=0.\] 
Notice also that, if we let $W_{2n+2}(B)$ act through $F^{n+1}:W_{2n+2}(B) \to
W_{n+1}(B)$, the differentials of this complex are $W_{2n+2}(B)$-linear, since 
$dF^{n+1}=p^{n+1}F^{n+1}d=0$ in $W_{n+1}$. We claim
\eq{starBA}{(*_{B/A})= F^{n+1}_*(*_{B_1/A})\otimes_{W_{2n+2}(B_1)} W_{2n+2}(B).}
Indeed we have the following diagrams (where the tensor products with $W_{2n+2}(B)$ 
are taken over $W_{2n+2}(B_1)$):
\eqn{\xymatrix@C=57pt{ 
F^{n+1}_*(F^n_*A\otimes\WC{n+1}{q-1}{B/A}) \ar[r]^-{1\otimes F^n} &
F^{2n+1}_*\Omega^{q-1}_{B/A} \\
F^{n+1}_*(F^n_*A\otimes\WC{n+1}{q-1}{B_1/A})\otimes
W_{2n+2}(B)\ar[r]^-{(1\otimes F^n)\otimes 1}\ar[u] & 
F^{n+1}_*(F^n_*\Omega^{q-1}_{B_1/A})\otimes W_{2n+2}(B),\ar[u]
} }
\eqn{\xymatrix@C=57pt{ 
F^{n+1}_*(F^n_*A\otimes\WC{n+1}{q-1}{B/A}) \ar[r]^-{-1\otimes F^n d} & 
F^{2n+1}_*\Omega^{q}_{B/A} \\
F^{n+1}_*(F^n_*A\otimes\WC{n+1}{q-1}{B_1/A}) \otimes
W_{2n+2}(B)\ar[r]^-{(-1\otimes F^nd)\otimes 1 }\ar[u] & 
F^{n+1}_*(F^n_*\Omega^{q}_{B_1/A})\otimes W_{2n+2}(B),\ar[u]
} }
both with vertical maps
\eqn{ (a\otimes\omega)\otimes b \mapsto a\otimes F^{n+1}(b)\omega, \hspace{1cm} 
\eta\otimes b \mapsto F^{2n+1}(b)\eta, }
and
\eqn{\xymatrix@C=43.8pt{
F^{2n+1}_*\Omega^{q-1}_{B/A}\oplus F^{2n+1}_*\Omega^q_{B/A} \ar[r]^{dV^n+V^n} & 
F^{n+1}_*\WC{n+1}{q}{B/A}\\
F^{n+1}_*(F^n_*\Omega^{q-1}_{B_1/A}\oplus
F^n_*\Omega^q_{B_1/A})\otimes W_{2n+2}(B)\ar[r]^-{(dV^n+V^n)\otimes 1}\ar[u] &
F^{n+1}_*\WC{n+1}{q}{B_1/A}\otimes W_{2n+2}(B),\ar[u]
} }
with vertical maps
\eqn{ (\eta,\omega)\otimes b \mapsto (F^{2n+1}(b)\eta,F^{2n+1}(b)\omega), \hspace{1cm}
\omega\otimes b \mapsto F^{n+1}(b)\omega. }
Using again the relation $dF^{n+1}=p^{n+1}F^{n+1}d=0$ in $W_{n+1}$, one checks
immediately that all three diagrams commute. Now the claim \eqref{starBA} follows,
since the vertical maps are isomorphisms by \eqref{etlocDRW2}. As $W_{2n+2}(B_1)\to
W_{2n+2}(B)$ is \'etale \cite[Prop.~A.8]{LZ04}, we are thus reduced to the case
$B=B_1=A[x_1,\ldots, x_d]$.

Now take $\alpha\in \Omega^q_{B/A}$ and $\beta\in \Omega^{q-1}_{B/A}$ with
$V^n(\alpha)=-dV^n(\beta)$. We have to show that there exists an element $\gamma\in
F^n_*A\otimes\WC{n+1}{q-1}{B/A}$ with
\eq{gamma}{-(1\otimes F^n d)(\gamma)=\alpha \quad \text{and} \quad (1\otimes F^n)(\gamma)=\beta.}
By Corollary \ref{candecompn} (and keeping the notation used there), we can write
$\alpha$ and $\beta$ uniquely as finite sums 
\eq{alphabeta}{\alpha=\sum_{k, \sP} e_1(\xi_{k,\sP}, k,\sP),\quad
\beta=\sum_{k,\sQ}e_1(\eta_{k,\sQ},k,\sQ),\quad \text{with }
\xi_{k,\sP},\eta_{k,\sQ}\in A,}
where the sums are over all integral weights $k$ and all admissible partitions
$\sP=(I_0,\ldots, I_q)$ of length $q$ (resp.~over all admissible partitions
$\sQ=(J_0,\ldots,J_{q-1})$ of length $q-1$). Using the rules \ref{rules} (iii) and
(iv), we obtain
\ml{applyrules1}{V^n(\alpha)= \sum_{i=0}^{n-1}\sum_{\genfrac{}{}{0pt}{}{\frac{k}{p^i} \text{
primitive}}{\text{and } I_0=\emptyset}} e_{n+1}(p^{n-i}V^n(\xi_{k,\sP}),
\tfrac{k}{p^n}, \sP)\\
 + \sum_{\genfrac{}{}{0pt}{}{\frac{k}{p^n} \text{ integral }}
{\text{or } I_0\neq\emptyset}} e_{n+1}(V^n(\xi_{k,\sP}),
\tfrac{k}{p^n}, \sP) }
and
\ml{applyrules2}{-dV^n(\beta)= \sum_{\genfrac{}{}{0pt}{}{\frac{k}{p^n} \text{ integral
} }{\text{and } J_0\neq\emptyset}} -p^{t(\frac{k}{p^n})}e_{n+1}(V^n(\eta_{k,\sQ}),
\tfrac{k}{p^n}, (\emptyset,\sQ)) \\
+ \sum_{\genfrac{}{}{0pt}{}{\frac{k}{p^n} \text{ not integral } }{\text{and }
J_0\neq\emptyset}} -e_{n+1}(V^n(\eta_{k,\sQ}), \tfrac{k}{p^n}, (\emptyset,\sQ)), }
where $t(k/p^n)$ is defined as in \eqref{deftofk}. By the uniqueness of this
presentation, and since $V^n: A\to W_{n+1}(A)$ is injective, the equality
$V^n(\alpha)=-dV^n(\beta)$ thus gives the following set of equations:
\begin{align}\label{relation-xi-eta}
\xi_{k,\sP} & =  -p^{-t(\frac{k}{p^n})}\eta_{k,\sQ}, \quad \text{if } \frac{k}{p^n}
\text{ is integral, } \sP=(\emptyset,\sQ) \text{ and } J_0\neq\emptyset,\\
\eta_{k,\sQ} & =  -p^{n-i}\xi_{k,\sP}, \quad \text{if } \frac{k}{p^i} \text{ is
primitive, } \sP=(\emptyset,\sQ), \, J_0\neq \emptyset\text{ and } 0\le i\le n-1, \nonumber\\
\xi_{k,\sP} & =  0,  \quad\text{if } I_0\neq \emptyset\nonumber.
\end{align}
We claim that \eqref{gamma} holds for the following choice of
$\gamma\in F^n_*A\otimes_{W_{n+1}(A)} \WC{n+1}{q-1}{B/A}$:
\begin{multline*}
\gamma:=\sum_{i=0}^{n-1}\left(\sum_{\genfrac{}{}{0pt}{}{\frac{k}{p^i} \text{
primitive}}{\text{and } J_0\neq\emptyset} } \left(-\xi_{k,(\emptyset,\sQ)}\otimes
e_{n+1}(V^{n-i}(1),\tfrac{k}{p^n},\sQ)\right)\right. \\
\left. +\sum_{\genfrac{}{}{0pt}{}{\frac{k}{p^i} \text{ primitive}}{\text{and }
J_0=\emptyset} } \left(\eta_{k,\sQ}\otimes
e_{n+1}(V^{n-i}(1),\tfrac{k}{p^n},\sQ)\right) \right)\\
+ \sum_{\frac{k}{p^n} \text{ integral, } \sQ} \eta_{k\sQ}\otimes e_{n+1}(1,
\tfrac{k}{p^n},\sQ).
\end{multline*}
Indeed the rules \ref{rules} (ii) and (iv) yield the following formulas for $k$ an integral weight, $\xi\in V^{u(\frac{k}{p^n})}W_{n+1-u(\frac{k}{p^n})}(A)$ and 
$\sQ=(J_0,\ldots, J_{q-1})$ a partition of length $q-1$ of ${\rm supp\,}k$\,:
\eqn{F^n e_{n+1}(\xi,\tfrac{k}{p^n},\sQ)=\begin{cases} e_1(F^n(\xi), k, \sQ) & \text{if $J_0\neq \emptyset$ or $\frac{k}{p^n}$ integral,}\\
                                                         e_1(F^i V^{-(n-i)}(\xi), k, \sQ) & \text{if $J_0=\emptyset$ and $\frac{k}{p^i}$ is primitive,}\\
                                                                                               & \text{\phantom{if} for $0\le i\le n-1$}
                                                                           \end{cases} }
and
\eqn{F^n d e_{n+1}(\xi, \tfrac{k}{p^n},\sQ)=\begin{cases}        0             & \text{if $J_0=\emptyset$,}\\
                                                          p^{-t(\frac{k}{p^n})} e_1(F^n(\xi), k, (\emptyset, \sQ)) & \text{if $J_0\neq \emptyset$ and $\tfrac{k}{p^n}$ integral,}\\
                                                          e_1(F^i V^{-(n-i)}(\xi), k, (\emptyset,\sQ)) & \text{if $J_0\neq\emptyset$ and $\frac{k}{p^i}$ is primitive,}\\
                                                                                                                                & \text{\phantom{if} for $0\le i\le n-1$.}
                                            \end{cases} }
Using this, the rule \ref{rules} (i) and the relations \eqref{relation-xi-eta} we obtain
\eqna{
(-1\otimes F^n d)(\gamma) & = & \sum_{i=0}^{n-1}\sum_{\genfrac{}{}{0pt}{}{\frac{k}{p^i} \text{primitive}}{\text{and } J_0\neq\emptyset} } e_1(\xi_{k,(\emptyset,\sQ)}, k, (\emptyset, \sQ))\\
                         &  &    + \sum_{\genfrac{}{}{0pt}{}{\frac{k}{p^n} \text{integral}}{\text{and } J_0\neq\emptyset}} e_1(-p^{-t(\frac{k}{p^n})}\eta_{k,\sQ}, k, (\emptyset,\sQ)) \\
                      & = & \alpha }
and 
\eqna{
(1\otimes F^n)(\gamma) & = & \sum_{i=0}^{n-1}\left(\sum_{\genfrac{}{}{0pt}{}{\frac{k}{p^i} \text{primitive}}{\text{and } J_0\neq\emptyset} } e_1(-p^{n-i}\xi_{k,(\emptyset,\sQ)},k,\sQ)\right. 
                         \left. +\sum_{\genfrac{}{}{0pt}{}{\frac{k}{p^i} \text{ primitive}}{\text{and } J_0=\emptyset} } e_1(\eta_{k,\sQ},k,\sQ) \right)\\
                     &  & + \sum_{\frac{k}{p^n} \text{ integral, } \sQ} e_1(\eta_{k\sQ},k,\sQ)\\
                     & = & \beta. }
This proves the proposition.
\end{proof}

\subsection{}\label{Cartier}
We now recall some facts from \cite[0, 2]{Il79} about the Cartier operator and its iterates.

Let $S$ be an $\FF_p$-scheme, $X\to S$ a smooth morphism, and set
$X^{(p^n)}:= S\times_{S,F_S^n} X$. We have the usual diagram, which defines the
iterates $F^n_{X/S}$ of the relative Frobenius morphism (we write $F_{X/S}=F^1_{X/S}$, 
$W = W^1$):
\eqn{ \xymatrix{
X\ar[r]^{F^n_{X/S}}\ar[dr]\ar@/^3pc/[rr]^{F^n_X}  & X^{(p^n)}\ar[d]\ar[r]^{W^n}  & X\ar[d]\\
			 &                    S\ar[r]^{F^n_S} &   S.
} } 
Notice that 
\eqn{F^n_{X/S} = F_{X^{(p^{n-1})}/S}\circ\ldots\circ F_{X/S}. }
For an $S$-morphism $f:X'\to X$ we denote by $f^{(p^n)}$ the base-change morphism
$f^{(p^n)} = \Id_S \times f : {X'}^{(p^n)}\to X^{(p^n)}$.

The inverse Cartier operator is a homomorphism of graded $\sO_{X^{(p)}}$-algebras
\eqn{ C^{-1}_{X/S}: \Omega\hbul_{X^{(p)}/S}\to \sH\hbul(\Omega\hbul_{X/S}), }
which is uniquely determined by 
\eq{defCartier}{{C^{-1}_{X/S}}\restr{\sO_{X^{(p)}}} = F_{X/S}^* \quad \text{and}\quad
C^{-1}_{X/S}(W^*dx)=x^{p-1}dx, \quad \text{for all } x\in \sO_X.}
The inverse Cartier operator is an isomorphism (since $X/S$ is smooth). For $n\ge 0$,
one defines abelian subsheaves of $\Omega^q_{X/S}$ 
\eq{defBZ}{B_n\Omega^q_{X/S}\subset Z_n\Omega^q_{X/S}\subset \Omega^q_{X/S}}
via
\eqn{B_0\Omega^q_{X/S}=0, \quad  Z_0\Omega^q_{X/S}=\Omega^q_{X/S}, }
\eqn{B_1\Omega^q_{X/S}=B\Omega^q_{X/S}=d\Omega^{q-1}_{X/S}, \quad
Z_1\Omega^q_{X/S}=Z\Omega^q_{X/S}={\rm Ker}(d:\Omega^q_{X/S}\to \Omega^{q+1}_{X/S}), }
and, for $n\ge 1$, 
\eq{defBn}{ C^{-1}_{X/S} : B_n\Omega^q_{X^{(p)}/S}\stackrel{\simeq}{\lra}
B_{n+1}\Omega^q_{X/S}/B_1\Omega^q_{X/S}, }
\eq{defZn}{ C^{-1}_{X/S} : Z_n\Omega^q_{X^{(p)}/S}\stackrel{\simeq}{\lra}
Z_{n+1}\Omega^q_{X/S}/B_1\Omega^q_{X/S}. }
We obtain a chain of inclusions
\ml{inclBnZn}{0\subset B_1\Omega^q_{X/S}\subset \ldots\subset B_n\Omega^q_{X/S}\subset
B_{n+1}\Omega^q_{X/S}\subset \ldots \\
\subset Z_{n+1}\Omega^q_{X/S}\subset Z_n\Omega^q_{X/S}\subset\ldots\subset
Z_1\Omega^q_{X/S}\subset \Omega^q_{X/S}.}

\begin{prop}[{\cite[0, (2.2.7), Prop.~2.2.8]{Il79}}]\label{propBnZn}
Let $S$ be an $\FF_p$-scheme and $X$ a smooth $S$-scheme. Then, for all $q\ge 0$
and $n\ge 1$, the sheaves $Z_n\Omega^q_{X/S}$ and $B_n\Omega^q_{X/S}$ satisfy the 
following properties.

\romain $Z_n\Omega^q_{X/S}$ and $B_n\Omega^q_{X/S}$ are locally free
$\sO_{X^{(p^n)}}$-modules of finite type, and, for any $h: S'\to S$, we have
\eqn{ h_X^{(p^n)*}Z_n\Omega^q_{X/S} \riso Z_n\Omega^q_{X'/S'},\quad
h_X^{(p^n)*}B_n\Omega^q_{X/S} \riso B_n\Omega^q_{X'/S'}, }
where $h_X : X' := S'\times_S X\to X$ is the base-change map.

\romain  If $f: X'\to X$ is an \'etale $S$-morphism, then there are natural isomorphisms
\eqn{{f^{(p^n)}}{}^*Z_n\Omega^q_{X/S} \riso Z_n\Omega^q_{X'/S},\quad
{f^{(p^n)}}{}^*B_n\Omega^q_{X/S} \riso B_n\Omega^q_{X'/S}. }

\romain $B_n\Omega^q_{X/S}$ is the sub-$\sO_S$-module of $\Omega^q_{X/S}$ locally
generated by sections of the form $a_1^{p^r-1}\cdots a_q^{p^r-1}da_1\cdots da_q,$ with
$a_i\in \sO_X$ and $0\le r\le n-1$.

\romain $Z_n\Omega^q_{X/S}$ is the sub-$\sO_S$-module of $\Omega^q_{X/S}$ locally
generated by $B_n\Omega^q_{X/S}$ and sections of the form $ba_1^{p^n-1}\cdots
a_q^{p^n-1}da_1\cdots da_q$, with $a_i\in \sO_X$ and $b\in \sO_{X^{(p^n)}}$.
\end{prop}

\begin{prop}[cf.~{\cite[I, Prop.~3.3]{Il79}}] \label{GenCartier}
For $X/S$ smooth as above, there is a unique map of $W_n(\sO_S)$-modules
\eqn{ C^{-1}_n: F_{S*}W_n(\sO_S)\otimes_{W_{n}(\sO_S)}\WC{n}{q}{X/S}\lra
\frac{\WC{n}{q}{X/S}}{dV^{n-1}\Omega^q_{X/S}}, }
which makes  the following diagram commutative
\eqn{ \xymatrix{
F_*W_n(\sO_S)\otimes_{W_{n+1}(\sO_S)}\WC{n+1}{q}{X/S}\ar[r]^-{1\otimes
F}\ar[d]^-{1\otimes R} & \WC{n}{q}{X/S}\ar[d]\\
F_*W_n(\sO_S)\otimes_{W_n(\sO_S)}\WC{n}{q}{X/S}\ar[r]^-{C^{-1}_n} &
\frac{\WC{n}{q}{X/S}}{dV^{n-1}\Omega^q_{X/S}}.
} }
For $n=1$ we have $F_{S*}\sO_S\otimes_{\sO_S}\Omega^q_{X/S}=\Omega^q_{X^{(p)}/S}$, and
$C^{-1}_n=C^{-1}:\Omega^q_{X^{(p)}/S}\lra \frac{\Omega^q_{X/S}}{d\Omega^q_{X/S}}$ is
the inverse Cartier operator.
\end{prop}

\begin{proof}
Since $1\otimes R$ is surjective, it is enough to see that the kernel of $1\otimes R$
is mapped to $dV^{n-1}\Omega^q_{X/S}$ under $1\otimes F$. But an element in the kernel
of $1\otimes R$ is a sum of elements of the form $a\otimes V^n\omega$ and $a\otimes
dV^n\eta$, with $a\in W_n(\sO_S)$, $\omega\in\Omega^q_{X/S}$ and $\eta\in
\Omega^{q-1}_{X/S}$. We have in $\WC{n}{q}{X/S}$
\eqn{ (1\otimes F)(a\otimes V^n\omega)=a V^{n-1}(p\omega)=0,\quad (1\otimes F)(a\otimes
dV^n\eta)=dV^{n-1}(F^{n-1}(a)\eta). }
This gives the existence and the uniqueness of $C^{-1}_n$. The second statement follows
from the fact that $1\otimes F$ is compatible with products, and from the formula
$1\otimes F(a\otimes d[x])=ax^{p-1}dx$, for $a\in \sO_S$, $x\in\sO_X$.
\end{proof}

\begin{cor}[cf.~{\cite[I, Prop.~3.11]{Il79}}] \label{ImFn}
Let $X/S$ be as above. Then:

\romain $\image(1\otimes F^n: F^n_*\sO_S\otimes_{W_{n+1}(\sO_S)}\WC{n+1}{q}{X/S}\to
\Omega^q_{X/S})=Z_n\Omega^q_{X/S}. $

\romain $\image(1\otimes F^{n-1}d:
F^n_*\sO_S\otimes_{W_{n+1}(\sO_S)}F_*\WC{n}{q-1}{X/S}\to
\Omega^q_{X/S})=B_n\Omega^q_{X/S}. $
\end{cor}

\begin{proof}
We do induction on $n$. For $n=1$, (i) follows from Proposition \ref{GenCartier} and
the relation $d = FdV$, and (ii) holds by definition. Now assume the statements are
proven for $n$. To prove (i) for $n+1$, we consider the following commutative diagram 
of abelian sheaves on $X$:
\eqn{ \xymatrix{ F^{n+1}_*\sO_S\otimes_{W_{n+2}(\sO_S)}
\WC{n+2}{q}{X/S}\ar[r]^-{1\otimes F^n}\ar[d]^{1\otimes R} &
F_*\sO_S\otimes_{W_2(\sO_S)} \WC{2}{q}{X/S}\ar[r]^-{1\otimes F}\ar[d]^{1\otimes R} &
\Omega^q_{X/S}\ar[d]\\
F^{n+1}_*\sO_S \otimes_{W_{n+1}(\sO_S)}\WC{n+1}{q}{X/S}\ar[r]^-{1\otimes F^n} &
F_*\sO_S\otimes_{\sO_S}\Omega^{q}_{X/S}=\Omega^{q}_{X^{(p)}/S}\ar[r]^-{C^{-1}}
& \frac{\Omega^q_{X/S}}{d\Omega^{q-1}_{X/S}}.
} }
By induction hypothesis we have 
\eqn{\image \Big( (1\otimes R)\circ (1\otimes F^n) \Big) =\image\Big( (1\otimes
F^n)\circ(1\otimes R)\Big) = F_{S*}\sO_S\otimes_{\sO_S} Z_n\Omega^q_{X/S}=
Z_n\Omega^q_{X^{(p)}/S}, }
where the last equality follows from the compatibility with base-change. Now, thanks to
the relation $d = F^{n+1}dV^{n+1}$, (i) follows from the definition of
$Z_{n+1}\Omega^q_{X/S}$. The proof of (ii) is similar.
\end{proof}

\begin{lem}\label{ImFnd}
Let $X/S$ be as above. The sheaf $B_n\Omega^q_{X/S}$ is given by
\begin{multline*}
\image(1\otimes F^{n-1}d:
F^n_*\sO_S\otimes_{W_{n+1}(\sO_S)}F_*\WC{n}{q-1}{X/S}\to \Omega^q_{X/S})\\
= \{(1\otimes F^n d)(\alpha)\,|\, \alpha\in
F^n_*\sO_S\otimes_{W_{n+1}(\sO_S)}\WC{n+1}{q-1}{X/S} \mathrm{\ with\ }
(1\otimes F^n)(\alpha)=0\}.
\end{multline*}
\end{lem}

\begin{proof}
We call the left hand side $\sA$, and the right hand side $\sB$. We know from the
previous corollary that $B_n\Omega^q_{X/S} = \sA$, and we want now to show that $\sA =
\sB$. In the following, all non-specified tensor products are over $W_{n+1}(\sO_S)$. We
have the commutative diagram
\eqn{\xymatrix{ 
F^n_*\sO_S\otimes F_*\WC{n}{q-1}{X/S} \ar[rr]^-{1\otimes V} \ar[dr]_{1\otimes F^{n-1}
d} & & F^n_*\sO_S\otimes \WC{n+1}{q-1}{X/S} \ar[dl]^{1\otimes F^n d}\\
& \Omega^q_{X/S} & \hspace{3.5cm}.
} }
Since we also have $(1\otimes F^n)\circ(1\otimes V)=0$ it follows that $\sA\subset
\sB$. It remains to show
\begin{multline}\label{BinA}
\Ker\left(1\otimes F^n : F^n_*\sO_S\otimes\WC{n+1}{q-1}{X/S}\to
F^n_*\Omega^{q-1}_{X/S}\right)\\
\subset\image\left(F^n_*\sO_S\otimes (F_*\WC{n}{q-1}{X/S}\oplus F_*\WC{n}{q-2}{X/S})
\xra{1\otimes(V+ dV)} F^n_*\sO_S\otimes\WC{n+1}{q-1}{X/S} \right).
\end{multline}
Indeed, if we take an element $\alpha$ in the kernel on the left hand side and we write
it as an element in the right hand side $\alpha = (1\otimes V)(\beta) + (1\otimes
dV)(\gamma)$, then $(1\otimes F^n d)(\alpha)= (1\otimes F^{n-1}d)(\beta)$, i.e.,
$\sB\subset\sA$. The question is local in $X$, we may thus assume $X$ is \'etale over
$\A^d_S$. For a $W_n(\sO_X)$-module $\sM$ we write $F^r_*\sM_*F^s$ for $\sM$ viewed as
a left $W_{n+r}(\sO_S)$-module via $F^r$ and as a right $W_{n+s}(\sO_X)$-module via
$F^s$. Then we have the following commutative diagram, in which the most right tensor
product in the upper line is over $W_{2n+2}(\sO_{\A^d_S})$:
\eqn{\xymatrix{ \bigg(F^n_*\sO_S\otimes F_*(\WC{n}{q-2}{\A^d_S/S})_*F^{n+2}
\ar[r]^-{1\otimes dV} \ar[d]_-{(1\otimes \can)\otimes 1} &  F^n_*\sO_S\otimes
(\WC{n+1}{q-1}{\A^d_S/S})_*F^{n+1} \bigg)\otimes W_{2n+2}(\sO_X)
\ar[d]^{(1\otimes \can)\otimes 1}\\
F^n_*\sO_S \otimes F_*(\WC{n}{q-2}{X/S})_*F^{n+2} \ar[r]^-{1\otimes dV} &
F^n_*\sO_S \otimes (\WC{n+1}{q-1}{X/S})_*F^{n+1}.
} }
If we write $V$ instead of $dV$ and $q-1$ on the left hand side instead of $q-2$, we
obtain again a commutative diagram. Since $X/\A^d_S$ is \'etale, the vertical maps are
isomorphisms (in both diagrams). Thus if we denote the image in \eqref{BinA} by
$\Im(X/S)$ we obtain
\eqn{\Im(X/S)\cong\Im(\A^d_S/S)_*F^{n+1}\otimes_{W_{2n+2}(\sO_{\A^d_S})}W_{2n+2}(\sO_X).
}
Similarly, denoting the kernel in \eqref{BinA} by $\Ker(X/S)$ one finds
\eqn{\Ker(X/S)\cong
\Ker(\A^d_S/S)_*F^{n+1}\otimes_{W_{2n+2}(\sO_{\A^d_S})}W_{2n+2}(\sO_X). }
And, since $W_{2n+2}(\sO_X)$ is \'etale over $W_{2n+2}(\sO_{\A^d_S})$ 
\cite[Prop.~A.8]{LZ04}, it is thus enough to prove \eqref{BinA} in the case $S=\Spec
A$, with $A$ an $\FF_p$-algebra, and $X=\Spec B$, with $B=A[x_1,\ldots, x_d]$.

Now, using the notation of Corollary \ref{candecompn}, any element $\alpha\in
F^n_*A\otimes \WC{n+1}{q-1}{B/A}$ can be written as a finite sum
\eq{candecompn2}{\alpha= \sum_i\sum_{\genfrac{}{}{0pt}{}{p^n k \text{ integral}}{\sP
=(I_0,\ldots, I_{q-1})}} a_i\otimes e_{n+1}(V^{u(k)}(\eta_{k,\sP,i}), k,\sP),\quad
\eta_{k,\sP,i}\in W_{n+1-u(k)}(A). }
By the rule \ref{rules}, (ii) we have 
\eqn{F^n e_{n+1}(V^{u(k)}(\eta),k,\sP) = 
\begin{cases}
e_1(F^{n-u(k)}(\eta), p^nk,\sP) & \text{if } I_0=\emptyset \text{ or }
(I_0\neq\emptyset, k \text{ integral),}\\
0 & \text{if } I_0\neq\emptyset \text{ and } k \text{ not integral.}
\end{cases} } 
It follows that an element $\alpha$ as in \eqref{candecompn2} lies in ${\rm
Ker}(1\otimes F^n)={\rm Ker}(B/A)$ iff it satisfies 
\eq{kercond}{\sum_i a_i F^{n-u(k)}(\eta_{k,\sP,i})=0, \quad \text{for } I_0=\emptyset
\text{ or } (I_0\neq\emptyset, k \text{ integral).} }
We consider the following three cases:

\numero {\em $k$ is integral, i.e., $u(k)=0$.} Then, by Definition \ref{basicWdif},
$e_{n+1}(\eta, k, \sP)=\eta e_{n+1}(1,k, \sP)$. By \eqref{kercond}, we get 
\eqn{\sum_i a_i\otimes e_{n+1}(\eta_{i,k,\sP},k,\sP)=\left(\sum_i
a_i F^n(\eta_{i,k,\sP})\right)\otimes e_{n+1}(1,k,\sP)=0. }

\numero {\em $k$ is not integral and $I_0=\emptyset$.} In this case
$e_{n+1}(\eta,k,\sP)\in\image(dV)$ by Definition \ref{basicWdif}. Thus
\eqn{\sum_i a_i\otimes e_{n+1}(\eta_{i,k,\sP},k,\sP)\in\image(1\otimes dV). }

\numero {\em $k$ is not integral and $I_0\neq\emptyset$.} Now $e_{n+1}(\eta,k,
\sP)\in\image(V)$ by Definition \ref{basicWdif}. Hence
\eqn{\sum_i a_i\otimes e_{n+1}(\eta_{i,k,\sP},k,\sP)\in\image(1\otimes V). }
Putting the three cases together, we see that $\alpha\in \Ker(1\otimes F^n) $
implies $\alpha\in \image(1\otimes V +1\otimes dV)=\Im(B/A)$. This gives the
statement.
\end{proof}

\begin{thm}[cf.~{\cite[I, Cor.~3.9]{Il79}}, {\cite[Th.~4.2.15]{Ol07}}]\label{Structgrn}
Let $S$ be an $\FF_p$-scheme and let $X$ be a smooth $S$-scheme. For $n,q\ge 0$, denote
by $\gr^n\WC{}{q}{X/S}$ the $n$-th graded piece of the canonical filtration
\eqn{ \Fil^n\WC{}{q}{X/S} = 
V^n \WC{}{q}{X/S}+dV^n \WC{}{q}{X/S} = \Ker(\WC{}{q}{X/S} \to \WC{n}{q}{X/S}). }
Then we have an exact sequence of $\sO_X$-modules
\eq{structgrn}{ 0\lra F^{n+1}_{X*}\frac{\Omega^q_{X/S}}{B_n\Omega^q_{X/S}}\xra{V^n} 
\gr^n\WC{}{q}{X/S} \xra{U_n} F^{n+1}_{X*}\frac{\Omega^{q-1}_{X/S}}{Z_n\Omega^{q-1}_{X/S}} \lra
0, }
where the map $U_n$ is given by $V^n(\alpha)+dV^n(\beta) \mapsto \beta$ and the
$\sO_X$-module structure on ${\rm gr}^n\WC{}{q}{X/S}$ is given via
\eqn{\sO_X=\frac{W_n\sO_X}{VW_{n-1}\sO_X} \xra{\ F\ } \frac{W_{n+1}\sO_X}{pW_n\sO_X}. }
Furthermore $F^n_{X/S*}\frac{\Omega^q_{X/S}}{B_n\Omega^q_{X/S}}$ and
$F^n_{X/S*}\frac{\Omega^{q-1}_{X/S}}{Z_n\Omega^{q-1}_{X/S}}$ are locally free
$\sO_{X^{(p^n)}}$-modules.
\end{thm}

\begin{proof}
The exactness of the sequence follows from Proposition \ref{KerR}, Corollary \ref{ImFn}
and Lemma \ref{ImFnd}. The second statement is proven as in \cite[I, Cor.~3.9]{Il79}.
By \'etale base change (Proposition \ref{propBnZn}, (ii)), we reduce the question of
the local freeness of the two extreme $\sO_{X^{(p^n)}}$-modules in the exact sequence
to the case $X=\A^d_S$. Since everything is compatible with arbitrary base change in
the base $S$ (by Proposition \ref{propBnZn}, (i)), we may also assume $S=\Spec
\FF_p$, and even $S=\Spec k$ with $k$ algebraically closed. But now the sheaves
in question are coherent on $(\A^d_k)^{(p^n)}\cong\A^d_k$, hence locally free in some
non-empty open subset, whose translates under certain closed points cover the whole of
$(\A^d_k)^{(p^n)}$. As they are invariant under translation, this gives the statement.
\end{proof}

\section{The Hodge-Witt trace morphism for projective spaces}\label{HWtrp}

Let $X$ be a noetherian $\FF_p$-scheme with a dualizing complex, and let $f : Y \to X$
be a projective complete intersection morphism of virtual relative dimension $0$. Our
goal in the next two sections is to prove that, given a factorization $f = \pi \circ i$,
where $\pi : P = \P^{\,d}_X \to X$ is the structural morphism of some projective space over
$X$, and $i : Y \inj P$ is a closed immersion, one can define for all $n \geq 1$ a
morphism
\eqn{ \tau_{i,\pi,n} : \RR f_*W_n\sO_Y \lra W_n\sO_X }
so as to satisfy the following properties:

\romain For $n = 1$, $\tau_{i,\pi,n}$ is the morphism $\tau_f$ of Theorem 
\ref{Thtau};

\romain For variable $n$, $\tau_{i,\pi,n}$ commutes with $R$, $F$ and $V$. 

Our construction of $\tau_{i,\pi,n}$ will be based on a generalization for arbitrary
$n$ of the description of $\tau_f$ given in Proposition \ref{Taugamma}: we will
construct on the one hand a trace morphism $\RR \pi_*\WC{n}{d}{P/X}[d] \to W_n\sO_X$,
which will be a generalization of the trace morphism $\Trp_{\pi}$ for the projective
space, and on the other hand a morphism $i_*W_n\sO_Y \to \WC{n}{d}{P/X}[d]$ which will be
a generalization of the morphism $\gamma_f : \sO_Y \to \omega_{P/X}[d]$ defined in 
\eqref{defgammaf}.

We begin with the trace morphism for projective spaces.

\subsection{}\label{Ordi}
We recall first from \cite[D\'ef.~1.1]{Il90} that a smooth proper
$\FF_p$-morphism $f:X\to S$ is called {\em ordinary}, if it satisfies
\eqn{ R^if_*B\Omega^q_{X/S}=0,\quad \text{for all } i,q\ge 0.}
This notion is compatible with arbitrary base-change in the base $S$, and
$\P^{\,d}_{\FF_p}$ is ordinary over $\Spec \FF_p$ \cite[Prop.~1.2,
Prop.~1.4]{Il90}. Hence if $\sE$ is a locally free $\sO_X$-module of finite rank on
some $\FF_p$-scheme $X$, then $\P(\sE)=\Proj(\Sym_{\sO_X}\sE) $ is
ordinary over $X$.

\begin{lem}\label{cohord}
Let $f:X\to S$ be ordinary. Then, for all $n\ge 1$ and $q\ge 0$, 
\eqn{ V^n : F^{n+1}_{S*}\RR f_*\Omega^q_{X/S}\riso \RR f_*\gr^n\WC{}{q}{X/S} }
is an isomorphism in the derived category of quasi-coherent $\sO_S$-modules \lp where the
$\sO_S$-module structure on the right hand side comes from the $\sO_X$-module structure
defined in Theorem \ref{Structgrn}\,\rp.
\end{lem}

\begin{proof}
This follows immediately from Theorem \ref{Structgrn} and the following claim:
\eq{imdirZB}{R^if_*Z_n\Omega^q_{X/S} \riso R^if_*\Omega^q_{X/S}, \quad
R^if_*B_n\Omega^q_{X/S}=0,\quad\text{for all } i,q\ge 0, n\ge 1.}
We prove this by induction on $n$. The statement for $B_1$ holds by definition of
ordinarity and for $Z_1$ follows from the exact sequence
\eqn{ 0\lra Z\Omega^q_{X/S}\lra \Omega^q_{X/S}\stackrel{d}{\lra} B\Omega^{q+1}_{X/S}\lra 0. }
Now for the general case consider the following commutative diagram (in which $f_*$ is
viewed as a functor on the category of abelian sheaves for the Zariski topology on $|X|
= |X^{(p)}|$)
\eqn{ \xymatrix{
R^i f_* Z_n\Omega^q_{X^{(p)}/S}\ar[r]^-{C^{-1}_{X/S}}\ar[d] & 
R^i f_*\frac{Z_{n+1}\Omega^q_{X/S}}{B_1\Omega^q_{X/S}}\ar[d] & 
R^if_*Z_{n+1}\Omega^q_{X/S}\ar[d]\ar[l] \\
R^i f_*\Omega^q_{X^{(p)}/S}\ar[r]^-{C^{-1}_{X/S}} &  
R^i f_* \frac{Z_1\Omega^q_{X/S}}{B_1\Omega^q_{X/S}} & 
R^if_* Z_1\Omega^q_{X/S}\hspace{3mm}.\hspace{-3mm}\ar[l]
} }
The horizontal maps are isomorphisms as is the vertical map on the left by induction
(notice that $X^{(p)}/S$ is also ordinary). Hence all maps in the diagram are
isomorphisms, which yields the claim for $Z_{n+1}$. To prove the statement for
$B_{n+1}$ it is enough to consider the upper line in the diagram, with $Z$ replaced by
$B$, and one immediately obtains the statement.
\end{proof}

\subsection{}\label{Dlog}
Let $S$ be a scheme on which $p$ is locally nilpotent, and $X$ an $S$-scheme. As in the
classical case \cite[I, 3.23]{Il79}, we define for any $n \geq 1$ the log derivation
$\dlog_n$ to be the morphism of abelian sheaves 
\eqn{ \dlog_n: \sO^\times_X \lra \WC{n}{1}{X/S},\quad a\mapsto
\dlog_n(a):=\frac{d[a]}{[a]}. }
We may write simply $\dlog$ if $n$ is fixed. 

For variable $n$, the maps $\dlog_n$ satisfy the following relations:
\eq{dlog}{ R(\dlog_n(a)) = \dlog_{n-1}(a), \quad\quad F(\dlog_n(a)) = \dlog_{n-1}(a). }

The maps $\dlog_n$ allow to define Chern classes for line bundles, and to prove for 
relative Hodge-Witt cohomology the analog of the classical theorem on the cohomology of 
projective bundles (cf.~\cite[XI, Thm.~1.1]{SGA 7 II}).

\begin{thm}\label{HWcoh}
Let $X$ be an $\FF_p$-scheme, $\sE$ a locally free $\sO_X$-module of rank $d+1$,
$P = \P(\sE)$, and let $\pi: P\to X$ be the canonical projection. Denote by $\eta_n \in
H^0(X, R^1\pi_*\WC{n}{1}{P/X})$ the image under $\dlog_n$ of the class of $\sO_P(1)$ in
$R^1\pi_*\sO^\times_P$, and by $\eta^q_n \in H^0(X, R^q\pi_*\WC{n}{q}{P/X})$ its
$q$-fold cup product. Then, for all $n \ge 1$ and all $q$ such that $0 \leq q \leq d$,
we have
\eq{hwcoh1}{ R^j\pi_*\WC{n}{q}{P/X} = 0 \quad \text{for }j\neq q, }
and multiplication with $\eta^q_n$ induces an isomorphism in the derived category of
$W_n(\sO_X)$-modules
\eq{hwcoh2}{ W_n(\sO_X)[-q] \riso \RR\pi_*\WC{n}{q}{P/X}. } 
Furthermore these isomorphisms are compatible with restriction, Frobenius and
Verschiebung on both sides.
\end{thm}

\begin{proof}
To prove \eqref{hwcoh1}, we can argue by induction using the exact sequences 
\eqn{ 0 \lra \gr^n \WC{n+1}{q}{P/X} \lra \WC{n+1}{q}{P/X} \lra \WC{n}{q}{P/X} \lra 0. }
For $n=1$, the claim follows from \cite[XI, Thm.~1.1]{SGA 7 II}, and, since $\P(\sE)$ is
ordinary over $X$, Lemma \ref{cohord} implies similarly the claim for all $n$. 

Therefore, we obtain a canonical isomorphism
\eq{hwisom}{ \RR\pi_*\WC{n}{q}{P/X} \riso R^q\pi_*\WC{n}{q}{P/X}[-q], }
and we can define the morphism \eqref{hwcoh2} as corresponding via \eqref{hwisom} and 
translation to the morphism
\eq{hwcoh3}{ W_n(\sO_X) \lra R^q\pi_*\WC{n}{q}{P/X}, \quad w \mapsto w\eta^q_n. } 
This reduces the proof of the theorem to proving that \eqref{hwcoh3} is an 
isomorphism, compatible with $R$, $F$ and $V$.

From \eqref{dlog}, we get for all $w \in W_{n+1}(\sO_X)$ the relations
\eq{dlogRF}{ R(w\eta^q_{n+1})=R(w)\eta^q_n, \quad\quad F(w\eta^q_{n+1}) = F(w)\eta^q_n }
in $R^q\pi_*\WC{n}{q}{P/X}$. From the second relation, we also get 
\eq{dlogV}{ V(w\eta^q_{n-1}) = V(wF(\eta^q_n)) = V(w)\eta^q_n }
for all $w \in W_{n-1}(\sO_X)$. So the homomorphisms \eqref{hwcoh3} satisfy the 
required compatibilities.

To prove that the homomorphisms \eqref{hwcoh3} are isomorphisms, we may now again argue
by induction on $n$, using the compatibility with $R$ and $V$. Then Lemma \ref{cohord}
reduces the proof to the case $n=1$, which is known by \cite[Exp. XI, Thm.~1.1]{SGA 7
II}.
\end{proof}

\begin{defn}\label{DefHWtrp}
Under the assumptions of Theorem \ref{HWcoh}, we define the Hodge-Witt trace morphism
for the projective space $\P(\sE)$ to be the $W_n\sO_X$-linear map
\eq{drwTrp}{ \Trp_{\pi,n}: \RR\pi_*\WC{n}{d}{\P(\sE)/X}[d] \riso W_n\sO_{X} }
obtained by inverting the isomorphism \eqref{hwcoh2}, shifting by $d$ and multiplying
by $(-1)^{d(d-1)/2}$. Theorem \ref{HWcoh} implies that $\Trp_{\pi,n}$ is compatible with
restriction, Frobenius and Verschiebung. 
\end{defn}

\begin{prop}\label{CompTrp}
With the hypotheses of Theorem \ref{HWcoh}, assume in addition that $X$ is locally
noetherian. Then the morphism
\eq{W1Trp}{ \Trp_{\pi,1} : \RR\pi_*\Omega^d_{\P(\sE)/X}[d] \riso \sO_X }
defined by \eqref{drwTrp} for $n=1$ is equal to the morphism $\Trp_{\pi}$ defined by 
\cite[(2.3.5)]{Co00} for $\sO_X$. 
\end{prop}

\begin{proof}
By \eqref{hwcoh2}, it suffices to prove the proposition locally on $X$. So we may
assume that $\P(\sE) = \P^{\,d}_X$. Let $X_0,\ldots,X_d$ be the standard homogeneous
coordinates on $\P^{\,d}_X$, $x_i = X_i/X_0$, $U_i = D_+(X_i)$, and let $\fU =
(U_i)_{i=0,\ldots,d}$ be the corresponding covering of $\P^{\,d}_X$. Using \v{C}ech
cohomology relative to $\fU$, $\eta_1$ is defined by the 1-cocycle 
$(\dlog(X_j/X_i))_{i<j} = (d(X_j/X_i)/(X_j/X_i))_{i<j}$, and $\eta_1^d$ by the
$d$-cocycle given by 
\mln{ \dlog(X_1/X_0) \wedge \cdots \wedge \dlog(X_d/X_{d-1}) = \\  
dx_1/x_1 \wedge (dx_2/x_2 - dx_1/x_1) \wedge \ldots \wedge (dx_d/x_d - dx_{d-1}/x_{d-1}) = \\
dx_1\wedge\cdots\wedge dx_d/x_1\cdots x_d }
on $U_0 \cap \ldots \cap
U_d$. Thus $\Trp_{\pi,1}$ is the only morphism which induces on degree $0$ cohomology
the isomorphism mapping the class $dx_1\wedge\cdots\wedge dx_d/x_1\cdots x_d$ to
$(-1)^{d(d-1)/2}$.

To prove the proposition, it suffices to check that, with Conrad's definitions, the map
induced by $\Trp_{\pi} : \RR f_*(f^\sharp(\sO_X)) = \RR f_*(\omega_{P/X}[d]) \to \sO_X$
on degree $0$ cohomology is such that
\eq{ConradTrp}{ \Trp_\pi(dx_1\wedge\cdots\wedge dx_d/x_1\cdots x_d) = (-1)^{d(d-1)/2}. }
As $(-1)^d(-1)^{d(d-1)/2} = (-1)^{d(d+1)/2}$, this follows from the definition of the
isomorphism \cite[(2.3.1)]{Co00}
\eq{2.3.1}{ \gamma : R^d\pi_*(\omega_{P/X}) \riso \sO_X, } 
which sends $dx_1\wedge\cdots\wedge dx_d/x_1\cdots x_d$ to $(-1)^{d(d+1)/2}$ 
\cite[(2.3.3)]{Co00}, and from the discussion on pp.~35-36 of \cite{Co00}, which
explains that an additional $(-1)^d$ sign is required to recover \eqref{2.3.1} from the
map induced in degree $0$ by $\Trp_\pi$ (note that by ``induced'', we mean  
that we use here as we always do the standard identifications \cite[(1.3.1), (1.3.4)]{Co00}
to compute the cohomology objects of a translated complex).

This ends the proof of the proposition, but, as formula \eqref{ConradTrp} is only 
implicit in the discussion [Co00, p.~35-36], it may be worth adding a few lines to 
give a proof explaining where this extra $(-1)^d$ sign comes from. Conrad's 
construction of the projective trace $\Trp_{\pi}$ is the same as Hartshorne's in 
\cite[III, 4.3]{Ha66}, but using \cite[Lemma 2.1.1]{Co00} instead of \cite[I, Proposition 
7.4]{Ha66}. Because $\pi_*$ has cohomological dimension $d$ on the category of 
quasi-coherent $\sO_P$-modules, and any quasi-coherent $\sO_P$-module can be written 
as a quotient of modules for which the functors $R^i\pi_*$ vanish for $i \neq d$ 
\cite[III, Lemmas 4.1 and 4.2]{Ha66}, Lemma 2.1.1 of \cite{Co00} provides an
isomorphism of functors on $D(\Qcoh(\sO_P))$
\eqn{ \psi : \RR \pi_* \riso \LL(R^d\pi_*)[-d]. }
For complexes of the form $\sF\hbul = \sF[0]$, where $\sF$ is a quasi-coherent 
$\sO_P$-module, $\psi_{\sF\hbul}$ induces in degree $d$ the identity of 
$R^d\pi_*(\sF)$ \cite[Cor.~2.1.2]{Co00}. Moreover, the compatiblity of $\psi$ with 
translations, given by \cite[(2.1.1)]{Co00}, implies that, for any $m \in \Z$, we 
have 
\eqn{ \psi_{\sF\hbul[m]} = (-1)^{md}\psi_{\sF\hbul}[m]. }
In particular, $\psi_{\omega_{P/X}[d]}$ induces in degree $0$ multiplication by
$(-1)^{d^2} = (-1)^d$ on $R^d\pi_*(\omega_{P/X})$. Now, for $\sG\hbul \in
\Dpqc(\sO_X)$, the trace morphism for $\sG\hbul$ is the composition
\eqn{ \xymatrix{
 \entry{\RR \pi_*(\pi^\sharp\sG\hbul)} 
\ar[d]^-{\Trp_{\pi,\sG}\raisebox{-.4mm}{\hspace{-.4mm}$\cdot$}} \ar@{=}[r] &
\entry{\RR\pi_*(\omega_{P/X}[d]\otimes_{\sO_P}\pi^*\sG\hbul)} 
\ar[r]_-{\sim}^-{\psi_{\pi^\sharp\sG}\raisebox{-.75mm}{\hspace{-.4mm}$\cdot$}} & 
\entry{\LL(R^d\pi_*)(\omega_{P/X}[d] \otimes_{\sO_P} \pi^*\sG\hbul)[-d]}
\ar[d]_-{\sim} \\
\entry{\sG\hbul} & \entry{(R^d\pi_*)(\omega_{P/X}) \otimes_{\sO_X} \sG\hbul}  
\ar[l]^-{\sim}_-{\gamma\otimes\Id} & 
\entry{(R^d\pi_*)(\omega_{P/X} \otimes_{\sO_P} \pi^*\sG\hbul)} \ar[l]^-{\sim}
} }
(see \cite[III, 4.3]{Ha66} for details). Taking $\sG\hbul = \sO_X[0]$, and applying the 
previous remark to $\pi^\sharp\sO_X = \omega_{P/X}[d]$, we obtain that $\Trp_{\pi,\sO_X}$ 
induces $(-1)^d\gamma$ in degree $0$, which gives \eqref{ConradTrp}.
\end{proof}

\begin{rem}
Because of the differences in sign conventions between Hartshorne \cite[III,
Th~3.4]{Ha66} and Conrad \cite[2.3]{Co00}, our trace morphism $\Trp_{\pi,n}$ differs by
$(-1)^{d(d-1)/2}$ from the trace morphism defined by Ekedahl \cite[I, Lemma 3.2]{Ek84}
when $X = \Spec k$, $k$ being a perfect field.
\end{rem}

\section{The Hodge-Witt fundamental class of a regularly embedded 
subscheme}\label{HWclass}

In this section, we assume that $X$ is a locally noetherian scheme of characteristic
$p$, and we consider a regular immersion $i : Y \inj P$ of codimension $d$, where $P$
is a smooth $X$-scheme. Under these assumptions, we want to associate to $Y$ a
canonical class $\gamma_Y \in \Gamma(P, \sH^d_Y(\WC{n}{d}{P/X}))$, for each $n \geq 1$.

\begin{prop}\label{Immersion} Under the previous assumptions:

\romain If $t_1,\ldots,t_d$ is a regular sequence of sections of $\sO_P$, then, for all
$n\geq 1$ and all $r\geq 1$, $[t_1]^r,\ldots,[t_d]^r$ is a regular sequence of sections of
$W_n(\sO_P)$.

\romain For all $n\geq 1$ and all $q$, $\sH^j_Y(\WC{n}{q}{P/X}) = 0$ for $j \neq d$.
\end{prop}

\begin{proof} We proceed by induction on $n$. In the exact sequence of 
$W_{n+1}(\sO_P)$-modules 
\eqn{ 0 \lra F^n_*\sO_P \xra{V^n} W_{n+1}(\sO_P) \xra{\ R\ } W_n(\sO_P) \lra 0, }
the action of $[t_i]^r$ on $F^n_*\sO_P$ is given by multiplication by $t_i^{rp^n}$ on
$\sO_P$. As $P$ is a locally noetherian scheme, the sequence 
$t_1^{rp^n},\ldots,t_d^{rp^n}$ is regular in $\sO_P$, and the first claim follows 
easily.

For $n = 1$, the second one is a well known consequence of the regularity of the
sequence $t_1,\ldots,t_d$. As $\sO_P$ is locally free of finite rank over
$\sO_{P^{(p^n)}}$, we also have $\sH^j_Y(\sO_{P^{(p^n)}}) = 0$ for $j \neq d$. In the
exact sequence
\eqn{ 0 \lra \gr^n\WC{n+1}{q}{P/X} \lra \WC{n+1}{q}{P/X} \xra{\ R\ } \WC{n}{q}{P/X}
\lra 0, }
Theorem \ref{Structgrn} allows to endow the kernel $\gr^n\WC{n+1}{q}{P/X}$ with an 
$\sO_P$-module structure for which it is an extension of two $\sO_P$-modules which are 
locally free over $\sO_{P^{(p^n)}}$. Therefore, $\sH^j_Y(\gr^n\WC{n+1}{q}{P/X}) = 0$ 
for $j \neq d$. The second claim follows by induction. 
\end{proof}

\begin{thm}\label{Indep}
Under the assumptions of this section, let $\mbt = (t_1,\ldots,t_d)$ and $\mbt' =
(t'_1,\ldots,t'_d)$ be two regular sequences of sections of $\sO_P$ generating the
ideal $\sI$ of $Y$ in $P$. Let $n \geq 1$ be an integer, and let $\sJ =
([t_1],\ldots,[t_d])$, $\sJ' = ([t'_1],\ldots,[t'_d])$ be the ideals of $W_n(\sO_P)$
generated by the Teichm\"uller representatives of these generators. If
\eqn{ \beta_{\sJ} : \sExt^d_{W_n(\sO_P)}(W_n(\sO_P)/\sJ,\WC{n}{d}{P/X})
\lra \sH^d_Y(\WC{n}{d}{P/X}) }
is the canonical homomorphism \lp and similarly for $\beta_{\sJ'}$\rp, then, with the 
notations of \ref{KoszulToExt}, 
\eq{indep}{ \beta_{\sJ}(\left[\begin{array}{c} d[t_1]\cdots d[t_d] \\ {[t_1],\ldots,[t_d]}
\end{array}\right]) = \beta_{\sJ'}(\left[\begin{array}{c} d[t'_1]\cdots d[t'_d] \\
{[t'_1],\ldots,[t'_d]} \end{array}\right]). }
\end{thm}

\begin{proof}
It suffices to prove \eqref{indep} in a neighbourhood of each point $y \in Y$.
Localizing, one can reduce the proof of Theorem \ref{Indep} to the case of a very
simple change of generators in $\sI$, thanks to the following remarks (see also 
\cite[Cycle, Lemme 2.2.3]{SGA 412}).

\alphab If the sequence $(t'_1,\ldots,t'_d)$ is deduced from $(t_1,\ldots,t_d)$ by 
permutation, then $\sJ = \sJ'$, and formula \eqref{changet} implies the theorem.

\alphab If there exists invertible sections $a_1,\ldots,a_d \in \sO_P^{\times}$ such that
$t'_i = a_i t_i$ for all $i$, then $[t'_i] = [a_i][t_i]$ for all $i$. So $\sJ = \sJ'$, we can
apply Lemma \ref{Changet}, and we can choose the matrix $C$ to be the diagonal matrix
with entries $[a_i]$. Then the theorem follows from formula \eqref{changet}, because
an element such as \eqref{defSymbol} only depends upon the class of $m$ mod
$(t_1,\ldots,t_d)M$, and here we have the congruence
\eqn{ d[t'_1]\cdots d[t'_d] \equiv (\prod_{i=1}^d [a_i])\,d[t_1]\cdots d[t_d] \mod
\sJ\WC{n}{d}{P/X}. }

\alphab Given $y \in Y$, there exists a permutation $\sigma \in \mathfrak{S}_d$ such
that, for any $i$, $1 \leq i \leq d$, the sequence $\mbt^{(i)} = (t'_{\sigma(1)},
\ldots, t'_{\sigma(i)}, t_{i+1},\ldots,t_d)$ is a regular sequence of generators of $\sI$
around $y$. Indeed, a sequence of elements of $\sI_y$ is a regular sequence of
generators if and only if it gives a basis of $\sI_y/\mathfrak{m}_y\sI_y$, and this reduces
the claim to an elementary result in linear algebra over a field. If we set
$\mbt^{(0)}=(t_1,\ldots,t_d)$, then $\mbt^{(0)}=\mbt$, and $\mbt^{(d)}$ is deduced from
$\mbt'$ by permutation. So, using remark a), it suffices to prove the theorem for
the couple of sequences $\mbt^{(i-1)}$ and $\mbt^{(i)}$, for all $i$, $1 \leq i \leq d $.

This reduces the proof to the case where there exists an integer $i_0 \in 
\{1,\ldots,d\}$ such that
\eqn{t'_i = t_i \quad\mbox{for $i\neq i_0$,}\quad\quad t'_{i_0} = \sum_{j=1}^d 
c_{i_0,j}t_j.} 
Using remark a), we may assume that $i_0=1$. Moreover, the fact that $\mbt$ and $\mbt'$ 
induce bases of the vector space $\sI_y/\mathfrak{m}_y\sI_y$ implies that the 
coefficient $c_{1,1}$ is invertible around $y$. 

\alphab In this last case, we define inductively elements $t_1^{(j)}$ for $0 \leq j
\leq d$ by setting
\eqn{t_1^{(0)} = t_1, \quad t_1^{(1)}=c_{1,1}t_1^{(0)}, \quad t_1^{(j)} = 
t_1^{(j-1)}+c_{1,j}t_j \quad \mbox{for $1 < j$}.}
If, for $0 \leq j \leq d$, we define $\mbt^{(j)} = (t_1^{(j)},t_2,\ldots,t_d)$, then 
$\mbt^{(0)} = \mbt$, $\mbt^{(d)}=\mbt'$, and it suffices to prove the theorem for each 
of the couples $\mbt^{(j-1)}$, $\mbt^{(j)}$, for $1 \leq j \leq d$. The theorem is 
true for $\mbt^{(0)}$, $\mbt^{(1)}$, thanks to remark b), and, applying again remark 
a), we can write all the remaining couples as changes of generators of the form 
\eq{lastpair}{ t'_1 = t_1 + ct_2, \quad \mbox{for some $c \in \sO_P$}, \quad\quad t'_i
= t_i \quad \mbox{for $i \geq 2$}. }

Thus it suffices to prove the theorem for the change of generators of $I$ given by
\eqref{lastpair}. Let $h \in VW_{n-1}(\sO_P)$ be defined by setting 
\eq{defh}{ [t_1] + [c][t_2] = [t_1+ct_2] + h = [t'_1] + h }
in $W_n(\sO_P)$. Since $[t'_2]=[t_2]$, this can be rewritten as
\eq{defh2}{ [t_1] = [t'_1] - [c][t'_2] + h.}
The binomial formula gives
\eq{binomial}{[t_1]^{p^{n-1}} = ([t'_1]-[c][t'_2])^{p^{n-1}} + \sum_{i=1}^{p^{n-1}} 
\frac{p^{n-1}!}{(p^{n-1}-i)!i!}h^i ([t'_1]-[c][t'_2])^{p^{n-1}-i}.} 
Because the ideal $VW_{n-1}(\sO_P) \subset W_n(\sO_P)$ is a PD-ideal, we can write $h^i
= i!h^{[i]}$, with $h^{[i]} \in VW_{n-1}(\sO_P)$ when $i \geq 1$. Therefore the
numerical coefficient of $h^{[i]}$ in the $i$-th term of the sum is divisible by
$p^{n-1}$ for all $i \geq 1$. Since $p^{n-1}$ kills $VW_{n-1}(\sO_P)$, equation
\eqref{binomial} reduces to
\eq{binomial2}{ [t_1]^{p^{n-1}} = ([t'_1]-[c][t'_2])^{p^{n-1}}. } 

If, for all $k \geq 1$, we denote by $\sJ^{(k)}$ the ideal $([t_1]^k,\ldots,[t_d]^k)$,
this shows that $\sJ^{(p^{n-1})} \subset \sJ'$. So we can apply Lemma \ref{Changet} to
the sequences $([t'_1],\ldots,[t'_d])$ and $([t_1]^{p^{n-1}},\ldots,[t_d]^{p^{n-1}})$,
which are regular by Lemma \ref{Immersion}. Moreover, we can write equation
\eqref{binomial2} as
\eqn{ [t_1]^{p^{n-1}} = [t'_1]{}^{p^{n-1}-1}\cdot[t'_1] + c_{1,2}\cdot[t'_2],}
so that we can use as matrix $C$ in Lemma \ref{Changet} an upper triangular matrix with
diagonal entries $[t'_1]{}^{p^{n-1}-1}, \ldots, [t'_d]{}^{p^{n-1}-1}$ (since
$[t_i]^{p^{n-1}} = [t'_i]{}^{p^{n-1}-1}\cdot[t'_i]$ for $i \geq 2$). In particular,
$\det(C) = [t'_1]{}^{p^{n-1}-1}\cdots[t'_d]{}^{p^{n-1}-1}$. Thus, formula
\eqref{changet} provides the equality
\eq{change1}{ 
\alpha'(\left[\begin{array}{c} d[t'_1]\cdots d[t'_d] \\ {[t'_1],\ldots,[t'_d]}
\end{array}\right]) = \left[\begin{array}{c}
[t'_1]{}^{p^{n-1}-1}\cdots[t'_d]{}^{p^{n-1}-1}\,d[t'_1]\cdots d[t'_d] \\
{[t_1]^{p^{n-1}},\ldots,[t_d]^{p^{n-1}}} \end{array}\right], 
}
where $\alpha'$ is the canonical homomorphism
\eqn{ \sExt^d_{W_n(\sO_P)}(W_n(\sO_P)/\sJ',\WC{n}{d}{P/X}) \lra 
\sExt^d_{W_n(\sO_P)}(W_n(\sO_P)/\sJ^{(p^{n-1})},\WC{n}{d}{P/X}). } 

On the other hand, we also have $\sJ^{(p^{n-1})} \subset \sJ$. So we can also apply
Lemma \ref{Changet} to the regular sequences $([t_1],\ldots,[t_d])$ and
$([t_1]^{p^{n-1}},\ldots,[t_d]^{p^{n-1}})$, using now for $C$ the diagonal matrix with
entries $[t_1]^{p^{n-1}-1},\ldots,[t_d]^{p^{n-1}-1}$. If we denote by
\eqn{ \alpha : \sExt^d_{W_n(\sO_P)}(W_n(\sO_P)/\sJ,\WC{n}{d}{P/X}) \lra 
\sExt^d_{W_n(\sO_P)}(W_n(\sO_P)/\sJ^{(p^{n-1})},\WC{n}{d}{P/X}) }
the canonical homomorphism, formula \eqref{changet} provides the second
equality 
\eq{change2}{ 
\alpha(\left[\begin{array}{c} d[t_1]\cdots d[t_d] \\ {[t_1],\ldots,[t_d]} \end{array}\right]) 
= \left[\begin{array}{c} [t_1]^{p^{n-1}-1}\cdots[t_d]^{p^{n-1}-1}\,d[t_1]\cdots d[t_d] \\
{[t_1]^{p^{n-1}},\ldots,[t_d]^{p^{n-1}}} \end{array}\right]. 
 }
As $\beta_{\sJ} = \beta_{\sJ^{(p^{n-1})}} \circ \alpha$ and $\beta_{\sJ'} = 
\beta_{\sJ^{(p^{n-1})}} \circ \alpha'$, relation \eqref{indep} will follow if we prove 
the equality 
\eq{indep2}{ \left[\begin{array}{c} [t'_1]{}^{p^{n-1}-1}\cdots[t'_d]{}^{p^{n-1}-1}\,
d[t'_1]\cdots d[t'_d] \\
{[t_1]^{p^{n-1}},\ldots,[t_d]^{p^{n-1}}} \end{array}\right] =
\left[\begin{array}{c} [t_1]^{p^{n-1}-1}\cdots[t_d]^{p^{n-1}-1}\,d[t_1]\cdots d[t_d] \\
{[t_1]^{p^{n-1}},\ldots,[t_d]^{p^{n-1}}} \end{array}\right] }
in $\sExt^d_{W_n(\sO_P)}(W_n(\sO_P)/\sJ^{(p^{n-1})},\WC{n}{d}{P/X})$. To prove it, it 
suffices to prove in $\WC{n}{d}{P/X}$ the congruence 
\ml{indepcong1}{ [t'_1]{}^{p^{n-1}-1}[t'_2]{}^{p^{n-1}-1}\cdots[t'_d]{}^{p^{n-1}-1} 
\,d[t'_1]\,d[t'_2]\cdots d[t'_d] \\ 
\equiv [t_1]^{p^{n-1}-1}[t_2]^{p^{n-1}-1}\cdots[t_d]^{p^{n-1}-1}
\,d[t_1]\,d[t_2]\cdots d[t_d] }
mod $([t_1]^{p^{n-1}}, [t_2]^{p^{n-1}},\ldots,[t_d]^{p^{n-1}})\WC{n}{d}{P/X}$. As $t_i
= t'_i$ for $i > 2$, it suffices by multiplicativity to prove in $\WC{n}{2}{P/X}$ the
congruence
\eqn{ [t'_1]{}^{p^{n-1}-1}[t'_2]{}^{p^{n-1}-1} \,d[t'_1]\,d[t'_2] \equiv 
[t_1]^{p^{n-1}-1}[t_2]^{p^{n-1}-1}\,d[t_1]\,d[t_2] }
mod $([t_1]^{p^{n-1}}, [t_2]^{p^{n-1}})\WC{n}{2}{P/X}$, and, thanks to \eqref{FTeich},
the latter will follow by applying $F^{n-1}$ if we prove the congruence
\eq{indepcong2}{ d[t'_1]\,d[t'_2] \equiv d[t_1]\,d[t_2] \mod ([t_1],
[t_2])\WC{2n-1}{2}{P/X}. }

So let us prove \eqref{indepcong2}. We still denote by $h \in VW_{2n-2}\sO_P$ the
difference $h = [t_1] + [c][t_2] - [t'_1] = [t_1] + [ct_2] - [t_1+ct_2]$ computed in
$W_{2n-1}\sO_P$. Since $t'_2 = t_2$, it suffices to prove the congruence
\eq{indepcong3}{ dh\,d[t_2] \equiv 0 \mod ([t_1], [t_2])\WC{2n-1}{2}{P/X}. }
For all $i$, let 
\eqn{ S_i(X_0,\ldots,X_i,Y_0,\ldots,Y_i) \in \Z[X_0,\ldots,X_i,Y_0,\ldots,Y_i] }
be the universal polynomial defining the $i$-th component of the sum of two Witt 
vectors, and 
\eq{smallsum}{ s_i(X_0,Y_0) = S_i(X_0,0,\ldots,0,Y_0,0,\ldots0) \in \Z[X_0,Y_0].}
Note that, for $i \geq 1$, the polynomial $s_i(X_0,Y_0)$ is divisible by $X_0Y_0$,
since $(0,\ldots,0)$ is the zero element in a Witt vector ring. By definition, we have
\eqn{ [t_1] + [ct_2] = (t_1+ct_2, s_1(t_1,ct_2),\ldots,s_{2n-2}(t_1,ct_2)), }
and
\eqn{ h = (0,  s_1(t_1,ct_2),\ldots,s_{2n-2}(t_1,ct_2)).}
Since $s_i(X_0,Y_0)$ is divisible by $Y_0$, we can write $s_i(t_1,ct_2)=z_it_2$ for 
some section $z_i\in \sO_P$. We obtain
\eqn{ h = (0,z_1t_2,\ldots,z_{2n-2}t_2), }
which we can write as
\eqn{ h = \sum_{i=1}^{2n-2}V^i([z_i][t_2]). }
For each $i$, $1 \leq i \leq 2n-2$, we now obtain the relations 
\gan{ dV^i([z_i][t_2])\,d[t_2] = dV^i([z_i][t_2]\,F^i(d[t_2])) =  
dV^i([z_i][t_2]^{p^i}d[t_2]) \\
= dV^i([z_i]F^i([t_2])d[t_2]) = d([t_2]V^i([z_i]d[t_2])), 
}
\eqn{ d([t_2]V^i([z_i]d[t_2])) \equiv d[t_2]\,V^i([z_i]d[t_2]) \mod 
[t_2]\WC{2n-1}{2}{P/X}, }
\eqn{ d[t_2]\,V^i([z_i]d[t_2]) = V^i(F^i(d[t_2])[z_i]d[t_2]) = 
V^i([t_2]^{p^i-1}d[t_2][z_i]d[t_2]) = 0, }
which imply \eqref{indepcong3}. 
\end{proof}

\begin{defn}\label{Deffundclass}
Under the assumptions of this section, we define the \textit{$n$-th Hodge-Witt
fundamental class} $\gamma_{Y,n}$ of $Y$ in $P$ relatively to $X$ as being the section
of $\sH^d_Y(\WC{n}{d}{P/X})$ obtained by glueing the sections
$\beta_{\sJ}(\left[\begin{array}{c} d[t_1]\cdots d[t_d] \\
{[t_1],\ldots,[t_d]} \end{array}\right])$ defined locally by regular sequences of
generators of the ideal $\sI$ of $Y$ in $P$.
\end{defn}

\begin{prop}\label{Propfundclass}
For $n \geq 1$, let 
\gan{ R : \sH^d_Y(\WC{n+1}{d}{P/X}) \lra \sH^d_Y(\WC{n}{d}{P/X}), \\
F : \sH^d_Y(\WC{n+1}{d}{P/X}) \lra \sH^d_Y(\WC{n}{d}{P/X}), \\
V : \sH^d_Y(\WC{n}{d}{P/X}) \lra \sH^d_Y(\WC{n+1}{d}{P/X})}
be the homomorphisms defined by functoriality. Then 
\eq{propfundclass}{ R(\gamma_{Y,n+1}) = \gamma_{Y,n}, \quad\quad
F(\gamma_{Y,n+1}) = \gamma_{Y,n}, \quad\quad
V(\gamma_{Y,n}) = p\gamma_{Y,n+1}. }
\end{prop}

\begin{proof}
We may assume that there exists a regular sequence $t_1,\ldots,t_d$ such that $\sI =
(t_1,\ldots,t_d)$. For each $n \geq 1$, let $\sJ_n$ be the ideal of $W_n(\sO_P)$
generated by the Teichm\"uller representatives $[t_i]$ of the $t_i$'s, and let
$K\lbul([\mbt]_n)$ be the Koszul complex defined by the $[t_i]$'s over 
$W_n(\sO_P)$. Since $R([t_i]) = [t_i]$, scalar
extension through $R$ yields an isomorphism
\eqn{W_n(\sO_P) \otimes_{W_{n+1}(\sO_P)} K\lbul([\mbt]_{n+1}) \riso
K\lbul([\mbt]_n). }
Using the fact that the $[t_i]$'s form a regular sequence both in $W_{n+1}(\sO_P)$ and
in $W_n(\sO_P)$, it can be seen in the derived category of $W_n(\sO_P)$-modules as an
isomorphism
\eq{torindR}{ W_n(\sO_P) \otimesL_{W_{n+1}(\sO_P)} W_{n+1}(\sO_P)/\sJ_{n+1} \riso 
W_n(\sO_P)/\sJ_n. }
By adjunction, \eqref{torindR} defines for any $W_n(\sO_P)$-module $\sM$ and any $q \geq 
0$ an isomorphism
\eq{isoExtR}{ \sExt^q_{W_n(\sO_P)}(W_n(\sO_P)/\sJ_n, \sM) \riso 
\sExt^q_{W_{n+1}(\sO_P)}(W_{n+1}(\sO_P)/\sJ_{n+1}, \sM), }
and we obtain the diagram 
\eq{functR}{ \xymatrix@R=15pt@C=30pt{
\sH^d(\sHom\hbul_{W_{n+1}(\sO_P)}(K\lbul([\mbt]_{n+1}), 
\WC{n+1}{d}{P/X}))\hspace{-3cm} 
\ar[dd]_-{R} \ar[dr]^-{\sim} \\
& \hspace{-3cm}\sExt^d_{W_{n+1}(\sO_P)}(W_{n+1}(\sO_P)/\sJ_{n+1}, \WC{n+1}{d}{P/X}) 
\ar[r]^-{\beta_{\sJ_{n+1}}} \ar[dd]^-{R} & \sH^d_Y(\WC{n+1}{d}{P/X}) \ar[dd]^-{R} \\
\sH^d(\sHom\hbul_{W_{n+1}(\sO_P)}(K\lbul([\mbt]_{n+1}), 
\WC{n}{d}{P/X}))\hspace{-3cm} \ar[dr]^-{\sim} \\
 & \hspace{-3cm}\sExt^d_{W_{n+1}(\sO_P)}(W_{n+1}(\sO_P)/\sJ_{n+1}, \WC{n}{d}{P/X}) 
 \ar[r]^-{\beta_{\sJ_{n+1}}}  & \sH^d_Y(\WC{n}{d}{P/X}) \ar@{=}[dd] \\
\sH^d(\sHom\hbul_{W_n(\sO_P)}(K\lbul([\mbt]_n), 
\WC{n}{d}{P/X}))\hspace{-3cm} \ar[uu]^-{\wr} \ar[dr]^-{\sim} \\
& \hspace{-3cm}\sExt^d_{W_n(\sO_P)}(W_n(\sO_P)/\sJ_n, \WC{n}{d}{P/X}) 
\ar[uu]^-{\wr}_{\eqref{isoExtR}}
\ar[r]^-{\beta_{\sJ_n}} & \sH^d_Y(\WC{n}{d}{P/X}), \\
} }
in which the lower left hand square commutes by construction. On the other
hand, \eqref{isoExtR} implies that injective $W_n(\sO_P)$-modules are acyclic for the
functor $\sHom_{W_{n+1}(\sO_P)}(W_{n+1}(\sO_P)/\sJ_{n+1}, -)$. Replacing
$\WC{n}{d}{P/X}$ by an injective resolution over $W_n(\sO_P)$, it is then easy to
check that the lower right square commutes. As the upper part of the diagram 
commutes by functoriality, and $R(d[t_1]\cdots d[t_d]) =
d[t_1]\cdots d[t_d]$, the first relation of \eqref{propfundclass} follows.

Viewing now $W_n(\sO_P)$ as a $W_{n+1}(\sO_P)$-algebra via $F$, one proceeds similarly
to prove the second one. Since $F([t_i]) = [t_i^p] = [t_i]^p$, and the sequence
$[t_i]^p,\ldots,[t_d]^p$ is a regular sequence in $W_n(\sO_P)$, we obtain for any 
$W_n(\sO_P)$-module $\sM$ and any $q \geq 0$ isomorphisms
\ga{}{ W_n(\sO_P) \otimes_{W_{n+1}(\sO_P)} K\lbul([\mbt]_{n+1}) \riso
K\lbul([\mbt]_n^p), \notag\\
W_n(\sO_P) \otimesL_{W_{n+1}(\sO_P)} W_{n+1}(\sO_P)/\sJ_{n+1} \riso 
W_n(\sO_P)/\sJ_n^{(p)}, \label{torindF}\\
\sExt^q_{W_n(\sO_P)}(W_n(\sO_P)/\sJ_n^{(p)}, \sM) \riso 
\sExt^q_{W_{n+1}(\sO_P)}(W_{n+1}(\sO_P)/\sJ_{n+1}, \sM). \label{isoExtF} }
They provide a commutative diagram similar to \eqref{functR}:
\eq{functF}{ \xymatrix@R=15pt@C=30pt{
\sH^d(\sHom\hbul_{W_{n+1}(\sO_P)}(K\lbul([\mbt]_{n+1}), 
\WC{n+1}{d}{P/X}))\hspace{-3cm} 
\ar[dd]_-{F} \ar[dr]^-{\sim} \\
& \hspace{-3cm}\sExt^d_{W_{n+1}(\sO_P)}(W_{n+1}(\sO_P)/\sJ_{n+1}, \WC{n+1}{d}{P/X}) 
\ar[r]^-{\beta_{\sJ_{n+1}}} \ar[dd]^-{F} & \sH^d_Y(\WC{n+1}{d}{P/X}) \ar[dd]^-{F} \\
\sH^d(\sHom\hbul_{W_{n+1}(\sO_P)}(K\lbul([\mbt]_{n+1}), 
\WC{n}{d}{P/X}))\hspace{-3cm} \ar[dr]^-{\sim} \\
& \hspace{-3cm}\sExt^d_{W_{n+1}(\sO_P)}(W_{n+1}(\sO_P)/\sJ_{n+1}, \WC{n}{d}{P/X}) 
\ar[r]^-{\beta_{\sJ_{n+1}}} & \sH^d_Y(\WC{n}{d}{P/X}) \ar@{=}[dd]\\
\sH^d(\sHom\hbul_{W_n(\sO_P)}(K\lbul([\mbt]_n^p), 
\WC{n}{d}{P/X}))\hspace{-3cm} \ar[uu]^-{\wr} \ar[dr]^-{\sim} \\
& \hspace{-3cm}\sExt^d_{W_n(\sO_P)}(W_n(\sO_P)/\sJ_n^{(p)}, \WC{n}{d}{P/X}) 
\ar[uu]^-{\wr}_{\eqref{isoExtF}} \ar[r]^-{\beta_{\sJ_n^{(p)}}} & \sH^d_Y(\WC{n}{d}{P/X}). \\
} }
Since $F(d[t_1]\cdots d[t_d]) = [t_1]^{p-1}\cdots[t_d]^{p-1}d[t_1]\cdots d[t_d]$, it 
follows that 
\eqn{ F(\beta_{\sJ_{n+1}}(\left[\begin{array}{c} d[t_1]\cdots d[t_d] \\
{[t_1],\ldots,[t_d]} \end{array}\right])) = 
\beta_{\sJ_n^{(p)}}(\left[\begin{array}{c} [t_1]^{p-1}\cdots[t_d]^{p-1}\,d[t_1]\cdots
d[t_d] \\
{[t_1]^p,\ldots,[t_d]^p} \end{array}\right]). }
On the other hand, if $\alpha$ denotes the canonical homomorphism 
\eqn{ \alpha : \sExt^d_{W_n(\sO_P)}(W_n(\sO_P)/\sJ_n, \WC{n}{d}{P/X}) \lra 
\sExt^d_{W_n(\sO_P)}(W_n(\sO_P)/\sJ_n^{(p)}, \WC{n}{d}{P/X}), } 
we have by \eqref{changet} 
\eqn{ \alpha(\left[\begin{array}{c} d[t_1]\cdots d[t_d] \\
{[t_1],\ldots,[t_d]} \end{array}\right]) = 
\left[\begin{array}{c} [t_1]^{p-1}\cdots[t_d]^{p-1}\,d[t_1]\cdots d[t_d] \\
{[t_1]^p,\ldots,[t_d]^p} \end{array}\right]. }
As $\beta_{\sJ_n^{(p)}} \circ \alpha = \beta_{\sJ_n}$, it follows that 
$F(\gamma_{Y,n+1}) = \gamma_{Y,n}$.

The last relation of \eqref{propfundclass} follows formally, because $V(\gamma_{Y,n}) 
= V(F(\gamma_{Y,n+1})) = p\gamma_{Y,n+1}$. 
\end{proof}

\begin{prop}\label{Defgamipi}
Let $n \geq 1$ be an integer, and let $\gamma_{Y,n} \in \sH^d_Y(\WC{n}{d}{P/X})$ be the
Hodge-Witt fundamental class of $Y$ in $P$ relatively to $X$, as defined in 
\ref{Deffundclass}.

\romain The linear homomorphism $W_n(\sO_P) \to
\sH^d_Y(\WC{n}{d}{P/X})$ sending $1$ to $\gamma_{Y,n}$ vanishes on $W_n(\sI) :=
\Ker(W_n(\sO_P) \surj i_*W_n(\sO_Y))$.

\romain Let $\gamma_{i,\pi,n}$ be the composition
\eq{defgamipi}{ \gamma_{i,\pi,n} : i_*W_n(\sO_Y) \lra \sH^d_Y(\WC{n}{d}{P/X})  \riso 
\RR\uGamma_Y(\WC{n}{d}{P/X}[d]) \lra \WC{n}{d}{P/X}[d], }
where the first morphism is defined thanks to the previous assertion. Then
$\gamma_{i,\pi,n}$ commutes with $R$, $F$ and $V$.

\romain For $n=1$, we have $\gamma_{i,\pi,1} = \gamma_f$, where $\gamma_f$ is the 
morphism defined by \eqref{defgammaf}.
\end{prop}

\begin{proof}
To prove assertion (i), we may again assume that $\sI$ is generated by a regular
sequence $t_1,\ldots,t_d$. Any section $w$ of $W_n(\sI)$ can then be written as a sum
\eqn{ w = \sum_{i=0}^{n-1}V^i([a_{i,1}][t_1]+\cdots+[a_{i,d}][t_d]), }
with $a_{i,j} \in \sI$ and $[a_{i,j}], [t_j] \in W_{n-i}(\sO_P)$. By functoriality, we
have $V(a)\omega = V(aF(\omega))$ for any $a \in W_i(\sO_P)$, $\omega \in
\sH^d_Y(\WC{i+1}{d}{P/X})$, $i \geq 1$. Using \eqref{propfundclass}, we obtain
\eqn{ V^i([a_{i,j}][t_j])\gamma_{Y,n} = V^i([a_{i,j}][t_j]F^i(\gamma_{Y,n})) =
V^i([a_{i,j}][t_j]\gamma_{Y,n-i}). }
The symbol \eqref{defSymbol} is linear with respect to $m$, therefore we have 
\eqn{ [a_{i,j}][t_j]\gamma_{Y,n-i} = 
\beta_{\sJ}(\left[\begin{array}{c} [a_{i,j}][t_j]\,d[t_1]\cdots d[t_d] \\
{[t_1],\ldots,[t_d]} \end{array}\right]) = 0 }
since the upper entry in the symbol belongs to $([t_1],\ldots,[t_d])\WC{n-i}{d}{P/X}$. 

In the definition of $\gamma_{i,\pi,n}$, the last two arrows commute with $R$, $F$ and 
$V$ by functoriality. Relations \eqref{propfundclass} imply that the first one also 
commutes with $R$, $F$ and $V$, since $R(1) = F(1) = 1$, and $V(1) = p$.

Let us assume that $n = 1$, and check assertion (iii). By construction,
$\gamma_{i,\pi,1}$ is the composition of the morphism $i_*\sO_Y \to
\sH^d_Y(\Omega^d_{P/X})$ sending $1$ to $\gamma_{Y,1}$ with the canonical morphism
\eqn{ \sH^d_Y(\Omega^d_{P/X})  \riso 
\RR\uGamma_Y(\Omega^d_{P/X}[d]) \lra \Omega^d_{P/X}[d]. }
Comparing with the definition of $\gamma_f$ in \ref{Defgammaf}, and using the same 
notations, it suffices to show that the composed morphism 
\eqn{ \sO_Y \xra{\varphi_f} \omega_{Y/X} \xra{\eta_i^{-1}\,\circ\,\zeta'_{i,\pi}}
\sExt^d_{\sO_P}(\sO_Y, \Omega^d_{P/X}) \xra{\beta_{\sI}} \sH^d_Y(\Omega^d_{P/X}) }
sends $1$ to $\gamma_{Y,1}$. Since this is a morphism of sheaves (rather than 
complexes in the derived category), it is a local verification, which is provided by 
Proposition \ref{DeltaFLI}.
\end{proof}

\begin{defn}\label{Deftauipi}
Let $X$ be a noetherian $\FF_p$-scheme with a dualizing complex, $\sE$ a locally free
$\sO_X$-module of rank $d+1$, $P=\P(\sE)$, $\pi : P \to X$ the canonical projection, $i
: Y \inj P$ a regular closed immersion of codimension $d$. For each integer $n \geq 1$,
we define a trace morphism $\tau_{i,\pi,n}$ by
\eq{deftauipi}{ \tau_{i,\pi,n} : \RR f_*(W_n(\sO_Y)) \xra{\RR\pi_*(\gamma_{i,\pi,n})} 
\RR\pi_*(\WC{n}{d}{P/X}[d]) \xra{\Trp_{\pi,n}} W_n(\sO_X), }
where $\gamma_{i,\pi,n}$ is the morphism \eqref{defgamipi}, and $\Trp_{\pi,n}$ is the 
Hodge-Witt trace morphism defined in \eqref{drwTrp}. 
\end{defn}

\begin{rem}
As mentioned in the introduction, we expect that $\tau_{i,\pi,n}$ depends only on $f$, 
and not on the factorization $f = \pi \circ i$. We also expect that the analog of 
Theorem \ref{Thtau} holds for the trace morphisms $\tau_{f,n}$ that would be thus 
defined. More generally, one can hope that these constructions are part of a theory of 
canonical classes for relative de Rham-Witt cohomology (see \cite{El78}, \cite{Ek84}, 
\cite{Gs85} for such results over a field). In order to develop this program, 
generalizations and non-trivial properties of our constructions are needed (even for 
the independence statement), which would lead to expand too much this article. As most 
of them are not needed for the proof of our main results, we do not include them here, 
and we hope to return to these questions elsewhere. However, we will give in the next 
section a partial generalization of Theorem \ref{Thtau}, (iii), which is the key to the 
injectivity property of Theorem \ref{Tworeg}. 
\end{rem}

\begin{prop}\label{Proptauipi}
Under the assumptions of \ref{Deftauipi}, the morphisms $\tau_{i,\pi,n}$ satisfy the 
following properties.

\romain For variable $n$, $\tau_{i,\pi,n}$ commutes with $R$, $F$ and $V$. 

\romain For $n=1$, $\tau_{i,\pi,1} = \tau_f$. 
\end{prop}

\begin{proof}
Taking into account Proposition \ref{Taugamma}, both assertions follow from the similar
properties of $\gamma_{i,\pi,n}$ and $\Trp_{\pi,n}$ proved in \ref{Defgamipi} and
\ref{CompTrp}.
\end{proof}

\begin{defn}\label{Deftauipi2}
Under the assumptions of \ref{Deftauipi}, we can use the previous constructions 
to define a morphism $\tau_{i,\pi} : \RR f_*(W(\sO_Y)) \lra W(\sO_X)$
which commutes with $F$ and $V$, and is such that $R_n \circ \tau_{i,\pi} =
\tau_{i,\pi,n} \circ R_n$ for all $n$, $R_n$ denoting both restriction maps $W(\sO_X)
\to W_n(\sO_X)$ and $W(\sO_Y) \to W_n(\sO_Y)$. 

To construct $\tau_{i,\pi}$, we first recall that, for any scheme $X$, the inverse
system $(W_n(\sO_X))_{n\geq 0}$ is $\varprojlim$-acyclic, as the cohomology of each
term vanishes on affine open subsets, and the inverse system of sections on such a
subset has surjective transition maps. So, if $f_{\sbul\,*}$ denotes the obvious
extension of the direct image functor to the category of inverse systems, it suffices
to define a morphism
\eq{tauipibul}{ \tau_{i,\pi,\sbul} : \RR f_{\sbul\,*}(W\lbul(\sO_Y)) \lra W\lbul(\sO_X) }
in the derived category of inverse systems on $X$, and to apply the functor
$\RR\varprojlim$ and the canonical isomorphism $\RR f_* \circ \RR\varprojlim \simeq
\RR\varprojlim\,\circ\,\RR f_{\sbul\,*}$. On the one hand, the relations
$R(\gamma_{Y,n+1}) = \gamma_{Y,n}$ imply that, for variable $n$, the fundamental
classes define a morphism of inverse systems $i_{\sbul\,*}(W\lbul(\sO_Y)) \to
\sH_Y^d(\WC{\sbul}{d}{P/X})$. As the canonical morphisms
\eqn{ \sH^d_Y(\WC{\sbul}{d}{P/X}) \riso \RR\uGamma_Y(\WC{\sbul}{d}{P/X}[d]) \lra 
\WC{\sbul}{d}{P/X}[d] }
make sense in the derived category of inverse systems, we can define in this derived
category a morphism $\gamma_{i,\pi,\sbul} : i_{\sbul\,*}(W\lbul(\sO_Y)) \to
\WC{\sbul}{d}{P/X}[d]$ which has the morphisms $\gamma_{i,\pi,n}$ defined in
\eqref{defgamipi} as components. On the other hand, the homomorphisms $\dlog_n$ used to
define Chern classes for invertible bundles form an inverse system of homomorphisms,
hence, for variable $n$, the powers of the Chern classes of $\sO_P(1)$ define a
morphism $W\lbul(\sO_P)[-d] \to \RR\pi_{\sbul\,*}(\WC{\sbul}{d}{P/X})$, which is an
isomorphism of the derived category of inverse systems. Composing its inverse with the
projection by $\RR\pi_{\sbul\,*}$ of $\gamma_{i,\pi,\sbul}$ provides
$\tau_{i,\pi,\sbul}$. It is clear that $\tau_{i,\pi,\sbul}$ has the morphisms 
$\tau_{i,\pi,n}$ as components, and commutes with $F$ and $V$. Then the morphism 
\eq{deftauipi2}{ \tau_{i,\pi} : \RR f_*(W(\sO_Y)) \riso \RR\varprojlim \RR 
f_{\sbul\,*}(W\lbul(\sO_Y)) \xra{\RR\varprojlim(\tau_{i,\pi,\sbul})} W(\sO_X) }
has the required properties. 

Finally, as $f$ is a morphism of noetherian schemes, $f_*$ and $\RR f_*$ commute with
tensorisation with $\Q$. So we can define a morphism again denoted $\tau_{i,\pi} : \RR
f_*(W\sO_{Y,\Q}) \lra W\sO_{X,\Q}$ by
\eq{deftauipiQ}{ \tau_{i,\pi} :  \RR f_*(W\sO_{Y,\Q}) \riso \RR f_*(W\sO_Y)\otimes\Q 
\xra{\tau_{i,\pi}\otimes\Q} W\sO_{X,\Q}. }
This morphism also commutes with $F$ and $V$.
\end{defn}

\section{Proof of the injectivity theorem for Witt vector 
cohomology}\label{Injectivity2}

The main result of this section is Theorem \ref{ThinjWitt} below, which gives an
injectivity property for the functoriality morphisms induced on Witt vector cohomology
by some complete intersection morphisms of virtual relative dimension $0$. As explained
in Remark \ref{Implies}, Theorem \ref{Tworeg} is a particular case of this result.

\begin{thm}\label{ThinjWitt}
Let $f : Y \to X$ be a projective morphism between two flat noetherian
$\Z_{(p)}$-schemes with dualizing complexes, which is complete intersection of virtual
relative dimension $0$. We assume that there exists a scheme-theoretically dense open
subscheme $U \subset X$ such that $f^{-1}(U) \to U$ is finite locally free of constant
rank $r \geq 1$. Let $f_n : Y_n \to X_n$ be the reduction of $f$ mod $p^{n+1}$. 

\romain For all $q \geq 0$, the kernels of the functoriality homomorphisms 
\ga{}{ f^* : H^q(X, \sO_X) \lra H^q(Y, \sO_Y), \label{functO} \\
f_n^* : H^q(X_n, \sO_{X_n}) \lra H^q(Y_n, \sO_{Y_n}), \label{functOn} \\
f_0^* : H^q(X_0, W_n(\sO_{X_0})) \lra H^q(Y_0, W_n(\sO_{Y_0})), \label{functWn} \\
f_0^* : H^q(X_0, W(\sO_{X_0})) \lra H^q(Y_0, W(\sO_{Y_0})), \label{functW} }
are annihilated by $r$.

\romain For all $q \geq 0$, the functoriality homomorphism
\eq{functWQ}{ f_0^* : H^q(X_0, W\sO_{X_0,\Q}) \lra H^q(Y_0, W\sO_{Y_0,\Q}) }
is injective.
\end{thm}

\begin{rmk}\label{Implies}
Theorem \ref{ThinjWitt} implies Theorem \ref{Tworeg}. Indeed, let $f : Y \to X$ be as
in \ref{Tworeg}. The morphisms $X_k \inj X_0$ and $Y_k \inj Y_0$ are nilpotent
immersions, hence the canonical homomorphisms
\eqn{ H^q(X_0, W\sO_{X_0,\Q}) \lra H^q(X_k, W\sO_{X_k,\Q}), \quad  
H^q(Y_0, W\sO_{Y_0,\Q}) \lra H^q(Y_k, W\sO_{Y_k,\Q}) }
are isomorphisms \cite[Prop.~2.1]{BBE07}. Therefore it suffices to check that $f$
satisfies the hypotheses of Theorem \ref{ThinjWitt}. We may assume that $X$ is
connected, and replace $Y$ by one of its connected components mapping surjectively to
$X$, so that $X$ and $Y$ are integral schemes. At any closed point $y \in Y$, with
image $x = f(y)$, we may choose a closed immersion $Y \inj P$ around $y$, with $P$
smooth over $X$. If $\dim \sO_{X,x} = n$, then $\sO_{P,y}$ is a regular local ring of
dimension $n+d$ for $d = \dim(P/X)$, and $\sO_{Y,y}$ is a regular quotient of
$\sO_{P,y}$ of dimension $n$. Therefore, the ideal $\sI$ of $Y$ in $P$ is regular of
codimension $d$ around $y$, and it follows that $f$ is complete intersection of virtual
relative dimension $0$. Moreover, the function field extension $K(X) \inj K(Y)$ is
finite, hence $f$ is finite and locally free of constant rank $\geq 1$ above a non
empty open subset $U$. As $X$ is integral, $U$ is scheme-theoretically dense and the
hypotheses of Theorem \ref{ThinjWitt} are satisfied.
\end{rmk}

In order to prove Theorem \ref{ThinjWitt}, we will choose a factorization $f = \pi\circ
i$, where $i : Y \inj P = \P^{\,d}_X$ is a closed immersion, and $\pi : P \to X$ the
structural morphism. Let $i_0, \pi_0$ be the reductions mod $p$ of $i, \pi$. The key
point will be to relate the trace morphisms $\tau_{i_0,\pi_0,n}$ constructed in
\ref{Deftauipi} to the trace morphism $\tau_f$ given by Theorem \ref{Thtau}, and this
is made possible by the following constructions.

\begin{lem}\label{QuotdRW}
Let $X$ be a scheme on which $p$ is locally nilpotent, $P$ a smooth $X$-scheme, $\fa
\subset \sO_X$ a quasi-coherent ideal, $X' \inj X$ the closed subscheme defined by
$\fa$, $P' = X' \times_X P$. For each $n \geq 1$, let $\sN\hbul_n \subset
\WC{n}{\sbul}{P/X}$ be the additive subgroup generated by sections of the form
\eq{genker1}{ V^r([a]\omega), \quad dV^r([a]\omega), \quad \text{with\ \ }a \in \fa, 
\ \ \omega \in \WC{n-r}{\sbul}{P/X},\ \ 0 \leq r \leq n-1. }
Then, for variable $n$, the canonical homomorphisms $\WC{n}{\sbul}{P/X} \to
\WC{n}{\sbul}{P'/X'}$ induce a transitive family of isomorphisms
\eq{reddRW}{ \WC{n}{\sbul}{P/X}/\sN\hbul_n \riso \WC{n}{\sbul}{P'/X'}. }
\end{lem}

\begin{proof}
Thanks to \eqref{VFlin}, one first notices that $\sN\hbul_n$ is a differential graded
ideal of $\WC{n}{\sbul}{P/X}$. Using \eqref{Vd}, one sees that, for all $n \geq 1$,
$V(\sN\hbul_n) \subset \sN\hbul_{n+1}$. Using \eqref{FV} (and a direct computation for
$r = 0$), one sees that $F(\sN\hbul_{n+1}) \subset \sN\hbul_n$. Therefore, the
projective system $\{\WC{n}{\sbul}{P/X}/\sN\hbul_n\}$ is an \FV-procomplex over $P/X$.
In degree $0$, it is easy to see by induction on $n$ that the ideal $\sN^{\,0}_n \subset
W_n(\sO_P)$ is the kernel of $W_n(\sO_P) \to W_n(\sO_{P'})$. It follows that 
$\{\WC{n}{\sbul}{P/X}/\sN\hbul_n\}$ is actually an \FV-procomplex over $P'/X'$. It is 
then clear that it satisfies the universal property which defines 
$\{\WC{n}{\sbul}{P'/X'}\}$, which implies that \eqref{reddRW} is an isomorphism of 
\FV-procomplexes.
\end{proof}

\begin{prop}[{see also \cite[Th.~4.2.3]{Ol07}}]\label{DefFtilde}
Let $X$ be a $\Z_{(p)}$-scheme and denote $X_n = X\otimes_{\Z_{(p)}} \Z_{(p)}/p^{n+1}$.

\romain For all $n \geq 1$, there exists a unique homomorphism of sheaves of rings 
\eqn{ \tF^n: W_n(\sO_{X_0}) \lra \sO_{X_{n-1}} }
making the following diagram commute 
\eqn{ \xymatrix{   W_{n+1}(\sO_{X_{n-1}}) \ar[r]^-{F^n} \ar[d] & \sO_{X_{n-1}}\\
W_n(\sO_{X_0}) \ar[ur]_{\tF^n} & \hspace{1.2cm},
} }
where the vertical map is the natural reduction map. Furthermore, if we assume $X$ to
be flat over $\Z_{(p)}$ and denote by $R_n: W(\sO_{X_0}) \to W_n(\sO_{X_0})$ the
natural reduction map, then
\eq{zerointer}{ \Ker(F-\Id: W(\sO_{X_0}) \to W(\sO_{X_0})) \cap\bigcap_{n \geq 1}
\Ker( \tF^n \circ R_n)\ =\ 0. }

\romain Let $P$ be a smooth $X$-scheme and denote $P_n = P \times_X X_n$. For all $n
\geq 1$, there exists a unique homomorphism of sheaves of graded algebras
\eqn{ \tF^n:\WC{n}{\sbul}{P_0/X_0} \lra \sH\hbul(\Omega\hbul_{P_{n-1}/X_{n-1}}) }
making the following diagram commute 
\eqn{ \xymatrix{ \WC{n+1}{\sbul}{P_{n-1}/X_{n-1}} \ar[r]^-{F^n} \ar[d] & 
Z\Omega\hbul_{P_{n-1}/X_{n-1}} \ar[d] \\
\WC{n}{\sbul}{P_0/X_0} \ar[r]^-{\tF^n} & \sH\hbul(\Omega\hbul_{P_{n-1}/X_{n-1}}).
} }
Furthermore, for all $a \in \sO_{P_0}^{\times}$ and all $\tilde{a} \in
\sO^{\times}_{P_{n-1}}$ lifting $a$, we have
\eq{Ftidlog}{ \tF^n(\dlog([a])) = \cl(d\tilde{a}/\tilde{a}).  }
When $X_0$ is a perfect scheme and $X_{n-1} = W_n(X_0)$, $\tF^n$ is the isomorphism 
\eq{IRiso}{ \theta_n : \WC{n}{\sbul}{P_0/X_0} \riso
\sH\hbul({\Omega\hbul_{P_{n-1}/X_{n-1}}}) }
defined by Illusie-Raynaud \cite[III, (1.5)]{IR83}. 
\end{prop}

Note that, in formula \eqref{Ftidlog}, the class of $d\tilde{a}/\tilde{a}$ does not 
depend upon the choice of the liftng $\tilde{a}$: if $\tilde{b} = \tilde{a} + pw$, then 
\eqn{ d\tilde{b}/\tilde{b}  =  d\tilde{a}/\tilde{a} + d(\log(1 +p\frac{w}{\tilde{a}})), }
where $\log(1+pw/\tilde{a})$ is defined thanks to the canonical divided powers of $p$. 

\begin{proof}
(i) We may assume $X$ is affine. The kernel of the vertical map in the diagram is
locally generated (as an abelian group) by elements of the form $V^n([a])$ and
$V^r([pb])$ for some $a, b\in \sO_{P_{n-1}}$ and $0 \leq r \leq n$. As these elements
are clearly mapped to $0$ under $F^n$, this gives the unique existence of $\tF^n$.

To prove $\eqref{zerointer}$, let $w \in \Ker(F-\Id) \cap \bigcap_n
\Ker(\tF^n\circ R_n)$. If $w \neq 0$, we can write 
\eqn{ w =\sum_{i\ge s}^\infty V^i([a_i]),\quad \text{with } a_i\in \sO_{X_0} 
\text{ and } a_s\neq 0. }
Then $R_{s+1}(w) = V^s([a_s]) \in W_{s+1}(\sO_{X_0})$. If $\tilde{a}_s\in \sO_{X_s}$ is 
any lifting of $a_s$, and if $[\tilde{a}_s]$ is the Teichm\"uller representative of 
$\tilde{a}_s$ in $W_2(\sO_{X_s})$, so that $V^s([\tilde{a}_s])$ is a lifting of 
$V^s([a_s])$ in $W_{s+2}(\sO_{X_s})$, we have by construction 
\eqn{ \tF^{s+1}(V^s([a_s])) = F^{s+1}(V^s([\tilde{a}_s])) = p^s F([\tilde{a}_s]) = p^s 
\tilde{a}_s^p }
in $\sO_{X_s}$. Thus $\tF^{s+1}(R_{s+1}(w)) = 0$ if and only if $p^s{\tilde{a}_s}^p = 0$ 
in $\sO_{X_s}$. Since $X_s$ is flat over $\Z/p^{s+1}\Z$, we obtain $\tilde{a}_s^p \in
p\sO_{X_s}$, in particular $a_s^p=0\in \sO_{X_0}$. But by assumption we have
\eqn{ F(w)=\sum_{i\ge s} V^i([a_i^p]) = \sum_{i\ge s} V^i([a_i]) = w. }
Hence $a_s = a_s^p = 0$, a contradiction. 

(ii) First of all, since $d F^n = p^n F^n d$, the image of $F^n :
\WC{n+1}{\sbul}{P_{n-1}/X_{n-1}} \to \Omega\hbul_{P_{n-1}/X_{n-1}}$ is clearly
contained in $Z\Omega\hbul_{P_{n-1}/X_{n-1}}$. Thus the diagram makes sense. Now, 
Lemma \ref{QuotdRW} and \cite[Prop.~2.19]{LZ04} imply that, in
degree $q$, the kernel of the vertical map on the left hand side is locally generated
(as an abelian group) by sections of the following form
\eq{genker}{ V^n(\alpha),\quad dV^n(\beta),\quad V^r([p]\omega), \quad
dV^r([p]\eta), }
with $\alpha \in \Omega^q_{P_{n-1}/X_{n-1}}$, $\beta \in
\Omega^{q-1}_{P_{n-1}/X_{n-1}}$, $0 \leq r \leq n$, $\omega \in
\WC{n+1-r}{q}{P_{n-1}/X_{n-1}}$ and $\eta \in \WC{n+1-r}{q-1}{P_{n-1}/X_{n-1}}$. One
immediately sees that, via $F^n$, the first two sections are mapped to $0$ in
$\sH^q(\Omega\hbul_{P_{n-1}/X_{n-1}})$. Since, for any $m\geq 1$ and any Teichm\"uller
representative $[a]_m$ in $W_m(\sO_{P_{n-1}})$, we have $F([a]_m) = [a]_{m-1}^p \in
W_{m-1}(\sO_{P_{n-1}})$, we obtain for the last two sections 
\gan{ F^nV^r([p]\omega) = p^r F^{n-r}([p]\omega) = p^{p^{n-r}+r}F^{n-r}(\omega) = 0,\\
F^ndV^r([p]\eta) = F^{n-r}d([p]\eta) = F^{n-r}([p]d\eta) = 
p^{p^{n-r}-(n-r)}d(F^{n-r}(\eta)) = 0 }
in $\sH^q(\Omega\hbul_{P_{n-1}/X_{n-1}})$. Thus $F^n$ maps all elements in the kernel
of the vertical map to $0$ in $\sH^q(\Omega\hbul_{P_{n-1}/X_{n-1}})$. Since the
vertical map is surjective, this yields the statement.

If $\tilde{a} \in \sO^\times_{P_{n-1}}$ lifts $a$, we get by construction 
\eqn{ \tF^n(d[a]/[a]) = \cl(F^n(d[\tilde{a}]/[\tilde{a}])) = 
\cl([\tilde{a}]^{p^n-1}d[\tilde{a}]/[\tilde{a}]^{p^n}), }
which gives \eqref{Ftidlog}. 

Finally, let us assume that $X_0$ is perfect and $X_{n-1} = W_n(X_0)$. By \cite[III,
(1.5)]{IR83}, $\sH\hbul(\Omega\hbul_{P_{n-1}/X_{n-1}})$ has the structure of a
differential graded algebra (dga) with the differential $d :
\sH^i(\Omega\hbul_{P_{n-1}/X_{n-1}}) \to \sH^{i+1}(\Omega\hbul_{P_{n-1}/X_{n-1}})$
given by the boundary of the long exact cohomology sequence coming from the short exact
sequence
\eqn{ 0 \lra \Omega\hbul_{P_{n-1}/X_{n-1}} \xra{\;p^n\,} \Omega\hbul_{P_{2n-1}/X_{2n-1}} 
\lra \Omega\hbul_{P_{n-1}/X_{n-1}} \lra 0. }
The isomorphism $\theta_n$ is compatible with the differential and the product, and
induces thus an isomorphism of dga's $\theta_n : \WC{n}{\sbul}{P_0/X_0} \riso
\sH\hbul(\Omega\hbul_{P_{n-1}/X_{n-1}})$. On the other hand, it follows from the
relation $d F^n=p^n F^n d$ that the morphism $\tF^n$ is compatible with the
differentials. Therefore $\tF^n$ also induces a morphism of dga's $\tF^n :
\WC{n}{\sbul}{P_0/X_0} \riso \sH\hbul(\Omega\hbul_{P_{n-1}/X_{n-1}})$. In degree $0$, 
$\theta_n$ is defined by 
\eqn{ \theta_n(a_0, \ldots, a_{n-1}) = \tilde{a}_0^{p^n} + p\tilde{a}_1^{p^{n-1}} +\cdots + 
p^{n-1}\tilde{a}_{n-1}^p, }
where $\tilde{a}_0,\ldots,\tilde{a}_{n-1}$ are liftings to $\sO_{P_{n-1}}$ of
$a_0,\ldots,a_{n-1}$ \cite[p.~142, l.~8]{IR83}. This definition shows that, in degree
$0$, $\theta_n$ is the factorization of the $n$-th ghost component $w_n:
W_{n+1}(\sO_{P_{n-1}}) \to \sO_{P_{n-1}}$, given by $w_n(a_0,\ldots,a_n) = \sum_{i=0}^n
p^ia_i^{p^{n-i}}$, with $p^na_n=0$ in $\sO_{P_{n-1}}$. From the definition of the
morphism of functors $F^n : W_{n+1} \to W_1$, we also get that, in degree $0$, $\tF^n$
is the factorization of the $n$-th ghost component. Since $\tF^n = \theta_n$ in degree
$0$ and $\WC{n}{\sbul}{P_{n-1}/X_{n-1}}$ is generated as dga by its sections in degree
$0$, $\tF^n$ and $\theta_n$ have to be equal.
\end{proof}

\begin{lem}\label{OmBOm}
Let $S$ be $\Spec \Z_{(p)}$, $X$ an $S$-scheme, $\pi : P := \P_X^{\,d} \to X$ the
structural morphism of a projective space over $X$. For $n \geq 0$, denote by $S_n,
X_n, P_n, \pi_n$ the reductions modulo $p^{n+1}$, and let $B\Omega^d_{P_n/X_n} \subset
\Omega^d_{P_n/X_n}$ be the subsheaf of exact differential forms.

\romain For all $n \geq 0$, the canonical homomorphism 
\eq{highcoh}{ b^d_n : R^d\pi_{n\,*}(\Omega^d_{P_n/X_n}) \lra 
R^d\pi_{n\,*}(\Omega^d_{P_n/X_n}/B\Omega^d_{P_n/X_n}) }
is an isomorphism.

\romain Assume that $X$ is flat over $S$, and let $Y_0 \inj P_0$ be a regular closed
immersion of codimension $m$. Then,
\eq{purity}{ \forall\ j \neq m,\ \forall\ n \geq 0, \quad
\sH_{Y_0}^j(\Omega^d_{P_n/X_n}/B\Omega^d_{P_n/X_n}) = 0. }
\end{lem}

\begin{proof}
Let $Q = \P_S^{\,d}$, and let $T_0,\ldots,T_d$ be homogeneous coordinates on $Q$. We
define an $S$-endomorphism $\phi : Q \to Q$ by sending $T_i$ to $T_i^p$, $0 \leq i
\leq d$. By base change by $u : X \to S$, we obtain an $X$-endomorphism of $P$, for
which we will keep the notation $\phi$, as well as for its reduction mod $p^{n+1}$.

Let us fix $n \geq 0$. We can use the morphism $\phi^{n+1}$ and view
$\phi^{n+1}_*\Omega\hbul_{P_n/X_n}$ as a complex of quasi-coherent $\sO_{P_n}$-modules,
the differential of which is then $\sO_{P_n}$-linear. But $P_n$ has an open covering by
$d+1$ open subsets which are relatively affine with respect to $X_n$, and therefore
$R^d\pi_{n\,*}$ is a right exact functor on the category of quasi-coherent
$\sO_{P_n}$-modules. As $R^d\pi_{n\,*}(\Omega^{d-1}_{P_n/X_n}) = 0$, 
assertion (i) follows.

To prove assertion (ii), we use $\phi^{n+2}$ to view
$\phi^{n+2}_*\Omega\hbul_{P_n/X_n}$ as a complex of quasi-coherent $\sO_{P_n}$-modules
with an $\sO_{P_n}$-linear differential, and we claim that the sheaf of
$\sO_{P_n}$-modules
\eqn{ \sH^d(\phi^{n+2}_*\Omega\hbul_{P_n/X_n}) = 
\phi_*(\phi^{n+1}_*\Omega^d_{P_n/X_n}/B\phi^{n+1}_*\Omega^d_{P_n/X_n}) }
has a filtration by sub-$\sO_{P_n}$-modules, the graded of which is locally free over 
$\sO_{P_0}$. As $Y_0$ is locally defined in $P_0$ by a regular sequence of $m$ 
sections, the claim clearly implies assertion (ii).  

To prove the existence of this filtration, we may replace $X$, $P$ by $S$, $Q$, 
because the projection $v : P \to Q$ is flat, and 
\eqn{ v^*(\phi^{n+2}_*\Omega\hbul_{Q_n/S_n}) \riso \phi^{n+2}_*\Omega\hbul_{P_n/X_n}. }
Now $S_0$ is a perfect scheme, and $S_n = W_{n+1}(S_0)$. Thanks to the last assertion
of Proposition \ref{DefFtilde} (ii), $F^{n+1}$ defines an isomorphism of graded
algebras
\eqn{ \tF^{n+1} : \WC{n+1}{\sbul}{Q_0/S_0} \riso \sH\hbul(\Omega\hbul_{Q_n/S_n}). }
We may view $\tF^{n+1}$ as an $\sO_{Q_n}$-linear isomorphism by endowing
$\sH\hbul(\Omega\hbul_{Q_n/S_n})$ with the $\sO_{Q_n}$-module structure provided by the
homomorphism $\sO_{Q_n} \to \sH^0(\Omega\hbul_{Q_n/S_n})$ defined by $\phi^{n+2}$, and
$\WC{n+1}{\sbul}{Q_0/S_0}$ with the structure corresponding to the previous one via
$(\tF^{n+1})^{-1} : \sH^0(\Omega\hbul_{Q_n/S_n}) \riso W_{n+1}(\sO_{Q_0})$. The canonical
filtration of $\WC{n+1}{d}{Q_0/S_0}$ is then a filtration by sub-$\sO_{Q_n}$-modules, 
which can be transported to $\sH^d(\Omega\hbul_{Q_n/S_n})$ via $\tF^{n+1}$. As we know 
by \cite[I, Cor.~3.9]{Il79} that the corresponding graded pieces are locally free 
$\sO_{Q_0}$-modules for the structure defined by the homomorphism 
\eq{factorF}{ \oF : \sO_{Q_0} \lra W_{n+1}(\sO_{Q_0})/pW_{n+1}(\sO_{Q_0}) }
factorizing $F : W_{n+1}(\sO_{Q_0}) \to W_{n+1}(\sO_{Q_0})$, the proof will be 
complete if we check the commutativity of the diagram 
\eq{diagoF}{ \xymatrix@C=45pt{ \sO_{Q_n} \ar[r]^-{\phi^{n+2\,*}} \ar@{>>}[d] & 
\sH^0(\Omega\hbul_{Q_n/S_n}) \ar[r]_-{\sim}^-{(\tF^{n+1})^{-1}} & 
W_{n+1}(\sO_{Q_0}) \ar@{>>}[d] \\
\sO_{Q_0} \ar[rr]^-{\oF} & & W_{n+1}(\sO_{Q_0})/pW_{n+1}(\sO_{Q_0}).
} }

It is enough to check that the diagram induced on sections over $D_+(T_i) \subset
Q_n$ commutes, for $0 \leq i \leq d$. So we may replace $\sO_{Q_n}$ by $A =
(\Z/p^{n+1}\Z)[\ux]$, with $\ux = (x_1,\ldots,x_d$) and $\phi^*(x_j)= x_j^p$, $1 \leq j
\leq d$. Take $f=\sum_I a_I \ux^I \in A$, with $a_I \in \Z/p^{n+1}\Z$. Then
\eqn{ (\tF^{n+1})^{-1} \circ \phi^{n+2\,*}(f)\ =\ 
\sum_I a_I (\tF^{n+1})^{-1}(\ux^{p^{n+2}I}). }
As $\tF^{n+1}$ is the factorization of the $(n+2)$-th ghost component $w_{n+1} :
W_{n+2}(A) \to A$, we see that $(\tF^{n+1})^{-1}(x_j^{p^{n+1}}) = [x_j]$, $1 \leq j
\leq d$. Therefore, we obtain
\eqn{ (\tF^{n+1})^{-1} \circ \phi^{n+2\,*}(f)\ =\ \sum_I a_I [\ux]^{pI}. }
Since $\oF$ is given by lifting an element of $A_0$ to $W_{n+1}(A_0)$, applying 
Frobenius and reducing modulo $p$, this gives the commutativity of \eqref{diagoF}. 
\end{proof}

\begin{prop}\label{Mult}
Under the assumptions of Theorem \ref{ThinjWitt}, let $f = \pi \circ i$ be a
factorization of $f$ as the composition of a regular closed immersion $i : Y \inj P =
\P^{\,d}_X$ of $Y$ into a projective space on $X$, followed by the canonical projection
$\pi : P \to X$. For all $n \geq 1$, let $f_n, i_n, \pi_n$ be the reductions of $f, i,
\pi$ modulo $p^{n+1}$. Then the compositions
\ga{}{ \sO_X \xra{f^*} \RR f_*(\sO_Y) \xra{\tau_f} \sO_X, \label{compf}\\
\sO_{X_n} \xra{f_n^*} \RR f_{n\,*}(\sO_{Y_n}) \xra{\tau_{f_n}} \sO_{X_n}, \label{compfn}\\
W_n(\sO_{X_0}) \xra{f_0^*} \RR f_{0\,*}(W_n(\sO_{Y_0})) \xra{\tau_{i_0,\pi_0,n}} 
W_n(\sO_{X_0}), \label{compWnf} \\
W(\sO_{X_0}) \xra{f_0^*} \RR f_{0\,*}(W(\sO_{Y_0})) \xra{\tau_{i_0,\pi_0}} 
W(\sO_{X_0}), \label{compWf} \\
W\sO_{X_0,\Q} \xra{f_0^*} \RR f_{0\,*}(W\sO_{Y_0,\Q}) \xra{\tau_{i_0,\pi_0}} 
W\sO_{X_0,\Q}, \label{compWfQ} }
are given by multiplication by $r$. 
\end{prop}

\begin{proof}
Since the restriction of $f$ above $U$ is finite locally free of rank $r$, it follows
from \eqref{tautrace} that the endomorphism of $\sO_U$ induced by $\tau_f \circ f^*$ is
mutiplication by $r$. But $U$ is scheme-theoretically dense in $X$, therefore the same
relation holds on $X$ itself. So \eqref{compf} is multiplication by $r$.

Thanks to the flatness of $X$ and $Y$ over $\Z_{(p)}$, the spectral sequence for the
composition of $\Tor$'s implies that, for all $n \geq 1$, $X_n$ and $Y$ are
Tor-independent over $X$. Therefore, by Theorem \ref{Thtau}, (ii), the morphism
$\tau_{f_n} \circ f_n^*$ is deduced from $\tau_f \circ f^*$ by base change from $X$ to
$X_n$, and \eqref{compfn} is also multiplication by $r$.

We want to deduce from this result that \eqref{compWnf} is also multiplication by $r$. 
We observe first that the homomorphisms $\tF^n$ defined by Lemma \ref{DefFtilde} 
provide morphisms 
\gan{ \tF^n_X : W_n(\sO_{X_0}) \lra \sO_{X_{n-1}}, \\
f_*(\tF^n_Y) :  f_{0\,*}(W_n(\sO_{Y_0})) \lra f_{n-1\,*}(\sO_{Y_{n-1}}), \\
R^d\pi_*(\tF^n_P) : R^d\pi_{0\,*}(\WC{n}{d}{P_0/X_0}) \lra 
R^d\pi_{n-1\,*}(\Omega^d_{P_{n-1}/X_{n-1}}/B\Omega^d_{P_{n-1}/X_{n-1}}). }
Moreover, we can use the isomorphism \eqref{highcoh} and define 
\eqn{ \tG^n_P := (b_n^d)^{-1} \circ R^d\pi_*(\tF^n_P) : R^d\pi_{0\,*}(\WC{n}{d}{P_0/X_0}) 
\lra R^d\pi_{n-1\,*}(\Omega^d_{P_{n-1}/X_{n-1}}). }
We consider the diagram 
\eq{WtoO}{ \xymatrix{
W_n(\sO_{X_0}) \ar[d]_{\tF^n_X} \ar[r]^-{f_0^*} & f_{0\,*}(W_n(\sO_{Y_0})) 
\ar[d]_-{f_*(\tF^n_Y)} \ar[rr]^-{\pi_{0\,*}(\gamma_{i_0,\pi_0,n})} & & 
R^d\pi_{0\,*}(\WC{n}{d}{P_0/X_0}) \ar[d]^-{\tG^n_P} \ar[r]^-{\Trp_{\pi_0,n}}_-{\sim} & 
W_n(\sO_{X_0}) \ar[d]^-{\tF^n_X} \\
\sO_{X_{n-1}} \ar[r]^-{f_{n-1}^*} & f_{n-1\,*}(\sO_{Y_{n-1}}) 
\ar[rr]^-{\pi_{n-1\,*}(\gamma_{f_{n-1}})} & & R^d\pi_{n-1\,*}(\Omega^d_{P_{n-1}/X_{n-1}}) 
\ar[r]^-{\Trp_{\pi_{n-1}}}_-{\sim} & \sO_{X_{n-1}},
} }
where the compositions of the upper and lower rows are respectively the maps induced by
\eqref{compWnf} and \eqref{compfn} on degree $0$ cohomology. Let us prove that this
diagram is commutative. The left square commutes because the morphism $\tF^n_X$ is
functorial with respect to $X$. To prove that the right square commutes, it suffices to
show that, if $\xi_{\dRW}$ and $\xi_{\dR}$ are the de Rham-Witt and de Rham Chern
classes of $\sO_P(1)$, then $\xi_{\dRW}^d$ and $\xi_{\dR}^d$ have same image in
$R^d\pi_{n-1\,*}(\Omega^d_{P_{n-1}/X_{n-1}}/B\Omega^d_{P_{n-1}/X_{n-1}})$. As
$R\hbul\pi_{n-1\,*}(\tF^n_P)$ and $b_n\hbul$ are compatible with cup-products, it
suffices to show that the diagram
\eqn{ \xymatrix{
R^1\pi_{0\,*}(\sO^{\times}_{P_0}) \ar[r]^-{\dlog} & 
R^1\pi_{0\,*}(\WC{n}{1}{P_0/X_0}) \ar[d]^-{R^1\pi_*(\tF^n_P)}  \\
R^1\pi_{n-1\,*}(\sO^{\times}_{P_{n-1}}) \ar[u] \ar[r]^-{\dlog} & 
R^1\pi_{n-1\,*}(\sH^1(\Omega\hbul_{P_{n-1}/X_{n-1}})) 
} }
is commutative, which follows from \eqref{Ftidlog}.

To simplify notations, we drop the base scheme from the indices, and denote
$C^d_{P_{n-1}} = \Omega^d_{P_{n-1}}/B\Omega^d_{P_{n-1}}$. To prove the commutativity of
the central square of \eqref{WtoO}, it suffices to prove the commutativity of the
diagram
\eqn{ \xymatrix{
i_{0\,*}(W_n(\sO_{Y_0})) \ar[r] \ar[dd]_-{i_*(\tF^n_Y)} & 
\sH^d_{Y_0}(\WC{n}{d}{P_0}) \ar[d]_-{\sH^d_Y(\tF^n_P)} &
\RR\uGamma_{Y_0}(\WC{n}{d}{P_0})[d] \ar[l]_-{\sim} \ar[d]^-{\RR\uGamma_Y(\tF^n_P)[d]} 
\ar[r] & \WC{n}{d}{P_0}[d] \ar[d]^-{\tF^n_P[d]} \\
& \sH^d_{Y_{n-1}}(C^d_{P_{n-1}}) & \RR\uGamma_{Y_{n-1}}(C^d_{P_{n-1}})[d] 
\ar[l]_-{\sim} \ar[r] & C^d_{P_{n-1}}[d] \\
i_{n-1\,*}(\sO_{Y_{n-1}}) \ar[r] & \sH^d_{Y_{n-1}}(\Omega^d_{P_{n-1}}) \ar[u] & 
\RR\uGamma_{Y_{n-1}}(\Omega^d_{Y_{n-1}})[d] \ar[l]_-{\sim} \ar[u] \ar[r] & 
\,\Omega^d_{Y_{n-1}}[d]\,, \ar[u] 
} }
to apply the functor $\RR\pi_{n-1\,*}$, and to pass to cohomology sheaves in degree
$0$. In this diagram, the upper left (resp.\ lower left) horizontal arrow maps $1$ to
$\gamma_{Y_0,n}$ (resp.\ $\gamma_{Y_{n-1},1}$), and the middle horizontal arrow is an
isomorphism thanks to Lemma \ref{OmBOm} (ii). The middle and right squares commute by
functoriality, and it suffices to prove that the left rectangle commutes. This part of
the diagram comes from a diagram of morphisms of sheaves, therefore the verification is
local on $P$. Thus we may assume that $Y$ is defined by a regular sequence
$t_1,\ldots,t_d$ in $P$. Then, since $Y$ and $P$ are flat over $\Z_{(p)}$, the images
of this sequence in $\sO_{P_{n-1}}$ and $\sO_{P_0}$ (still denoted $t_1,\ldots,t_d$)
are regular sequences defining $Y_{n-1}$ and $Y_0$. It is enough to show that the
symbols
\eqn{ \begin{split} \left[\begin{array}{c} d[t_1]\cdots d[t_d] \\
{[t_1],\ldots,[t_d]} \end{array}\right] \in \sExt^d_{W_n(\sO_{P_0})}(W_n(\sO_{Y_0}), 
\WC{n}{d}{P_0}) \hspace{3cm} \\
\text{and\quad} 
\left[\begin{array}{c} dt_1\cdots dt_d \\
{t_1,\ldots,t_d} \end{array}\right] \in 
\sExt^d_{\sO_{P_{n-1}}}(\sO_{Y_{n-1}}, \Omega^d_{P_{n-1}}) \end{split} }
have same image in $\sH^d_Y(C^d_{P_{n-1}})$. By functoriality, the image of
$\left[\begin{array}{c} dt_1\cdots dt_d \\
{t_1,\ldots,t_d} \end{array}\right]$ in $\sExt^d_{\sO_{P_{n-1}}}(\sO_{Y_{n-1}},
C^d_{P_{n-1}})$ is $\left[\begin{array}{c} \cl(dt_1\cdots dt_d) \\
{t_1,\ldots,t_d} \end{array}\right]$. On the other hand, it follows from the
construction of $\tF^n$ in Proposition \ref{DefFtilde} that $\tF^n_P([t_i]) = t_i^{p^n}
\in \sO_{P_{n-1}}$, and $\tF^n_P(d[t_i]) = \cl(t_i^{p^n-1}dt_i) \in
\sH^1(\Omega\hbul_{P_{n-1}})$. Since the $t_i^{p^n}$'s form a regular sequence in
$\sO_{P_{n-1}}$, we may argue as in the proof of Proposition \ref{Propfundclass} to
show that the symbols $\left[\begin{array}{c} d[t_1]\cdots d[t_d] \\
{[t_1],\ldots,[t_d]} \end{array}\right]$ and $\left[\begin{array}{c}
\cl(t_1^{p^n-1}\cdots t_d^{p^n-1}dt_1\cdots dt_d) \\
{t_1^{p^n},\ldots,t_d^{p^n}} \end{array}\right]$ have same image in
$\sH^d_{Y_{n-1}}(C^d_{P_{n-1}})$. The wanted equality is then a consequence of Lemma
\ref{Changet}, and the commutativity of \eqref{WtoO} follows.

Returning to the homomorphism \eqref{compWnf}, we observe that it is defined by
multiplication by a section $\kappa_n$ of $W_n(\sO_{X_0})$. Proposition
\ref{Proptauipi} (i) implies that, for variable $n$, the sections $\kappa_n$ form a
compatible family under restriction, and satisfy $F(\kappa_n)=\kappa_{n-1}$. If $\kappa
= \varprojlim_n \kappa_n \in \Gamma(X_0, W(\sO_{X_0}))$, then $F(\kappa - r) = \kappa -
r$. On the other hand, the commutativity of \eqref{WtoO} implies that $\tF^n_X(\kappa_n
- r) = 0$. So, if $R_n : W(\sO_{X_0}) \to W_n(\sO_{X_0})$ is the restriction
homomorphism, we obtain that
\eqn{ \kappa - r \in \Ker(F-\Id) \cap \bigcap_{n\geq 1}\Ker(\tF^n_X \circ R_n), }
which is zero by \eqref{zerointer}. Thus $\kappa = r$, hence $\kappa_n = r$ for all $n$. 

If we now consider in the derived category of inverse systems the composition 
\eqn{ W\lbul(\sO_{X_0}) \xra{f_{0\,\sbul}^*} \RR f_{0\,\sbul\,*}(W\lbul(\sO_{Y_0})) 
\xra{\tau_{i,\pi,\sbul}} W\lbul(\sO_{X_0}), }
we obtain a morphism which has \eqref{compWnf} as component of degree $n$. Therefore, 
this composition is multiplication by $r$ on the inverse system $W\lbul(\sO_{X_0})$. 
It follows that the composition 
\eqn{ W(\sO_{X_0}) \xra{\RR\varprojlim f_{0\,\sbul}^*} 
\RR\varprojlim \RR f_{0\,\sbul\,*}(W\lbul(\sO_{Y_0})) 
\xra{\RR\varprojlim\tau_{i,\pi,\sbul}} W(\sO_{X_0}) }
is multiplication by $r$. Using the isomorphism $\RR\varprojlim\,\circ\;\RR 
f_{0\,\sbul\,*} \simeq \RR f_{0\,*} \circ \RR\varprojlim$, we obtain that 
\eqref{compWf} is multiplication by $r$. Tensoring by $\Q$ and using the commutation 
of $\RR f_{0\,*}$ with tensorisation by $\Q$, we obtain that \eqref{compWfQ} is 
multiplication by $r$.
\end{proof}

\subsection{}\label{PfinjW}\textit{Proof of Theorem \ref{ThinjWitt}}. 
The first assertion is a particular case of Theorem \ref{Injth}. To prove the other
ones, we choose a factorization $f = \pi \circ i$, where $i$ is a closed immersion of
$Y$ into a projective space $P = \P^{\,d}_X$ over $X$, and $\pi$ is the structural
morphism, and we keep the notations of the previous subsections. Applying the functor
$H^q(X_n, -)$ (resp.~$H^i(X_0,-)$), the morphisms $\tau_{f_n}$, $\tau_{i,\pi,n}$
and $\tau_{i,\pi}$ define homomorphisms 
\gan{ H^q(Y_n, \sO_{Y_n}) \xra{\tau_{f_n}} H^q(X_n, \sO_{X_n}), \\
H^q(Y_0, W_n(\sO_{Y_0})) \xra{\tau_{i,\pi,n}} H^q(X_0, W_n(\sO_{X_0})), \\
H^q(Y_0, W(\sO_{Y_0})) \xra{\tau_{i,\pi}} H^q(X_0, W(\sO_{X_0})), \\
H^q(Y_0, W\sO_{Y_0,\Q}) \xra{\tau_{i,\pi}} H^q(X_0, W\sO_{X_0,\Q}). }
Proposition \ref{Mult} implies that the composition of these homomorphisms with the 
functoriality homomorphisms defined by $f_n$ (resp.~$f_0$) is multiplication by $r$, 
and this implies Theorem \ref{ThinjWitt}.
\hfill$\Box$
\bigskip

This also completes the proof of Theorems \ref{Tworeg}, 
\ref{Wittvanish} and \ref{Main}.

\section{An example}\label{TheExample}

Because Theorem \ref{Main} was previously known in some cases, and can be proved in
some other cases without using the most difficult results of this paper, it may be
worth giving an example for which we would not know how to prove congruence 
\eqref{maincong} without using them. We give here such an example for 
each $p \geq 7$, except perhaps when $p$ is a Fermat number.

\subsection{}\label{Conditions}
We begin with a list of conditions that we want our example to satisfy. In these 
conditions, $R$, $K$ and $k$ are as in Theorem \ref{Main}, and $X$ is an $R$-scheme. 

\vspace{1mm}
\numerop{} $X$ is a regular scheme, projective and flat over $R$.  

\numerop{} $H^0(X_K, \sO_{X_K}) = K$, and $H^q(X_K, \sO_{X_K}) = 0$ for all $q \geq 1$. 

\numerop{} There exists $q \geq 1$ such that $H^q(X_k, \sO_{X_k}) \neq 0$.

\numerop{} $X$ is not a semi-stable $R$-scheme (in particular, not smooth).

\numerop{} $\dim X_K \geq 3$.

\numerop{} $X_K$ is a variety of general type.
\vspace{1mm}

Conditions (1) and (2) will ensure that $X$ satisfies the hypotheses of Theorem 
\ref{Main}. Condition (3) will ensure that we are not in the trivial situation 
described in the first paragraph of subsection \ref{Trivial}. Condition (4) will 
ensure that Theorem \ref{Semistable} does not suffice to conclude. Condition (5) will 
rule out the case of surfaces, for which Theorem \ref{Main} is already known by
\cite[Th.~1.3]{Es06}. Condition (6) rules out rationally connected varieties, for which 
Theorem \ref{Main} is also known because they satisfy the coniveau condition of 
\cite[Th.~1.1]{Es06}. It also grants that, if $X$ can be embedded as a global 
complete intersection in some projective space over $R$, then congruence 
\eqref{maincong} cannot be proved by applying Katz's theorem \cite[Th.~1.0]{Kz71} to 
$X_k$, since a smooth complete intersection in a $K$-projective space for which Katz's
$\mu$ invariant is $\geq 1$ is a Fano variety.

\begin{rmks}\label{Remks}
We begin with a few remarks that make it easier to find an example satisfying the 
previous conditions.

\romain Examples such that $\dim_k H^1(X_k, \sO_{X_k}) > \dim_K H^1(X_K, \sO_{X_K}) =
0$ have been known since Serre's construction of a counter-example to Hodge symmetry in
characteristic $p$ \cite[Prop.~16]{Se58}. The general principle behind such examples, 
which goes back to Grothendieck (see \cite[XI, 6.11, (*)]{SGA 1} over an
algebraically closed field, and \cite[Prop.~6.2.1]{Ra70} for a general statement), is 
that the datum of a torsor $Y$ on $X$ under a finite group $G$ defines a morphism $G'
\to \underline{\Pic}_{X/R}$, where $G'$ is the Cartier dual of $G$. Then, under certain
conditions, the Lie algebra of $G'_k$ can have a non-zero image in the tangent space
$H^1(X_k, \sO_{X_k})$ to $\underline{\Pic}_{X_k/k}$. The simplest case (which was the
one considered by Serre) is when $G$ is the \'etale group $\Z/p\Z$. Then the
Artin-Schreier exact sequence shows that, when the torsor $Y_k$ remains non-trivial
after extension to an algebraic closure $\ok$ of $k$, its class gives a non-zero
element in $H^1(X_{\ok}, \sO_{X_{\ok}})$, and therefore $H^1(X_k, \sO_{X_k}) \neq 0$.
This happens in particular when $Y_k$ is a complete intersection in some projective
space, since we then have $\dim_{\ok}H^0(Y_{\ok},\sO_{Y_{\ok}}) = 1$.

To simplify our quest, we will therefore replace condition (3) (and condition (5)) by
the more restrictive condition:

\vspace{1mm}
\setcounter{numero}{2}
\numerop{'} $X$ is the quotient of an hypersurface $Y$ in a projective space
$\P^n_R$ of relative dimension $n \geq 4$ over $R$ by a free action of the group
$\Z/p\Z$.
\vspace{1mm}

\romain Assume that $X$ satisfies condition (3'). Then $H^0(Y_K, \sO_{Y_K}) = K$, and
$H^q(Y_K, \sO_{Y_K}) = 0$ for $q \neq 0, n-1$. Because $\Char(K) = 0$, we have
$H^q(X_K, \sO_{X_K}) = H^q(Y_K, \sO_{Y_K})^G$. Hence, $H^0(X_K, \sO_{X_K}) = K$, and
condition (2) is satisfied if and only if $\chi(\sO_{X_K}) = 1$. As $Y_K$ is an \'etale
cover of $X_K$ of degree $p$, the Riemann-Roch-Hirzebruch formula implies that
\eq{RRH}{ \chi(\sO_{Y_K}) = p\chi(\sO_{X_K}). }
Then condition (2) is satisfied if and only if $\chi(\sO_{Y_K}) = p$. If $d$ is the
degree of the hypersurface $Y$, we obtain
\eqna{ (-1)^{n-1}(p - 1) & = & \dim_K H^{n-1}(Y_K, \sO_{Y_K}) \\
& = & \dim_K H^n(\P^n_K, \sO_{\P^n_K}(-d)) \\
& = & \dim_K H^0(\P^n_K, \sO_{\P^n_K}(d-n-1)). }
The simplest choice for checking this equation is $d-n-1 = 1$, so that we get $\dim_K
H^0(\P^n_K, \sO_{\P^n_K}(d-n-1)) = n+1$. Then we have to satisfy the conditions
\eq{numcond}{ p > 2, \quad\quad n = p-2, \quad\quad d = p.  }

Therefore, we will simplify even further our quest by replacing condition (3') by the 
following more precise condition, which implies (2), (3) and (5):

\setcounter{numero}{2}
\numerop{''} $X$ is the quotient of an hypersurface $Y$ of degree $p$ in the projective
space $\P^n_R$ of relative dimension $n = p-2$ over $R$ by a free action of the group
$\Z/p\Z$, with $p \geq 7$.

\romain Assuming that $X$ satisfies conditions (1) and (3''), then condition (6)
follows automatically. Indeed, $Y_K$ is smooth over $K$ since $\Char(K) = 0$, and its
canonical sheaf is then $\sO_{Y_K}(-n-1+d) = \sO_{Y_K}(1)$. Since $Y_K$ is an \'etale
covering of $X_K$, it is the inverse image of the canonical sheaf on $X$, which
therefore is ample too.

So it suffices for our purpose to construct an example satisfying conditions (1), (3'')
and (4).
\end{rmks}

\subsection{}\label{Studyg}
We now begin the construction of our example. Assume that $p \geq 5$, and let $E$ be
the free $\Z_{(p)}$-module $(\Z_{(p)})^p$. We denote by $e_0,\ldots,e_{p-1}$ its
canonical basis. Let $\sigma$ be a generator of $G := \Z/p\Z$. We let $\sigma$ act on
$E$ by cyclic permutation of the basis:
\eq{sigma}{ \sigma : e_0 \mapsto e_1 \mapsto \cdots \mapsto e_{p-1} (\mapsto e_0). }
Let $H \subset E$ be the hyperplane consisting of elements for which the sum of
coordinates is $0$. It is stable under the action of $G$, and we endow it with the
basis $v_1, \ldots, v_{p-1}$ defined by $v_i = e_i-e_{i-1}$. We take as projective
space the space $\P(H) \simeq \P^{p-2}_{\Z_{(p)}}$, with the induced $G$-action, and we
denote by $X_1,\ldots,X_{p-1}$ the homogeneous coordinates on $\P(H)$ defined by the
dual basis to the basis $v_1,\ldots,v_{p-1}$ of $H$. Letting $G$ act by composition on 
functions on $H$, one checks easily that the orbit of $X_1$ is described by 
\eq{orbit}{ X_1 \mapsto -X_{p-1} \mapsto X_{p-1}-X_{p-2} \mapsto X_{p-2}-X_{p-3} 
\mapsto \cdots \mapsto X_2-X_1\ (\mapsto X_1). }

Let $g_0(X_1,\ldots,X_{p-1})$ be the sum of the elements of the orbit of $X_1^p$, i.e., 
\eq{defg0}{ g_0(X_1,\ldots,X_{p-1}) = X_1^p + (-X_{p-1})^p +
\sum_{i=2}^{p-1}(X_i-X_{i-1})^p. }
Then $g_0 \in p\Z[X_1,\ldots,Z_{p-1}]$, and we can define a polynomial 
$g(X_1,\ldots,X_{p-1}) \in \Z[X_1,\ldots,Z_{p-1}]$ by 
\eq{defg}{ g(X_1,\ldots,X_{p-1}) = \frac{1}{p}g_0(X_1,\ldots,X_{p-1}). }

Let $Z \subset \P(H)$ be the hypersurface defined by $g$. Since $g$ is $G$-invariant,
the action of $G$ on $\P(H)$ induces an action on $Z$. We 
denote by $\bg$ the reduction of $g$ in $\FF_p[X_1,\ldots,X_{p-1}]$. We first study the 
singular points of $Z_{\FF_p}$. They are solutions of the system of homogeneous equations 
$\del\bg/\del X_i = 0$, $1 \leq i \leq p-1$, which can be written as
\eq{syst}{\left\{ \begin{array}{r@{\ =\ }l}X_1^{p-1} & (X_2-X_1)^{p-1} \\
(X_2-X_1)^{p-1} & (X_3-X_2)^{p-1} \\
\vdots\hspace{8mm} & \hspace{8mm}\vdots \\
(X_{p-1}-X_{p-2})^{p-1} & (-X_{p-1})^{p-1}\ \ . \end{array} \right.}

\begin{lem}\label{SinginZ}
Let $\Fpb$ be an algebraic closure of $\FF_p$. 

\romain The solutions of \eqref{syst} in $\P^n(\Fpb)$ belong to $\P^n(\FF_p)$, and they
correspond bijectively to the families $(u_1,\ldots,u_{p-1}) \in
(\FF_p^{\times})^{p-1}$ such that
\eq{condu}{ 1 + u_1 + \cdots + u_{p-1} = 0. }

\romain For $u \in \FF_p^{\times}$, let $\tu = [u] \in \mu_{p-1}(\Z_p)$ be its
Teichm\"uller representative. Then a point $x \in \P^n(\FF_p)$ which is a solution of
\eqref{syst} belongs to $Z_{\FF_p}$ if and only if
\eq{condut}{ 1 + \tu_1 + \cdots + \tu_{p-1} \in p^2\Z_p, }
where $(u_1,\ldots,u_{p-1}) \in (\FF_p^{\times})^{p-1}$ corresponds to $x$ by 
\normalfont{(i)}. 
\end{lem}

\begin{proof}
Given $(u_1,\ldots,u_{p-1}) \in (\FF_{p}^{\times})^{p-1}$ satisfying \eqref{condu}, the
corresponding solution $x = (\xi_1:\ldots:\xi_{p-1}) \in \P^n(\FF_p)$ of the system
\eqref{syst} is obtained by choosing $\xi_1 \in \FF_p^{\times}$, setting
\eq{defxi}{ \xi_i-\xi_{i-1} = u_{i-1}\xi_1 \quad \text{for } 2 \leq i \leq p-1, }
and observing that \eqref{condu} implies that $-\xi_{p-1} = u_{p-1}\xi_1$. 
Assertion (i) is then straightforward. 

Let $\eta_1 \in \Z_p$ be a lifting of $\xi_1$, and let $\eta_i$ be defined inductively 
for $2 \leq i \leq p-1$ by 
\eq{defeta}{ \eta_i - \eta_{i-1} = \tu_{i-1}\eta_1. }
Define $\alpha \in \Z_p$ by 
\eq{defalpha}{ 1 + \tu_1 + \cdots + \tu_{p-1} = p\alpha. }
Then we get by adding the equations in \eqref{defeta}
\eq{etapm1}{ \eta_{p-1} = (1 + \cdots + \tu_{p-2}) \eta_1 = (p\alpha-\tu_{p-1})\eta_1. }
We can now substitute \eqref{defeta} and \eqref{etapm1} in $g_0$, and we obtain the 
relation 
\eqa{}{ g_0(\eta_1,\ldots,\eta_{p-1}) & = & \eta_1^p(1+\tu_1^p+\ldots+\tu_{p-2}^p 
+(\tu_{p-1}-p\alpha)^p) \notag \\ 
& = & \eta_1^p (1+\tu_1+\ldots+\tu_{p-2}+ \tu_{p-1} + \sum_{j=1}^p
{p \choose j}\tu_{p-1}^{p-j}(-p\alpha)^j) \label{keycong} \\
& \equiv & p\alpha\eta_1^p \mod p^2\Z_p. \notag}
Hence we get 
\eq{congg}{ g(\eta_1,\ldots,\eta_{p-1}) \equiv \alpha\eta_1^p \mod p\Z_p, } 
and assertion (ii) follows. 
\end{proof}

\begin{lem}\label{NofixonZ}\hspace{-8mm}
\romain The action of $G$ on $Z_{\FF_p}$ is free.

\romain If $p$ is not a Fermat number, then $Z_{\FF_p}$ is singular, and is not the special 
fibre of a semi-stable scheme.
\end{lem}

Let us recall that the Fermat numbers are the integers of the form $2^{2^n}+1$ with $n
\geq 0$, that any prime number of the form $2^n+1$ with $n>0$ is a Fermat number, and
that the only known prime Fermat numbers are $3$, $5$, $17$, $257$ and $65537$.

\begin{proof}
Over $\Fpb$, the matrix of the action of $\sigma$ on $(\Fpb)^p$ has $1$ as unique 
eigenvalue, with a corresponding eigenspace of dimension $1$, generated by the 
eigenvector $(1,\ldots,1)$. This eigenvector belongs to $H/pH$, where it has 
coordinates $(p-1,p-2,\ldots,1) = -(1,\ldots,p-1)$ in the basis
$v_1,\ldots,v_{p-1}$. Therefore, the only fixed point of $\sigma$ in $\P^n(\Fpb)$ is
the point $x_0 = (1:2:\ldots:p-1)$. This point is the solution of \eqref{syst}
corresponding to $u_1 = \ldots = u_{p-1} = 1$. Lemma \ref{SinginZ} (ii) implies that it
does not belong to $Z_{\FF_p}$, which proves assertion (i).

As the system \eqref{syst} has only a finite number of solutions, the singular points
of $Z_{\FF_p}$ are isolated. In particular, since $\dim Z_{\FF_p} \geq 4$, $Z_{\FF_p}$
cannot be the special fibre of a semi-stable scheme if it has a singular point. To find
a singular point on $Z_{\FF_p}$, Lemma \ref{SinginZ} shows that it suffices to construct a
family $(\tu_i)_{1\leq i\leq p-1}$ of $(p-1)$-th roots of unity in $\Z_p$ such that $1
+ \sum_i \tu_i \in p^2\Z_p$. Since $p$ is not a Fermat number, $p-1$ has an odd prime
factor $q$. We can choose a primitive $q$-th root of unity $\zeta$, and set $\tu_i =
\zeta^i$ for $1 \leq i \leq q-1$, $\tu_i = 1$ for $q \leq i \leq q + (p-q)/2 -1$,
$\tu_i = -1$ for $q + (p-q)/2 \leq i \leq p-1$. So $Z_{\FF_p}$ is singular.
\end{proof}

\subsection{}\label{Deff}
We now address the regularity condition in \ref{Conditions} (1). We replace $Z$ by
another equivariant lifting of $Z_{\FF_p}$ defined as follows. Let $R$ be the ring of
integers of a finite extension $K$ of $\Q_p$, of degree $>1$, with residue field $k$.
If $K/\Q_p$ is unramified, we set $\pi = p$, otherwise we choose a uniformizer $\pi$ of
$R$. Let $\lambda \in R$ be an element satisfying the following condition:

\alphab If $K/\Q_p$ is unramified, then the reduction of $\lambda$ mod $p$ does not 
belong to $\FF_p$;

\alphab If $K/\Q_p$ is ramified, then $\lambda \in R^{\times}$.

\noindent Let $h \in 
\Z[X_1,\ldots,X_{p-1}]$ be the product of the elements of the orbit of $X_1$, i.e., 
\eq{defh}{ h(X_1,\ldots,X_{p-1}) = X_1(-X_{p-1})\prod_{i=2}^{p-1}(X_i-X_{i-1}), }
and let $f \in R[X_1,\ldots,X_{p-1}]$ be defined by 
\eq{deff}{ f = g + \pi\lambda h. }

We define $Y \subset \P^n_R$ to be the hypersurface with equation $f$. Since $f$ is
invariant under $G$, the action of $G$ on $\P^n_R$ induces an action on $Y$. Its
special fibre $Y_k$ is equal to $Z_k$, on which $G$ acts freely by Lemma
\ref{NofixonZ}. Then the fixed locus of $\sigma$ is a closed subscheme of $Y$, and its
projection on $\Spec R$ is a closed subset which does not contain the closed point.
Therefore it is empty, and the action of $G$ on $Y$ is free. We define $X$ to be
the quotient scheme $X = Y/G$.

\begin{prop}\label{ProofofEx}
Assume that $p$ is an odd prime which is not a Fermat number. Then the scheme $X$
defined above satisfies conditions {\normalfont{(1) - (6)}} of\,
\normalfont{\ref{Conditions}}.
\end{prop}

\begin{proof}
As observed in \ref{Remks} (iii), it suffices to check that $X$ satisfies conditions 
(1), (3'') and (4), and condition (3'') holds by construction. 

The hypersurface $Y$ is projective and flat over $R$, since $g$ is not divisible by
$\pi$. So $X$ is also projective and flat. As $Y_k = Z_k$, Lemma \ref{NofixonZ} (ii)
implies that $Y$ is not semi-stable. Since $Y \to X$ is \'etale and semi-stability is a
local property for the \'etale topology, $X$ is not semi-stable either. So we only have
to prove that $X$ is regular. This is again a local property for the \'etale topology,
hence it suffices to prove that $Y$ is regular. Because $Y$ is excellent, its singular
locus is closed, and the same holds for its projection to $\Spec R$. So it is enough to
check the regularity of $Y$ at the points of its special fibre. The regularity is clear
at the smooth points of $Y_k$, and we need to prove it at the singular points.

Let $x = (\xi_1:\ldots:\xi_{p-1}) \in \P^n(k)$ be a singular point of $Y_k$. As $Y_k =
Z_k$, $x$ corresponds by Lemma \ref{SinginZ} to a family $(u_1,\ldots,u_{p-1}) \in
(\FF_p^{\times})^{p-1}$ such that
\begin{equation}\label{tui1}
1 + \tu_1 + \cdots + \tu_{p-1} = p^2\beta
\end{equation}
for some $\beta \in \Z_p$. We have seen in the proof of Lemma \ref{SinginZ} that 
$\xi_1 \in \FF_p^{\times}$, so we may assume that $\xi_1 = 1$. We set $\eta_1 = 1$, 
and we define inductively $\eta_i$ for $2 \leq i \leq p-1$ by \eqref{defeta}. This 
allows to work on the affine space $\A^n_R = D_+(X_1) \subset \P^n_R$, and we will 
denote 
\eqn{ a_*(X_2,\ldots,X_{p-1}) := a(1,X_2,\ldots,X_{p-1}) }
for any homogeneous polynomial $a(X_1,\ldots,X_{p-1}) \in R[X_1,\ldots,X_{p-1}]$. For 
$2 \leq i \leq p-1$, we set 
\eqn{ X_i = \eta_i + Y_i, }
so that $(\pi,Y_2,\ldots,Y_{p-1})$ is a regular sequence of generators of the maximal 
ideal $\fm_x$ of the regular local ring $\sO_{\A^n_R,x}$. 

We want to prove that $\sO_{\A^n_R,x}/(f_*)$ is regular, i.e., that $f_* \notin 
\fm_x^2$. We first claim that 
\eq{conggstar}{ g_* \equiv p\beta \mod \fm_x^2. }
Indeed, applying \eqref{keycong} with $\alpha = p\beta$, we obtain the congruence 
\eqn{ g_{0\,*}(\eta_2,\ldots,\eta_{p-1}) \equiv p^2\beta \mod p^3\Z_p, }
hence 
\eq{cong1}{ g_*(\eta_2,\ldots,\eta_{p-1}) \equiv p\beta \mod p^2\Z_p \subset \fm_x^2.}
On the other hand, equations \eqref{defeta} show that, for $2 \leq i \leq p-2$, 
\eq{cong2}{ \frac{\del g_*}{\del X_i}(\eta_2,\ldots,\eta_{p-1}) = 0. }
Finally, equations \eqref{defeta} and \eqref{etapm1} show that 
\eqa{cong3}{
\frac{\del g_*}{\del X_{p-1}}(\eta_2,\ldots,\eta_{p-1}) & = & 
(\eta_{p-1}-\eta_{p-2})^{p-1} - \eta_{p-1}^{p-1} 
\notag \\
& = & 1 - (p^2\beta - \tu_{p-1})^{p-1} \notag \\
& \equiv & 0 \mod p^2\Z_p \subset \fm_x^2. }
Applying \eqref{cong1}, \eqref{cong2} and \eqref{cong3} to the Taylor development of 
$g_*$ proves \eqref{conggstar}. 

From the definition of $h$, we obtain 
\eq{cong4}{ h_*(\eta_2,\ldots,\eta_{p-1}) = -(p^2\beta - \tu_{p-1}) 
\prod_{i=1}^{p-2}\tu_i \equiv \prod_{i=1}^{p-1}\tu_i \mod \fm_x. }
As $h_* \equiv h_*(\eta_2,\ldots,\eta_{p-1})$ mod $\fm_x$, $f_*$ satisfies the congruence 
\eq{cong4}{ f_* = g_* + \pi\lambda h_* \equiv \pi(\frac{p}{\pi}\beta + 
\lambda\prod_{i=1}^{p-1}\tu_i) \mod \fm_x^2. }
Let $w =\frac{p}{\pi}\beta + \lambda\prod_i\tu_i$. If $K/\Q_p$ is ramified, then
condition \ref{Deff} b) implies that $w$ is a unit. If $K/\Q_p$ is unramified, then 
$\pi=p$, and condition \ref{Deff} a) implies that the reduction mod $p$ of $w$ is 
non-zero, hence $w$ is again a unit. In each case, $f_* \notin \fm_x^2$, and $\sO_{Y,x}$ 
is regular.
\end{proof}

\begin{center}\textbf{Appendix: Complete intersection morphisms of virtual relative
dimension $0$}\end{center}
\addtocontents{toc}{Appendix: Complete intersection morphisms of virtual relative
dimension $0$}

\appendix

As mentioned in the introduction, we explain here the construction of the morphism
$\tau_f : \RR f_*\sO_Y \to \sO_X$ for a proper complete intersection morphism $f : Y
\to X$ of virtual dimension $0$, and we give a proof of Theorem \ref{Thtau}. 

The Appendix consists of two sections. In section A, we recall the construction of the
invertible sheaf $\omega_{Y/X}$ associated to a complete intersection morphism $f : Y
\to X$, and we prove some of its properties. We do not use duality theory here, even if
we keep for convenience the terminology ``relative dualizing sheaf''. Instead, we use
the complete intersection assumption to deduce our constructions from the elementary
properties of smooth morphisms and regular immersions, thanks to the canonical
isomorphisms defined by Conrad \cite[2.2]{Co00}. It is then easy to define the
canonical section $\delta_f$ of $\omega_{Y/X}$ when $f$ has virtual relative dimension
$0$, and to prove its basic properties.

In section B, we assume that $X$ is noetherian and has a dualizing complex. We then use
duality theory and the identification $\omega_{Y/X} \riso f^!\sO_X$ to deduce $\tau_f$
from the canonical section $\delta_f$. To translate the properties of $\delta_f$ into
the properties of $\tau_f$ listed in Theorem \ref{Thtau}, we need to use the
fundamental identifications of duality theory, as well as the various compatibilities
between these identifications. Our proofs rely in an essential way on Conrad's
exposition \cite{Co00}.

It may be worth pointing out that we need in this article the compatibility of $\tau_f$
with base change in a context which is not covered by the base change results of
\cite{Co00}. Indeed, we consider morphisms $f$ which are not flat in general (such as
in Theorem \ref{Tworeg}), and base change morphisms which are not flat either (such as
reduction mod $p^n$ in the proof of Proposition \ref{Mult}). The key property we use
here, which is familiar to the experts, but not so well documented in the literature,
is the Tor-independence of $f$ and the base change morphism.

\begin{center}\textsc{A. The canonical section of the
relative dualizing sheaf}\end{center}
\addcontentsline{toc}{section}{A.\hspace{3mm}The canonical section of the
relative dualizing sheaf}\setcounter{section}{1}

We recall now the construction of the invertible sheaf $\omega_{Y/X}$ for a complete
intersection morphism, and we explain some of its properties. As often, the main work
is to prove that the constructions are well-defined, and in particular to check the
sign conventions. As the details are easy but tedious, we leave most of them as
exercises, and only sketch the main steps of the verifications.

We first recall a standard base change result for complete intersection morphisms (see 
\cite[3.7.1]{SGA 6} for the finite Tor-dimension of $\RR f_*\sE\hbul$).

\begin{prop}\label{ci}
Let $f : Y \to X$ be a complete intersection morphism of virtual relative dimension 
$m$, and let
\eq{cart}{\xymatrix{Y' \ar[r]^v\ar[d]_{f'} & Y\ar[d]^f\\
X' \ar[r]^u & X}}
be a cartesian square such that $X'$ and $Y$ are Tor-independent over $X$. 

\romain The morphism $f'$ is a complete intersection morphism of virtual relative 
dimension $m$. 

\romain Assume that $X$ is quasi-compact, and that $f$ is separated of finite type. If
$\sE\hbul \in \Dbqc(\sO_Y)$ is of finite Tor-dimension over $\sO_Y$, then $\RR
f_*\sE\hbul$ is of finite Tor-dimension over $\sO_X$, and the base change morphism
\eq{basechange}{\LL u^* \RR f_*\sE\hbul \lra \RR f'_*\LL v^*\sE\hbul}
is an isomorphism.
\end{prop}

\begin{proof}
The first claim is local on $Y'$, so we may assume that there exists a factorization $f
= \pi \circ i$ such that $\pi : P \to X$ is a smooth morphism of relative dimension
$n$, and $i : Y \inj P$ is a closed immersion of codimension $d = n-m$. Then $i$ is a
regular immersion, defined by an ideal $\sI \subset \sO_P$, and, since the claim is
local, we may assume that $\sI$ is generated by a regular sequence $t_1,\ldots,t_d$ of
sections of $\sO_P$. Then the Koszul complex $K\lbul(t_1,\ldots,t_d)$ is a resolution
of $\sO_Y$ by $\sO_P$-modules which are flat relatively to $X$. Let $P'=X'\times_X P$,
and let $t'_1,\ldots,t'_d$ be the images of $t_1,\ldots,t_d$ in $\sO_{P'}$. Since $X'$
and $Y$ are Tor-independent over $X$, the Koszul complex $K(t'_1,\ldots,t'_d)$ is a
resolution of $\sO_{Y'}$ over $\sO_{P'}$, which shows that $f'$ is a complete
intersection morphism of virtual relative dimension $m$.

Assume now that the hypotheses of (ii) are satisfied. Since $X$ is quasi-compact, it
suffices to check that $\RR f_*\sE\hbul$ is of finite Tor-dimension when $X$ is affine.
We can then choose a finite covering $\fU$ of $Y$ by affine open subsets $U_\alpha$,
and we may assume that the $U_\alpha$ are small enough so that the restriction
$f_\alpha$ of $f$ to $U_\alpha$ can be factorized as $f_\alpha = \pi_\alpha \circ
i_\alpha$, where $\pi_\alpha : P_\alpha \to X$ is smooth and $i_\alpha : U_\alpha \inj
P_\alpha$ is a closed immersion defined by a regular sequence of sections of
$\sO_{P_\alpha}$. For each sequence $\alpha_0 < \cdots < \alpha_r$, denote $U_{\ualpha}
= U_{\alpha_0} \cap \cdots \cap U_{\alpha_r}$, $j_{\ualpha} : U_{\ualpha} \inj Y$, and
let $f_{\ualpha}$ be the restriction of $f$ to $U_{\ualpha}$. If $\sI\hbul$ is an
injective resolution of $\sE\hbul$, then the alternating \v{C}ech complex
$\vC\hbul(\fU,\sI\hbul)$ is a resolution of $\sE\hbul$. Since $j_{\ualpha}$ is an
affine open immersion, the complex $j_{\ualpha\,*}j^*_{\ualpha}\sI\hbul = \RR
j_{\ualpha\,*}j^*_{\ualpha}\sE\hbul$ belongs to $\Dqctdf(\sO_Y)$ for each $\ualpha$.
Therefore it suffices to prove that $\RR f_*\sE\hbul \in \Dqctdf(\sO_X)$ for complexes
$\sE\hbul$ of the form $\RR j_*\sF\hbul$, where $j$ is the inclusion of an affine open
subscheme $U$, and $\sF\hbul \in \Dqctdf(\sO_U)$. This reduces the proof to the case
where $Y$ is affine. Then there exists a bounded complex of $\sO_Y$-modules $\sP\hbul$
with flat quasi-coherent terms, and a quasi-isomorphism $\sP\hbul \to \sE\hbul$. Since
$\sO_Y$ has finite Tor-dimension over $\sO_X$, so does any flat $\sO_Y$-module, and the
first assertion of (ii) follows.

The complex $\LL v^*\sE\hbul$ belongs to $\Dqctdf(\sO_{Y'})$, and the base change
morphism \eqref{basechange} can be defined by adjunction as usual. Arguing as before,
it suffices to prove that it is an isomorphism when $X$ is affine and $\sE\hbul$ is of
the form $\RR j_*\sF\hbul$, where $j$ is the inclusion of an affine open subscheme $U
\subset Y$, and $\sF\hbul \in \Dqctdf(\sO_U)$. Let $U' = X' \times_X U$, and let $w :
U' \to U$ be the projection, $j' : U' \inj Y'$ the pull-back of $j$. Since $j$ is an
affine morphism and $\sF\hbul \in \Dqctdf(\sO_U)$, the base change morphism $\LL v^*\RR
j_*\sF\hbul \to \RR j'_*\LL w^*\sF\hbul$ is an isomorphism. This implies that the base
change morphism for $f$ and $\sE\hbul$ is an isomorphism if and only if the base change
morphism for $f\circ j$ and $\sF\hbul$ is an isomorphism. If one chooses a bounded,
flat, quasi-coherent resolution $\sP\hbul$ of $\sF\hbul$, the Tor-independence
assumption implies that, for each $n$, $(f\circ j)_*\sP^n$ is $u^*$-acyclic. It follows
easily that the base change morphism for $\sP\hbul$ is an isomorphism, which ends the
proof.
\end{proof}

\begin{rem}
Assertion (ii) holds more generally if one replaces the complete intersection 
hypothesis on $f$ by the assumption that $\sE\hbul$ has finite Tor-dimension over 
$\sO_X$. It is also standard to extend the assertion to the case where $f$ is only 
assumed to be coherent, i.e., quasi-compact and quasi-separated.
\end{rem}

\subsection{}\label{Defomega}
Let $f : Y \to X$ be a complete intersection morphism of relative dimension $m$. We 
now recall how one can associate to $f$ an invertible $\sO_Y$-module $\omega_{Y/X}$,
called the \textit{relative dualizing sheaf} of $Y/X$ (or $f$). We will use here the 
direct construction based on elementary algebra\footnote{\ For a more intrinsic 
construction, one can use the general properties of the cotangent complex $L_{Y/X}$
\cite{Il71}. Here, $L_{Y/X}$ is a perfect complex, of perfect amplitude in $[-1,0]$,
and of rank $m$ \cite[3.2.6]{Il71}. Taking its (graded) determinant in the sense of
Knudsen-Mumford \cite{KM76}, one obtains the complex $\omega_{Y/X}[m]$. Special
attention should be paid in this construction to sign compatibilities, as, for
historical reasons, the sign conventions used in \cite{Il71} and \cite{KM76} conflict
with those of \cite{Co00}.}, which makes explicit the existence of the canonical
section when $m=0$, and is a natural extension of Conrad's constructions for the
canonical isomorphisms $\zeta'_{f,g}$ \cite[2.2]{Co00}.

If $f = \pi \circ i$ is a factorization of $f$ where $\pi : P \to X$ is a smooth
morphism of relative dimension $n$ and $i : Y \inj P$ a closed immersion of codimension
$d = n-m$, defined by a regular ideal $\sI \subset \sO_P$, then one defines 
$\omega_{Y/X}$ by setting
\begin{eqnarray}\label{defomega}
\omega_{Y/X} & = & \omega_{Y/P} \otimes_{\sO_Y} i^*\omega_{P/X} \\
& = & \wedge^d((\sI/\sI^2)^\vee) \otimes_{\sO_Y} i^*\Omega^n_{P/X}.\notag
\end{eqnarray}
Up to canonical isomorphism, this construction is made independent of the choice of the
factorization as follows. Let $f = \pi' \circ i'$ be another factorization of $f$
through a smooth morphism $\pi' : P' \to X$, and let $\omega^P_{Y/X}$ and
$\omega^{P'}_{Y/X}$ be the invertible $\sO_Y$-modules defined by \eqref{defomega} using
the two factorizations. Assume first that there exists an $X$-morphism $u : P' \to P$
such that $u \circ i' = i$, and which is either a smooth morphism or a regular
immersion. Then, one defines an isomorphism $\varepsilon^{P',P}(u) : \omega^P_{Y/X}
\riso \omega^{P'}_{Y/X}$ by the commutative diagram
\eq{defepsi}{\xymatrix@C=40pt{\omega^P_{Y/X} = \omega_{Y/P} \otimes i^*\omega_{P/X} 
\ar[r]_-{\sim}^-{\zeta'_{i',u}\,\otimes\,\Id} \ar[dr]^{\sim}_{\varepsilon^{P',P}(u)} & 
\omega_{Y/P'} \otimes i'{}^*\omega_{P'/P} \otimes i'{}^*u^*\omega_{P/X} \\
& \omega_{Y/P'} \otimes i'{}^*\omega_{P'/X} = \omega^{P'}_{Y/X}. 
\ar[u]^{\sim}_{\Id\,\otimes\,i'^*(\zeta'_{u,\pi})}
}}
The definitions of $\zeta'_{i',u}$ and $\zeta'_{u,\pi}$ depend upon whether $u$ is a 
smooth morphism or a regular immersion (the two definitions agree when $u$ is an 
open and closed immersion):

\alphab If $u$ is smooth, then $\zeta'_{i',u}$ is defined by \cite[p.\ 29, (d)]{Co00},
and $\zeta'_{u,\pi}$ is defined by \cite[p.\ 29, (a)]{Co00}.

\alphab If $u$ is a regular immersion, then $\zeta'_{i',u}$ is defined by \cite[p.\ 29,
(b)]{Co00}, and $\zeta'_{u,\pi}$ is defined by \cite[p.\ 29, (c)]{Co00}.

Let $f = \pi''\circ i''$ be a third factorization of $f$ through a smooth morphism
$\pi'' : P'' \to X$, let $\omega^{P''}_{Y/X}$ be defined by \eqref{defomega} using this
factorization, and assume that there exists an $X$-morphism $v : P'' \to P'$ such that
$v\circ i'' = i'$ and such that each of the morphisms $v$ and $u\circ v$ is either a
smooth morphism or a regular immersion. Then it follows readily from Conrad's general
transitivity relation for compositions of smooth morphisms and regular immersions
\cite[(2.2.4)]{Co00} that
\eq{transepsi}{\varepsilon^{P'',P'}(v) \circ \varepsilon^{P',P}(u) = 
\varepsilon^{P'',P}(u\circ v).}

If $f = \pi\circ i = \pi'\circ i'$ are any factorizations as above, let now $P'' = P' 
\times_X P$, and let $i'' : Y \inj P''$ be the diagonal immersion, and $q : P'' \to P, 
q' : P'' \to P'$ the two projections. One defines the canonical isomorphism 
$\varepsilon^{P',P} : \omega^P_{Y/X} \xra{\sim} \omega^{P'}_{Y/X}$  by setting
\eq{defepsi2}{\varepsilon^{P',P} := \varepsilon^{P'',P'}(q')^{-1} \circ 
\varepsilon^{P'',P}(q).} 
Whenever there exists a smooth morphism or a regular immersion $u : P' \to P$ as 
above, it follows from \eqref{transepsi} that $\varepsilon^{P',P}(u) = 
\varepsilon^{P',P}$. One checks similarly that the isomorphisms $\varepsilon_{P',P}$ 
satisfy the usual cocycle condition for three factorizations. 

Thanks to this cocycle condition, one can then define the invertible sheaf
$\omega_{Y/X}$ even when there does not exist a global factorization $f = \pi \circ i$
as above, by choosing local factorizations and glueing the invertible sheaves obtained
locally by the previous construction. By construction, the sheaf $\omega_{Y/X}$ 
commutes with Zariski localization, and is equipped with a canonical isomorphism for 
which we keep the notation $\zeta'$:
\eq{transomega1}{ \zeta'_{\pi,i} : \omega_{Y/X} \riso 
\omega_{Y/P} \otimes_{\sO_Y} i^*\omega_{P/X},}
for any factorization $f = \pi \circ i$ where $\pi$ is a smooth morphism and $i$ is a 
regular immersion.

If $m$ is the virtual relative dimension of $Y$ over $X$, we will need to work with the
complex $\omega_{Y/X}[m]$ which is the single $\sO_Y$-module $\omega_{Y/X}$ sitting in
degree $-m$. If $f = \pi \circ i$ as above, we define in $\Db(\sO_Y)$ the isomorphism
of complexes
\eq{defzeta}{\zeta'_{i,\pi} : \omega_{Y/X}[m] \riso \omega_{Y/P}[-d] \otimesL_{\sO_Y} 
\LL i^*(\omega_{P/X}[n])}
by \eqref{transomega1} in degree $-m$, without any sign modification. If $f$ is a
smooth morphism or a regular immersion, this definition is consistent with
\cite[(2.2.6)]{Co00}. By \cite[(1.3.6)]{Co00}, the isomorphism \eqref{defzeta} is equal
to the composed isomorphism
\eqna{\omega_{Y/X}[m] \xrightarrow[{\eqref{transomega1}[m]}]{\sim} (\omega_{Y/P}
\otimesL_{\sO_Y} \LL i^*(\omega_{P/X}))[m]
& \riso & (\omega_{Y/P} \otimesL_{\sO_Y} \LL i^*(\omega_{P/X}[n]))[-d] \\
& \riso & \omega_{Y/P}[-d] \otimesL_{\sO_Y} \LL i^*(\omega_{P/X}[n])}
and differs from the composed isomorphism 
\eqna{\omega_{Y/X}[m] \xrightarrow[{\eqref{transomega1}[m]}]{\sim} (\omega_{Y/P}
\otimesL_{\sO_Y} \LL i^*(\omega_{P/X}))[m]
& \riso & (\omega_{Y/P}[-d] \otimesL_{\sO_Y} \LL i^*(\omega_{P/X}))[n] \\
& \riso & \omega_{Y/P}[-d] \otimesL_{\sO_Y} \LL i^*(\omega_{P/X}[n])}
by multiplication by $(-1)^{dn}$.

\begin{lem}\label{basechangeomega}
Under the assumptions of Proposition \ref{ci}, there exists
a canonical isomorphism
\eq{pbomega}{
\LL v^*(\omega_{Y/X}) \cong v^*(\omega_{Y/X}) \riso \omega_{Y'/X'}. }
Moreover, if the assumptions of Proposition \ref{ci} (ii) are satisfied, the
canonical base change morphism
\eq{bcomega}{ \LL u^*\RR f_*(\omega_{Y/X}) \to \RR f'_*(\omega_{Y'/X'}).}
is an isomorphism.
\end{lem}

\begin{proof}
Since $\omega_{Y/X}$ is invertible, $\LL v^*(\omega_{Y/X}) \riso v^*(\omega_{Y/X})$. To
prove the isomorphism \eqref{pbomega}, assume first that there exists a factorization
$f = \pi\circ i$ where $\pi$ is smooth and $i$ is a regular immersion. Let $f' =
\pi'\circ i'$ be the factorization deduced from $f = \pi\circ i$ by base change. Then,
if $\sI$ and $\sI'$ are the ideals defining $i$ and $i'$, the Tor-independence
assumption implies that the canonical homomorphism $u^*(\sI/\sI^2) \to \sI'/\sI'{^2}$
is an isomorphism, which defines \eqref{pbomega}. It is easy to check that, for two
factorizations of $f$, the corresponding isomorphisms are compatible with the
identifications \eqref{defepsi2}. This provides the isomorphism \eqref{pbomega} in the
general case.

When the assumptions of \ref{ci} (ii) are satisfied, the isomorphism \eqref{bcomega}
follows from \eqref{pbomega} and \eqref{basechange}.
\end{proof}

\subsection{}\label{chi}

Let
\eqn{\xymatrix{Y' \ar[r]^v\ar[d]_{f'} & Y\ar[d]^f\\
X' \ar[r]^u & X}}
be a cartesian square, and assume that:

\alphab $f$ and $u$ are complete intersection morphisms of relative dimensions $m$ and
$n$;

\alphab $X'$ and $Y$ are Tor-independent over $X$. 

Then Lemma \ref{basechangeomega} provides canonical isomorphisms
\eqn{v^*(\omega_{Y/X}) \riso \omega_{Y'/X'}, \quad\quad 
f'{}^*(\omega_{X'/X}) \riso \omega_{Y'/Y}.}
One defines the canonical isomorphism
\eq{defchi}{ \chi_{f,u} : \omega_{Y'/Y} \otimes_{\sO_{Y'}} v^*(\omega_{Y/X}) \riso 
\omega_{Y'/X'} \otimes_{\sO_{Y'}} f'{}^*(\omega_{X'/X}) }
as being the product by $(-1)^{mn}$ of the composite 
\eqn{ \omega_{Y'/Y} \otimes_{\sO_{Y'}} v^*(\omega_{Y/X}) \riso 
f'{}^*(\omega_{X'/X}) \otimes_{\sO_{Y'}} \omega_{Y'/X'} \riso 
\omega_{Y'/X'} \otimes_{\sO_{Y'}} f'{}^*(\omega_{X'/X}), }
where the first isomorphism is the product of the previous base change isomorphisms, 
and the second one is the usual commutativity isomorphism of the tensor product (see 
\cite[Appendix, (a)]{De83} and \cite[p.\ 215-216]{Co00}). 

The following relations follow easily from the local description of the isomorphisms 
$\zeta'_{f,g}$ given in \cite[p.~30, (a) - (d)]{Co00}:

\romain In the above cartesian square, assume that each of the three morphisms $u$, $f$
and $u \circ f' = f \circ v$ is either a smooth morphism or a regular immersion. Then
the following isomorphisms $\omega_{Y'/X} \riso \omega_{Y'/X'} \otimes_{\sO_{Y'}}
f'{}^*(\omega_{X'/X})$ are equal:
\eq{chizetasq}{ \zeta'_{f',u} = \chi_{f,u} \circ \zeta'_{v,f}. }

\romain Let
\eqn{\xymatrix@M=6pt{   
& Y' \ar@{^{(}->}[r]^v \ar[d]_{f'} & Y \ar[d]^f\\
X'' \ar@{^{(}->}[r]^i \ar@{^{(}->}[ur]^j & X' \ar@{^{(}->}[r]^u & X
}}
be a commutative diagram in which the square is cartesian, $f$ is smooth, $i$ and $u$ 
are regular immersions. Then the following isomorphisms 
\eqn{ \omega_{X''/X} \riso \omega_{X''/Y'} \otimes_{\sO_{X''}} j^*(\omega_{Y'/X'}) 
\otimes_{\sO_{X''}} i^*(\omega_{X'/X}) }
are equal:
\eq{chizetares}{ (\zeta'_{j,f'}\otimes\Id) \circ \zeta'_{i,u} = 
(\Id\otimes j^*(\chi_{f,u})) \circ (\zeta'_{j,v}\otimes\Id) \circ \zeta'_{vj,f}.}

\romain Let
\eqn{\xymatrix@M=6pt{
Y'' \ar[r]^{v'} \ar[d]_{f''} & Y' \ar[r]^v \ar[d]_{f'} & Y \ar[d]^f \\
X'' \ar[r]^{u'} & X' \ar[r]^u & X
}}
be a commutative diagram in which both squares are cartesian, each of the morphisms
$f$, $u$, $u'$ and $u \circ u'$ is either a smooth morphism or a regular immersion,
$X'$ and $Y$ are Tor-independent over $X$, and $X''$ and $Y'$ are Tor-independent over
$X'$ (so that $X''$ and $Y$ are Tor-independent over $X$, and all immersions are
regular). Then the following isomorphisms 
\eqn{ \omega_{Y''/Y}\otimes_{\sO_{Y''}}(vv')^*(\omega_{Y/X}) \riso 
\omega_{Y''/X''} \otimes_{\sO_{Y''}} 
f''{}^*(\omega_{X''/X'} \otimes_{\sO_{X''}} u'{}^*(\omega_{X'/X})) }
are equal:
\eq{chizetacomp}{ (\Id\otimes f''{}^*(\zeta'_{u',u})) \circ \chi_{f,uu'} =
(\chi_{f',u'}\otimes\Id) \circ (\Id\otimes v'{}^*(\chi_{f,u})) \circ
(\zeta'_{v',v}\otimes\Id). }
\medskip

We will also need to extend the isomorphism $\chi_{f,u}$ to the derived category. We 
define
\eq{defchi2}{ \chi_{f,u} : \omega_{Y'/Y}[n] \otimesL_{\sO_{Y'}} \LL 
v^*(\omega_{Y/X}[m]) \riso \omega_{Y'/X'}[m] \otimesL_{\sO_{Y'}} \LL 
f'{}^*(\omega_{X'/X}[n]) }
by \eqref{defchi} in degree $-(m+n)$, without any further sign modification. Because 
of the sign convention in the commutativity isomorphism for the derived tensor product 
\cite[p.~11]{Co00}, $\chi_{f,u}$ can also be described as the composite 
\begin{eqnarray*}
\omega_{Y'/Y}[n] \otimesL_{\sO_{Y'}} \LL v^*(\omega_{Y/X}[m]) & \riso &
\LL f'{}^*(\omega_{X'/X}[n]) \otimesL_{\sO_{Y'}} \omega_{Y'/X'}[m] \\
& \riso & \omega_{Y'/X'}[m] \otimesL_{\sO_{Y'}} \LL f'{}^*(\omega_{X'/X}[n]), 
\end{eqnarray*}
where the first isomorphism is the tensor product of the base change isomorphisms, and 
the second one is the commutativity isomorphism. With this definition, the previous 
relations \eqref{chizetasq} to \eqref{chizetacomp} remain valid in $\Db(\sO_{Y'})$.

\subsection{}\label{transomega}
We now consider the composition of two complete intersection morphisms $f : Y \to X$, 
$g : Z \to Y$, and define a canonical isomorphism 
\eq{transomega2}{\zeta'_{g,f} : \omega_{Z/X} \riso \omega_{Z/Y} \otimes_{\sO_Z} 
g^*(\omega_{Y/X})}
extending the isomorphism \eqref{transomega1}.

Assume first that there exists a factorization $f = \pi \circ i$, where $\pi : P \to 
X$ is a smooth morphism, and a factorization $ i \circ g = \pi'' \circ j$, where 
$\pi'' : P'' \to P$ is a smooth morphism (such factorizations always exist when $X$, 
$Y$ and $Z$ are affine). Let $\pi' : P' \to Y$ be the pull-back of $\pi''$,  so that 
we get a commutative diagram
\eq{transdiag}{\xymatrix{
& & P''  \ar@/^2.1pc/[ddrr]^{\psi} \ar[dr]^{\pi''} & & \\ 
& P' \ar@{^{(}->}[ur]^{i''} \ar[dr]^{\pi'} & \Box & P \ar[dr]^{\pi} & \\
Z \ar@{^{(}->}@< 2.5pt>@/^2.1pc/[uurr]!<0pt,-2pt>^j \ar@{^{(}->}[ur]^{i'} \ar[rr]_g & &
Y \ar@{^{(}->}[ur]^{i} \ar[rr]_f & & X
}}
where the middle square is cartesian. Using \eqref{transomega1} for $(j,\psi)$, and the
isomorphisms $\zeta'_{i',i''} \otimes j^*(\zeta'_{\pi'',\pi})$, we obtain an isomorphism
\begin{eqnarray*}
\omega_{Z/X}\ \cong\ \omega_{Z/P''} \otimes j^*(\omega_{P''/X}) & \riso & 
(\omega_{Z/P'} \otimes i'{}^*(\omega_{P'/P''})) \otimes j^*(\omega_{P''/P} \otimes 
\pi''{}^*(\omega_{P/X})) \\
& \riso & \omega_{Z/P'} \otimes i'{}^*(\omega_{P'/P''}\otimes i''{}^*(\omega_{P''/P})) 
\otimes g^*i^*(\omega_{P/X}).
\end{eqnarray*}
Using the isomorphism 
\eqn{ \chi_{\pi'',i} : \omega_{P'/P''}\otimes i''{}^*(\omega_{P''/P}) \riso 
\omega_{P'/Y} \otimes \pi'{}^*(\omega_{Y/P}) }
defined in \ref{chi}, and $(\zeta'_{i',\pi'} \otimes g^*(\zeta'_{i,\pi}))^{-1}$, we
then obtain the composed isomorphism
\begin{eqnarray*}
\omega_{Z/X} & \riso & \omega_{Z/P'} \otimes i'{}^*(\omega_{P'/Y} \otimes 
\pi'{}^*(\omega_{Y/P})) \otimes g^*i^*(\omega_{P/X}) \\
& \riso & (\omega_{Z/P'} \otimes i'{}^*(\omega_{P'/Y})) \otimes 
g{}^*(\omega_{Y/P} \otimes i^*(\omega_{P/X})) \\
& \riso & \omega_{Z/Y} \otimes g^*(\omega_{Y/X}),
\end{eqnarray*}
which defines \eqref{transomega2}. 

To prove that this isomorphism is well defined, and to glue the local constructions to
obtain a global one when a diagram \eqref{transdiag} does not exist globally, we must
check that it does not depend on the chosen factorizations. If we have two diagrams
\eqref{transdiag}, with factorizations $f = \pi_k \circ i_k$, $i_k \circ g = \pi''_k
\circ j_k$, for $k = 1, 2$, we can embed $Y$ diagonally in $P_1 \times_X P_2$, and $Z$
in $P''_1 \times_X P''_2$. This allows to reduce the verification to the case where
there exists a smooth $X$-morphism $u : P_2 \to P_1$ such that $u \circ i_2 = i_1$, and
a smooth morphism $u'' : P''_2 \to P''_1$ such that $\pi''_1 \circ u'' = u \circ
\pi''_2$, and $j_1 = u'' \circ j_2$. Morever, the same argument shows that we may
assume that the morphism $P''_2 \to P''_1 \times_{P_1} P_2$ is smooth. The verification
can then be reduced to the following two cases:

\romain The morphism $P''_2 \to P''_1 \times_{P_1} P_2$ is an isomorphism;

\romain The morphism $P_2 \to P_1$ is an isomorphism.

In each of these cases, the equality of the two definitions of \eqref{transomega2}
breaks down to a succession of elementary commutative diagrams involving isomorphisms
of the form $\zeta'_{f,g}$ and $\chi_{f,u}$. We omit details here, and only point out
that, in addition to \cite[(2.2.4)]{Co00}, the first case uses relation
\eqref{chizetasq}, and the second one uses relation \eqref{chizetares}. In particular,
the sign convention introduced in the definition of $\chi_{f,u}$ in \ref{chi} is
necessary for this independence result.
\medskip

If $m$ and $m'$ are the virtual relative dimensions of $f$ and $g$, we define as in
\ref{Defomega} the derived category variant of \eqref{transomega2} as being the morphism
\eq{transomega3}{ \zeta'_{g,f} : \omega_{Z/X}[m+m'] \riso \omega_{Z/Y}[m']
\otimesl_{\sO_Z} \LL f^*(\omega_{Y/X}[m]) }
defined by applying \eqref{transomega2} to the underlying modules (sitting in degree 
$-m-m'$), without any sign modification. 
\medskip

With the definition of $\zeta'_{g,f}$ provided by \eqref{transomega2}
(resp.~\eqref{transomega3}), we now extend to complete intersection morphisms Conrad's
transitivity relation \cite[(2.2.4)]{Co00}.

\begin{prop}\label{Transzetap}
Let 
\eqn{T \xra{\ h\ } Z \xra{\ g\ } Y \xra{\ f\ } X }
be three complete intersection morphisms. Then
\eq{transzetap}{ (\Id \otimes h^*(\zeta'_{g,f})) \circ \zeta'_{h,fg} = (\zeta'_{h,g}
\otimes \Id) \circ \zeta'_{gh,f}. }
\end{prop}

\begin{proof} As the verification is local on $T$, we may assume that there exists a 
commutative diagram
\eqn{\xymatrix@M=4pt{
& & & R'' \ar[rd]^{\psi''} & & & \\ 
& & Q'' \ar@{^{(}->}[ur]^{k''} \ar[rd]^{\varphi''} & & Q' \ar[rd]^{\varphi'} & & \\
& P'' \ar@{^{(}->}[ur]^{j''} \ar[rd]^{\pi''} & & 
P' \ar@{^{(}->}[ur]^{j'} \ar[rd]^{\pi'} & & P \ar[rd]^{\pi} & \\
T \ar@{^{(}->}[ur]^{i''} \ar[rr]_{h} & & Z \ar@{^{(}->}[ur]^{i'} \ar[rr]_{g} & & 
Y \ar@{^{(}->}[ur]^{i} \ar[rr]_{f} & & X
}}
in which the three squares are cartesian, the morphisms $\pi$, $\varphi'$, $\psi''$ 
are smooth, and the immersions $i$, $i'$, $i''$ are regular. Using 
\cite[(2.2.4)]{Co00} and the relation \eqref{chizetacomp}, the proof of 
\eqref{transzetap} again breaks into a succession of elementary commutative diagrams, 
which we do not detail here.
\end{proof}

\subsection{}\label{defdelta}

We now assume that $f : Y \to X$ is a complete intersection morphism of (virtual)
relative dimension $0$, and, under this hypothesis, we define a section 
$\delta_f \in \Gamma(Y, \omega_{Y/X})$, which we call the \textit{canonical section}. 

We first assume that there is a factorization $f = \pi \circ i$ such that $\pi : P
\to X$ is a smooth morphism of relative dimension $n$, and $i : Y \inj P$ is a regular
closed immersion, necessarily of codimension $n$ since $f$ has relative dimension $0$. 
Let $\sI \subset \sO_P$ be the ideal defining $i$. The
canonical derivation $d : \sO_P \to \Omega^1_{P/X}$ induces an $\sO_Y$-linear
homomorphism $\bd : \sI/\sI^2 \to i^*\Omega^1_{P/X}$. Taking its $n$-th exterior
power, we obtain a linear homomorphism
\eq{wedged}{\wedge^n \bd : \wedge^n(\sI/\sI^2) \lra i^* \Omega^n_{P/X}.}
Through the canonical isomorphisms 
\begin{eqnarray*}
\sHom_{\sO_Y}(\wedge^n(\sI/\sI^2), i^*\Omega^n_{P/X}) & \cong & 
(\wedge^n(\sI/\sI^2))^\vee\otimes_{\sO_Y}i^*\Omega^n_{P/X} \\
& \cong & \wedge^n((\sI/\sI^2)^\vee) \otimes_{\sO_Y}i^*\Omega^n_{P/X} \\
& = & \omega_{Y/X},
\end{eqnarray*}
it can be seen as a section of $\omega_{Y/X}$, which is the section $\delta_f$. If
$(t_1,\ldots,t_n)$ is a regular sequence of generators of $\sI$ on a neighbourhood $U$
of some point $y \in Y$, then
\eq{locdeltaf}{ \delta_f = (\bt_1^{\,\vee}\wedge\ldots\wedge\bt_n^{\,\vee}) \otimes 
i^*(dt_n\wedge\ldots\wedge dt_1) \ \in\Gamma(U, \omega_{Y/X}), }
since the canonical isomorphism $(\wedge^n(\sI/\sI^2))^\vee \cong 
\wedge^n((\sI/\sI^2)^\vee)$ maps $(\bt_n\wedge\ldots\wedge\bt_1)^\vee$ to 
$\bt_1^{\,\vee}\wedge\ldots\wedge\bt_n^{\,\vee}$. 

To end the construction of $\delta_f$, it suffices to check that the section obtained
in this way does not depend on the chosen factorization. Using the diagonal embedding,
it suffices as usual to compare the sections $\delta_f$ and $\delta'_f$ defined by two
factorizations $f = \pi \circ i = \pi' \circ i'$ when there exists a smooth
$X$-morphism $u : P' \to P$ such that $u \circ i' = i$. Let $\sI'$ be the ideal of $Y$
in $P'$, and
\eqn{\omega'_{Y/X} = \wedge^{n'}((\sI'/\sI'{}^2)^\vee) \otimes_{\sO_Y} 
i'{}^*\Omega^{n'}_{P'/X},} 
where $n'$ is the codimension of $Y$ in $P'$. Then the canonical identification
$\omega_{Y/X} \cong \omega'_{Y/X}$ is given by \eqref{defepsi}, case a), and, thanks to
\eqref{locdeltaf}, the equality $\delta_f = \delta'_f$ follows from \cite[p.\ 30, (a)
and (d)]{Co00}.

\begin{prop}\label{propdelta}
Let $f : Y \to X$ be a complete intersection morphism of virtual relative dimension $0$.

\romain Let $g : Z \to Y$ be a second complete intersection morphism of virtual
relative dimension $0$. The image of $\delta_{fg}$ under the isomorphism $\zeta'_{g,f}$
defined in \eqref{transomega2} is given by 
\eq{transdelta}{ \zeta'_{g,f}(\delta_{fg}) = \delta_{g} \otimes g^*(\delta_f). }

\romain For any cartesian square \eqref{cart}, the isomorphism \eqref{pbomega}
\eqn{v^*(\omega_{Y/X}) \riso \omega_{Y'/X'}}
maps $v^*(\delta_f)$ to $\delta_{f'}$. 
\end{prop}

\begin{proof}
As the first claim is local on $Z$, we may assume that there exists a diagram
\eqref{transdiag} in which the immersion $i$ is defined by a regular sequence
$(t_1,\ldots,t_n)$, and the immersion $j=i''\circ i'$ by a regular sequence
$(t'_1,\ldots,t'_{n'},t''_1,\ldots,t''_n)$, with $t''_i = \pi''{}^*(t_i)$. If we set
$\bt'_i = i''{}^*(t'_i)$, then $i'$ is defined by the regular sequence
$(\bt'_1,\ldots,\bt'_{n'})$. By construction, $\delta_{fg}$ corresponds by
$\zeta'_{j,\psi}$ to the section
\eqn{(t''_n{}^\vee \wedge \ldots \wedge t''_1{}^\vee \wedge
t'_{n'}{}^\vee \wedge \ldots \wedge t'_1{}^\vee) \otimes j^*(dt'_1 \wedge \ldots \wedge
dt'_{n'} \wedge dt''_1 \wedge \ldots \wedge dt''_n)}
of $\omega_{Z/P''}\otimes j^*(\omega_{P''/X})$, which is mapped by 
$\zeta'_{i',i''}\otimes j^*(\zeta'_{\pi'',\pi})$ to the section 
\eqn{((-1)^{nn'}(\bt'_{n'}{}^\vee \wedge \ldots \wedge \bt'_1{}^\vee) \otimes
i'{}^*(t''_n{}^\vee \wedge \ldots \wedge t''_1{}^\vee)) \otimes 
j^*((dt'_1 \wedge \ldots \wedge dt'_{n'}) \otimes 
\pi''{}^*(dt_1 \wedge \ldots \wedge dt_n)) }
of $(\omega_{Z/P'}\otimes i'{}^*(\omega_{P'/P''})) \otimes j^*(\omega_{P''/P} \otimes
\pi''{}^*(\omega_{P/X}))$. We then get via
$\chi_{\pi'',i}$ the section 
\eqn{ (\bt'_{n'}{}^\vee \wedge \ldots \wedge \bt'_1{}^\vee) \otimes 
i'{}^*(d\bt'_1 \wedge \ldots \wedge d\bt'_{n'}) \otimes 
i'{}^*\pi'{}^*(t_n{}^\vee \wedge \ldots \wedge t_1{}^\vee) \otimes 
j^*\pi''{}^*(dt_1 \wedge \ldots \wedge dt_n)), }
of $\omega_{Z/P'} \otimes i'{}^*(\omega_{P'/Y}) \otimes i'{}^*\pi'{}^*(\omega_{Y/P})
\otimes j^*\pi''{}^*(\omega_{P/X})$, which, by construction, corresponds by
$(\zeta'_{i',\pi'}\otimes g^*(\zeta'_{i,\pi}))^{-1}$ to the section $\delta_{g} 
\otimes g^*(\delta_{f})$ of $\omega_{Z/Y} \otimes g^*(\omega_{Y/X})$. 

The second claim follows from \eqref{locdeltaf}. 
\end{proof}

\begin{center}\textsc{B. The trace morphism $\tau_f$ on $\RR f_*(\sO_Y)$}\end{center}
\addcontentsline{toc}{section}{B.\hspace{3mm}The trace morphism $\tau_f$ on $\RR
f_*(\sO_Y)$}\setcounter{section}{2}\setcounter{subsection}{0}

Let $f : Y \to X$ be a proper complete intersection morphism of virtual relative
dimension $0$. This section is devoted to the construction of the ``trace morphism''
$\tau_f : \RR f_*\sO_Y \to \sO_X$, derived from the canonical section of $\omega_{Y/X}$
defined in \ref{defdelta}. The key step is to define an identification $\lambda_f$
between $\omega_{Y/X}$ as defined in \ref{Defomega}, and $f^!\sO_X$. The construction
is then a straightforward application of the relative duality theorem, and the
properties of $\tau_f$ listed in Theorem \ref{Thtau} follow from corresponding
properties of $\delta_f$ and $\lambda_f$.
\medskip

\subsection{}\label{Deflambda}
For the whole section, we assume that $X$ is a noetherian scheme with a dualizing
complex. If $f : Y \to X$ is a morphism of finite type, and $K\hbul$ a residual complex
on $X$, let $f^\Delta K\hbul$ be its inverse image on $Y$ in the sense of residual
complexes, which is a residual complex on $Y$. Then $K\hbul$ and $f^\Delta K\hbul$
define respectively duality $\delta$-functors $D_X$ on $\Dcoh(\sO_X)$ and $D_Y$ on
$\Dcoh(\sO_Y)$. We recall that, following \cite[3.3]{Co00}, the functor $f^! : 
\Dpcoh(\sO_X) \to \Dpcoh(\sO_Y)$ is defined by $f^! = D_Y \circ \LL f^* \circ D_X$. We 
also recall that, when $f$ is smooth of relative dimension $m$, $f^\sharp :
\Dpcoh(\sO_X) \to \Dpcoh(\sO_Y)$ denotes the functor defined by
\eq{deffsharp}{ f^\sharp(\sE\hbul) := \omega_{Y/X}[m] \otimesL_{\sO_Y} \LL
f^*(\sE\hbul),}
while, when $f$ is finite, $f^\flat : \Dpcoh(\sO_X) \to \Dpcoh(\sO_Y)$ denotes the 
functor defined by
\eq{deffflat}{f^\flat(\sE\hbul) := \overline{f}^*\RR\sHom_{\sO_X}(f_*\sO_Y, \sE\hbul),}
where $\overline{f}$ is the (flat) morphism of ringed spaces $(Y,\sO_Y) \to
(X,f_*\sO_Y)$.

Assume now that $f : Y \to X$ is a complete intersection morphism of virtual relative
dimension $m$. We first explain the relation between the relative dualizing module
$\omega_{Y/X}$ defined in the previous section, and the extraordinary inverse image
functor $f^!$. We will consider complexes of the form $\sE\hbul = \sL[r] \in
\Dbcoh(\sO_X)$, where $r \in \Z $ is some integer, and $\sL$ is an invertible $\sO_X$-module. 
For such a complex, we generalize the above notation $f^\sharp$, and define again 
\eq{deffsharp2}{ f^\sharp(\sE\hbul) := \omega_{Y/X}[m] \otimesL_{\sO_Y} \LL f^*(\sE\hbul).}
We observe that $f^\sharp(\sE\hbul)$ is another complex concentrated in a single degree,
with an invertible cohomology sheaf. We can then construct a canonical isomorphism
\eq{deflambda}{ \lambda_{f,\sE\hbul} : f^\sharp(\sE\hbul) \riso f^!(\sE\hbul) }
as follows.

\alphab If $f$ is smooth, then definitions \eqref{deffsharp} and \eqref{deffsharp2} 
coincide, and we set
\eq{deflambdas}{ \lambda_{f,\sE\hbul} = e_f : f^\sharp(\sE\hbul) \riso f^!(\sE\hbul), }
where $e_f$ is the isomorphism defined by \cite[(3.3.21)]{Co00}. 

\alphab If $f$ is a regular immersion, then we define $\lambda_{f,\sE\hbul}$ to be the 
composition
\eq{deflambdai}{ \lambda_{f,\sE\hbul} : f^\sharp(\sE\hbul) \xrightarrow[\sim]{\eta_f^{-1}} 
f^\flat(\sE\hbul) \xrightarrow[\sim]{d_f} f^!(\sE\hbul), }
where $\eta_f$ is defined by \cite[(2.5.3)]{Co00} and $d_f$ by \cite[(3.3.19)]{Co00}. 

\alphab In the general case, let us assume first that there exists a factorization $f =
\pi\circ i$, where $\pi : P \to X$ is a smooth morphism of relative dimension $n$, and
$i$ is a regular immersion of codimension $d = n-m$. Then we define $\lambda_{f,\sE\hbul}$ 
by the commutative diagram
\eq{deflambdag}{ \xymatrix@C=40pt{\entry{\omega_{Y/X}[m] \otimesL_{\sO_Y} \LL f^*\sE\hbul} 
\ar[d]!<0ex,3ex>^-\sim_-{\zeta'_{i,\pi}\otimes\Id} 
\ar[rr]^-{\lambda_{f,\sE\raisebox{.3mm}{$\cdot$}}}_-{\sim} 
& & \entry{f^!\sE\hbul} \ar[d]!<0ex,2ex>^{c_{i,\pi}}_\sim \\
\entry{\omega_{Y/P}[-d] \otimesL_{\sO_Y} \LL i^*\pi^\sharp\sE\hbul} 
\ar[r]^-{\lambda_{i,\pi^\sharp\sE\raisebox{.3mm}{$\cdot$}}}_-\sim & 
\entry{i^!\pi^\sharp\sE\hbul} 
\ar[r]^-{i^!(\lambda_{\pi,\sE\raisebox{.3mm}{$\cdot$}})}_-\sim & \entry{i^!\pi^!\sE\hbul} 
}}
where $c_{i,\pi}$ is the transitivity isomorphism \cite[(3.3.14)]{Co00}.

This isomorphism is actually independent of the chosen factorization. To check it, one
can argue as in \ref{Defomega} to reduce the comparison between the isomorphisms
\eqref{deflambda} defined by two factorizations $f = \pi \circ i = \pi' \circ i'$ to
the case where there is a smooth $X$-morphism $u : P' \to P$ such that $u \circ i' =
i$. It is then a long but straightforward verification, using various functorialities,
the compatibility between $\zeta'_{i',u}$ and the isomorphism $i^{\flat} \simeq
i'{}^{\flat}u^{\sharp}$ \cite[(2.7.4)]{Co00}, the compatibility between
$\zeta'_{u,\pi}$ and the isomorphism $\pi'{}^{\sharp} \simeq u^{\sharp}\pi^{\sharp}$
\cite[(2.2.7)]{Co00}, and the properties (VAR1), (VAR3) and (VAR5) of the functor $f^!$
(see \cite[III, Th.~8.7]{Ha66} and \cite[p.~139]{Co00}).

Since $f^!\sE\hbul$ is acyclic outside degree $-m-r$, a morphism $\omega_{Y/X}[m] 
\otimesl \LL f^*\sE\hbul \to
f^!\sE\hbul$ in $D(\sO_Y)$ is simply a module homomorphism $\omega_{Y/X}\otimes f^*\sL \to
\sH^{-m-r}(f^!\sE\hbul)$. Therefore, the previous construction provides in the general
case local isomorphisms which can be glued to define a global isomorphism even if there
does not exist a global factorization $f = \pi \circ i$ as above.

When $\sE\hbul = \sO_X[0]$, the isomorphism \eqref{deflambda} will simply be denoted 
\eq{deflambda2}{ \lambda_f : \omega_{Y/X}[m] \riso f^!\sO_X.}
If $f$ is flat, hence is a CM map, it provides the identification between the 
construction of $\omega_{Y/X}$ used in this article, and the construction of Conrad 
for CM maps \cite[3.5, p.~157]{Co00}.

We now give for the isomorphisms $\lambda_{f,\sE\hbul}$ a transitivity property which 
generalizes \eqref{deflambdag}.

\begin{prop}\label{Translambda}
Let $f : Y \to X$, $g : Z \to Y$ be two complete intersection morphisms, with virtual
relative dimensions $m$, $m'$, and let $\sE\hbul = \sL[r]$, for some invertible
$\sO_X$-module $\sL$ and some integer $r$. Then the diagram
\eq{translambda}{ \xymatrix@C=35pt{
\entry{\omega_{Z/X}[m+m'] \otimesL_{\sO_{Z}} \LL (fg)^*\sE\hbul} 
\ar[d]!<0ex,3ex>^-\sim_-{\zeta'_{g,f}\otimes\Id} 
\ar[rr]^-{\lambda_{fg,\sE\raisebox{.3mm}{$\cdot$}}}_-{\sim} 
& & \entry{(fg)^!\sE\hbul} \ar[d]!<0ex,2ex>^(.55){c_{g,f}}_(.55)\sim \\
\entry{\omega_{Z/Y}[m'] \otimesL_{\sO_Z} \LL g^*f^\sharp\sE\hbul} 
\ar[r]^-{\lambda_{g,f^\sharp\sE\raisebox{.3mm}{$\cdot$}}}_-\sim \ar@{=}[d]!<0ex,3ex> & 
\entry{g^!f^\sharp\sE\hbul} 
\ar[r]^-{g^!(\lambda_{f,\sE\raisebox{.3mm}{$\cdot$}})}_-\sim & 
\entry{g^!f^!\sE\hbul} \ar@{=}[d]!<0ex,2ex> \\
\entry{\omega_{Z/Y}[m'] \otimesL_{\sO_Z} \LL g^*f^\sharp\sE\hbul} 
\ar[r]^-{g^\sharp(\lambda_{f,\sE\raisebox{.3mm}{$\cdot$}})}_-\sim & 
\entry{\omega_{Z/Y}[m'] \otimesL_{\sO_Z} \LL g^*f^!\sE\hbul} 
\ar[r]^-{\lambda_{g,f^!\sE\raisebox{.3mm}{$\cdot$}}}_-\sim & \entry{g^!f^!\sE\hbul} 
}}
commutes.
\end{prop}

\begin{proof}
The commutativity of the lower part of the diagram is due to the functoriality of the 
isomorphism $\lambda_{g}$ with respect to morphisms between two complexes 
concentrated in the same degree.

We first observe that the commutativity of \eqref{translambda} is clear in the 
following cases:

\alphab If $f$ is smooth and $g$ is a closed immersion, the diagram is
\eqref{deflambdag}, which commutes by construction.

\alphab If $f$ and $g$ are smooth, the isomorphism $(fg)^\sharp \cong g^\sharp
f^\sharp$ is defined by $\zeta'_{g,f}$, hence the commutativity of \eqref{translambda}
is the compatibility of the isomorphisms $e_f$ with composition, i.e., property (VAR3)
of the functor $f^!$ \cite[p.~139]{Co00}.

\alphab If $f$ and $g$ are regular immersions, then isomorphisms such as $\eta_{fg}$
commute with $\zeta'_{g,f}$ and $c_{g,f}$ \cite[Th.~2.5.1]{Co00}, and the commutativity
of \eqref{translambda} follows from the compatibility of the isomorphisms $d_f$ with
composition, i.e., property (VAR2) of the functor $f^!$ \cite[p.~139]{Co00}.

We will also use the following remark. Let $h : T \to Z$ be a third complete
intersection morphism, yielding the four couples of composable complete intersection
morphisms $(h,g)$, $(g,f)$, $(gh,f)$ and $(h,fg)$. Then, if the diagrams
\eqref{translambda} for the couples $(h,g)$ and $(g,f)$ are commutative, the
commutativity of \eqref{translambda} for $(gh,f)$ is equivalent to the commutativity of
\eqref{translambda} for $(h,fg)$: this is a consequence of \eqref{transzetap} and of
the compatibility of the isomorphisms $c_{g,f}$ with triple composites (i.e., property
(VAR1) of the functor $f^!$ \cite[p.~139]{Co00}).

In the general case, the complexes entering in \eqref{translambda} are concentrated in
the same degree, hence its commutativity can be checked locally. So we may assume that
there exists a diagram \eqref{transdiag}. Thanks to the three particular cases listed
above, one can then deduce the commutativity of \eqref{translambda} for $(f,g)$ from
the commutativity of \eqref{translambda} for $(\pi',i)$, by applying the previous
remark successively to the triples $(i',i'',\pi'')$, $(i''i',\pi'',\pi)$, $(i',\pi',i)$
and $(g,i,\pi)$.

To prove the commutativity of \eqref{translambda} for $(\pi',i)$, we use the
factorization $i\circ \pi' = \pi''\circ i''$ to define $\lambda_{i\pi',\sE}$. Let $d$
be the codimension of $Y$ in $P$, and $n'$ the relative dimension of $P''$ over $P$.
Then, if $\sE$ is a complex on $P$ as in \ref{Deflambda}, \eqref{translambda} for
$(\pi',i)$ is made of the exterior composites in the diagram
\eq{bigtransdiag}{ \hspace{-2mm}\xymatrix@C=-48pt{ 
& \omega_{P'/P}[n'-d] \otimesL \LL(i\pi')^*\sE\hbul 
\ar[ld]!<8ex,3.5ex>_-(.9){\zeta'_{\pi',i}\otimes\Id}^(.6){\sim} 
\ar[rd]!<-8ex,3.5ex>^(.9){\zeta'_{i'',\pi''}\otimes\Id}_(.6){\sim} & \\
\entry{\omega_{P'/Y}[n'] \otimesL \pi'{}^*(\omega_{Y/P}[-d] \otimesL \LL i^*\sE\hbul)}
\ar[rr]^-{\chi_{\pi'',i}\otimes\Id}_-{\sim} \ar[d]^{\pi'{}^\sharp(\eta_i^{-1})}_{\sim} 
\ar@/_2.5pc/[dd]_{\pi'{}^\sharp(\lambda_{i,\sE\raisebox{0.2mm}{$\cdot$}})}
& & \entry{\omega_{P'/P''}[-d] \otimesL \LL i''{}^*(\omega_{P''/P}[n'] \otimesL 
\pi''{}^*\sE\hbul)} \ar[d]_{\eta_{i''}^{-1}}^{\sim} 
\ar@/^2.5pc/[dd]^{\lambda_{i'',\pi''{}^\sharp\sE\raisebox{0.2mm}{$\cdot$}}} \\ 
\pi'{}^\sharp i^\flat\sE\hbul \ar[rr]^{\sim} \ar[d]^{\pi'{}^\sharp(d_i)}_{\sim} & & 
i''{}^\flat\pi''^\sharp\sE\hbul \ar[d]_{d_{i''}}^{\sim} \\
\pi'{}^\sharp i^!\sE\hbul \ar[d]^{e_{\pi'}=\lambda_{\pi',i^!\sE\raisebox{0.2mm}{$\cdot$}}}_{\sim} & & 
i''{}^!\pi''^\sharp\sE\hbul 
\ar[d]_{i''{}^!(\lambda_{\pi'',\sE\raisebox{0.2mm}{$\cdot$}})=i''{}^!(e_{\pi''})}^{\sim} \\
\pi'{}^!i^!\sE\hbul & (i\pi')^!\sE\hbul = (\pi''i'')^!\sE\hbul 
\ar[l]^-{c_{\pi',i}}_-{\sim} \ar[r]_-{c_{i'',\pi''}}^-{\sim} & 
i''{}^!\pi''^!\sE\hbul.
 } }
Here, the middle horizontal arrow is the standard isomorphism \cite[Lemma 2.7.3]{Co00},
and the lower rectangle commutes thanks to property (VAR4) of the functor $f^!$
\cite[Theorem 3.3.1]{Co00}. The upper triangle commutes thanks to \eqref{chizetasq}. To
check the comutativity of the middle rectangle, one observes on the one hand that
$\eta_i$ commutes with the flat base change $\pi''$ and that $\eta_{i''}$ commutes with
tensorisation by the invertible sheaf $\omega_{P''/P}$ (see \cite[last paragraph of
p.~54]{Co00}). On the other hand, $\eta_{i''}$ commutes also with the translation by
$n'$, provided that the convention \cite[(1.3.6)]{Co00} is used for the commutation of
the tensor product with translations applied to the second argument (see the discussion
on \cite[p.~53]{Co00}). This requires here multiplication by $(-1)^{dn'}$ on
$\omega_{P'/P''}\otimes i''{}^*\omega_{P''/P}$, since $\omega_{P'/P''}$ sits in
degree $d$. As this is the sign entering in the definition of $\chi_{\pi'',i}$, this
ends the proof.
\end{proof}

\subsection{}\label{Trsharp}
Assume now that $f$ is proper. As in \ref{Deflambda}, let $\sE\hbul = \sL[r] \in
\Dbcoh(\sO_X)$, $\sL$ being an invertible $\sO_X$-module and $r$ an integer. Using
\eqref{deflambda}, we can define the trace morphism $\Trs_{f,\sE\hbul}$ on $\RR
f_*f^\sharp\sE\hbul = \RR f_*(\omega_{Y/X}[m]\otimesl_{\sO_Y} \LL f^*\sE\hbul)$ as the
composite
\eq{trsharp}{ \Trs_{f,\sE\hbul} : \RR f_*f^\sharp\sE\hbul
\xrightarrow[\sim]{\RR f_*(\lambda_{f,\sE\raisebox{0.2mm}{$\cdot$}})} 
\RR f_*f^!\sE\hbul \xra{\Tr_f} \sE\hbul, }
where $\Tr_f$ denotes the classical trace morphism defined in \cite[VII, Cor.\
3.4]{Ha66} and \cite[3.4]{Co00}. When $\sE\hbul = \sO_X[0]$, we will use the shorter notation
\eq{trsharp2}{ \Trs_f : \RR f_*(\omega_{Y/X}[m]) \to \sO_X.}

We first give some basic properties of the morphism $\Trs_f$.

\begin{lem}\label{Projform}
With the previous hypotheses, let 
\eq{defmu}{ \mu_{f,\sE\hbul} : \RR f_*(\omega_{Y/X}[m]) \otimesL_{\sO_X} \sE\hbul \riso
\RR f_*(\omega_{Y/X}[m] \otimesL_{\sO_Y} \LL f^*\sE\hbul) = \RR f_*f^\sharp\sE\hbul }
be the isomorphism given by the projection formula \cite[II, Prop.~5.6]{Ha66}. Then 
the diagram
\eq{projform}{ \xymatrix{
\entry{\RR f_*(\omega_{Y/X}[m]) \otimesL_{\sO_X} \sE\hbul}
\ar[rr]^-{\mu_{f,\sE\raisebox{0.2mm}{$\cdot$}}}_-{\sim}
\ar[rd]^-{\sim}_{\Trs_f\otimes\Id} 
& & \entry{\RR f_*f^\sharp\sE\hbul}
\ar[dl]_(.6){\sim}^{\Trs_{f,\sE\raisebox{0.2mm}{$\cdot$}}} \\
& \sE\hbul &
} }
commutes. 
\end{lem}

\begin{proof}
When $f$ is flat, it suffices to invoke \cite[Th.~4.4.1]{Co00}. Since we make no such 
assumption on $f$, we give a direct argument, which is made a lot simpler by the very 
special nature of the complex $\sE\hbul$. 

Using the fact that $\sE\hbul = \sL[r]$, with $\sL$ invertible, one easily sees that 
there is a canonical isomorphism which commutes with translations acting on $\sE\hbul$
\eq{shriekstar}{ f^!\sO_X \otimesL_{\sO_Y} \LL f^*\sE\hbul \riso f^!\sE\hbul. } 
On the other hand, we have by definition a canonical isomorphism
\eq{sharpstar}{ f^\sharp\sO_X \otimesL_{\sO_Y} \LL f^*\sE\hbul \riso f^\sharp\sE\hbul, }
which also commutes with translations. A first observation is that the diagram 
\eq{takeout1}{ \xymatrix@C=50pt{
\entry{f^\sharp\sO_X \otimesL_{\sO_Y} \LL f^*\sE\hbul} 
\ar[d]!<0ex,3ex>_-{\lambda_f \otimes \Id}^-{\sim}
\ar[r]^-{\eqref{sharpstar}}_-{\sim}
& \entry{f^\sharp\sE\hbul} 
\ar[d]!<0ex,2ex>^(.52){\lambda_{f,\sE\raisebox{0.2mm}{$\cdot$}}}_(.52){\sim} \\
\entry{f^!\sO_X \otimesL_{\sO_Y} \LL f^*\sE\hbul} 
\ar[r]^-{\eqref{shriekstar}}_-{\sim} 
& \entry{f^!\sE\hbul} 
} }
commutes. Indeed, all complexes are concentrated in the same degree $m-r$, hence the 
verification can be done locally. This allows to assume that $\sL = \sO_X$, which 
reduces the verification to the commutation of the vertical arrows with translations 
acting on $\sE\hbul$. This now follows from the fact that the isomorphisms $e_f$, $\eta_f$ 
and $d_f$ used in the construction of $\lambda_f$ commute with translations.

Applying $\RR f_*$ to this diagram, and using the functoriality of the projection 
formula isomorphism, the proof is reduced to proving the commutativity of the diagram 
\eq{projform2}{ \xymatrix{ 
\entry{\RR f_*f^!\sO_X \otimesL_{\sO_X} \sE\hbul} \ar[rd]_(.54){\Tr_f\otimes\Id\ \ }^-{\sim}
\ar[r]^-{\nu_f}_-{\sim} &
\entry{\RR f_*(f^!\sO_X \otimesL_{\sO_Y} \LL f^*\sE\hbul)} 
\ar[r]^-{\eqref{shriekstar}}_-{\sim} & 
\entry{\RR f_*f^!\sE\hbul} \ar[ld]^{\Tr_f}_(.6){\sim} \\
& \sE\hbul & \hspace{1.5cm},
 } }
where $\nu_f$ is the projection formula isomorphism. As all morphisms of the diagram
commute with translations, we may assume that $r = 0$. We recall that $\Tr_f$ is
defined as the morphism of functors defined by the composite
\eqn{ \RR f_*f^!(\cdot) \riso \RR\sHom_{\sO_X}(D_X(\cdot),f_*f^\Delta K) 
\xra{\Tr_{f,K}} \RR\sHom_{\sO_X}(D_X(\cdot), K) \liso \Id,}
where the first isomorphism follows from the definition of $f^!$ and the adjunction
between $\LL f^*$ and $\RR f_*$, the second morphism is defined by the trace morphism
for residual complexes $\Tr_{f,K}$ and the last isomorphism is the local biduality
isomorphism (see \cite[p.~146]{Co00}). Each of these morphisms has a natural
compatibility with respect to the tensor product of the argument by an invertible
sheaf. Putting together these compatibilities yields the commutativity of
\eqref{projform2}.
\end{proof}

\begin{prop}\label{TransTrs}
Let $g : Z \to Y$ be a second proper complete intersection morphism, with virtual
relative dimension $m'$. Then the diagram
\eq{transTrs}{\xymatrix@C=70pt{
\entry{\RR f_*\RR g_*(\omega_{Z/X}[m'+m])} \ar[dd]_{\Trs_{fg}}
\ar[r]_-{\sim}^-{\RR f_* \RR g_*(\zeta'_{g,f})} & 
\entry{\RR f_*\RR g_*(\omega_{Z/Y}[m']\otimesL_{\sO_Z} \LL g^*(\omega_{Y/X}[m]))}
\ar[d]_{\sim} \\
 & \RR f_*(\RR g_*(\omega_{Z/Y}[m'])\otimesL_{\sO_Y} \omega_{Y/X}[m]) 
\ar[d]^{\RR f_*(\Trs_g\otimes\Id)} \\
 \sO_X  & 
\RR f_*(\omega_{Y/X}[m]) \ar[l]_{\Trs_f}
}}
\lp where the second isomorphism is given by the projection formula\hspace{1pt}\rp\ is
commutative.
\end{prop}

\begin{proof}
It follows from Lemma \ref{Projform} that the right vertical arrow is equal to the 
morphism 
\eqn{ \RR f_*\RR g_*(\omega_{Z/Y}[m']\otimes g^*(\omega_{Y/X}[m])) 
\xra{\Trs_{g,\omega_{Y/X}[m]}} \RR f_*(\omega_{Y/X}[m]). }
Then, using adjunction between $\RR f_*$ and $f^!$, and adjunction between $\RR g_*$ and 
$g^!$, one sees that the commutativity of \eqref{transTrs} is equivalent to the 
commutativity of \eqref{translambda}. 
\end{proof}

\begin{prop}\label{basechangeTrs}
With the hypotheses of Proposition \ref{ci}, assume in addition that $X$ and $X'$ are
noetherian schemes with dualizing complexes, and that one of the following conditions 
is satisfied:

\alphab $f$ is projective;

\alphab $f$ is proper and $u$ is residually stable \cite[p.~132]{Co00};

\alphab $f$ is proper and flat. 

Then the triangle
\eq{bctrace}{\xymatrix{
\LL u^*\RR f_*(\omega_{Y/X}[m]) \ar[dd]_{\eqref{bcomega}}^-{\sim} \ar[dr]^-{\ \LL u^*(\Trs_f)} &
\\
 & \sO_{X'} \\
\RR f'_*(\omega_{Y'/X'}[m]) \ar[ur]_-{\Trs_{f'}} 
}}
is commutative.
\end{prop}

\begin{proof}[Proof of Case a)]
We can choose a factorization $f = \pi\circ i$, where $\pi : P \to X$ is the structural
morphism of some projective space $P = \P^n_X$ over $X$, and $i$ is a regular
immersion of codimension $d = n-m$. Let $f' = \pi' \circ i'$ be the factorisation of
$f'$ defined by base change, with $\pi' : P' = \P^n_{X'} \to X'$, and let $w : P'
\to P$ be the projection.

The isomorphisms $\zeta'_{i,\pi}$ and $\zeta'_{i',\pi'}$ are clearly compatible with
the base change isomorphisms \eqref{pbomega} relative to $f$ and $u$, and the same
holds for the projection formula isomorphisms $\mu_{i,\omega_{P/X}[n]}$ and
$\mu_{i',\omega_{P'/X'}[n]}$, and the base change isomorphisms \eqref{pbomega} relative
to $i$ and $w$. Then, using Proposition \ref{TransTrs}, one sees that it suffices to
prove the proposition for $f = i$ and for $f = \pi$.

When $f = \pi : \P^n_X \to X$, let $X_0,\ldots,X_n$ be the canonical coordinates on
$\P^n_X$, and $x_i = X_i/X_0$, $1 \leq i \leq n$. If $\fU$ is the relatively affine
covering of $\P^n_X$ defined by $X_0,\ldots,X_n$, the corresponding alternating
\v{C}ech resolution provides a canonical isomorphism
\eq{trp}{f_*(\vC\hbul(\fU, \omega_{P/X})[n]) \riso \RR f_*(\omega_{P/X}[n]). }
Recall that $e_{\pi} : \pi^{\sharp} \cong \pi^!$ identifies the trace morphism for
projective spaces $\Trp_\pi$ with the general trace morphism $\Tr_\pi$ \cite[Lemma
3.4.3, (TRA3)]{Co00}. Then the commutativity of \eqref{bctrace} for $\pi$ follows from
the fact that, thanks to \eqref{ConradTrp}, $\Trp_{\pi}$ can be characterized as the
only morphism which, via \eqref{trp}, induces on $\sH^0$ the map sending the cohomology
class $dx_1\wedge \ldots \wedge dx_n / x_1\cdots x_n$ to $(-1)^{n(n-1)/2}$.

When $f = i : Y \hookrightarrow P$, recall that $d_i : i^\flat \cong i^!$ identifies
the trace morphism for finite morphisms $\Trf_i$ with the general trace morphism
$\Tr_i$ \cite[Lemma 3.4.3, (TRA2)]{Co00}, and that $\Trf_i : \RR\sHom_{\sO_P}(\sO_Y,
\sO_P) \to \sO_P$ is the canonical morphism induced by $\sO_P \surj
\sO_Y$. Using local cohomology with supports in $Y$, it can be factorized as
\eq{trviasupp}{\Trf_i : \RR\sHom_{\sO_P}(\sO_Y, \sO_P) \to 
\RR\uGamma_Y(\sO_P) \to \sO_P.}
On the other hand, there exists a canonical morphism
\eq{bcloccoh}{\LL w^*\RR\uGamma_Y(\sO_P) \lra \RR\uGamma_{Y'}(\sO_{P'}),}
which is an isomorphism: to check this, it suffices to choose a finite affine covering
$\fV$ of $V=P\setminus Y$, and to identify $\RR\uGamma_Y(\sO_P)$ with its flat
resolution provided by the total complex
\eqn{\sO_P \to j_*\vC(\fV,\sO_V),}
where $j$ denotes the inclusion of $V$ in $P$ and $\sO_P$ sits in
degree $0$. Moreover, this shows that the diagram
\eqn{\xymatrix{
\LL w^*\RR\uGamma_Y(\sO_P) \ar[d]_{\sim} \ar[r] & 
\LL w^*(\sO_P) \ar[d]^{\sim} \\
\RR\uGamma_{Y'}(\sO_{P'}) \ar[r] & \sO_{P'}
}}
commutes. Therefore, it suffices to prove the commutativity of the diagram 
\eq{bctri2}{\xymatrix@C=30pt{
\LL w^*i_*(\omega_{Y/P}[-d]) 
\ar[d]_{\sim} \ar[r] & \LL w^*\RR\uGamma_Y(\sO_P) \ar[d]^{\sim} \\
\omega_{Y'/P'}[-d] 
\ar[r]+<-29pt,0pt> & \hspace{9pt}\RR\uGamma_{Y'}(\sO_{P'})\hspace{6pt}.
}}
Since $Y' \inj P'$ is a regular immersion of codimension $d$, all complexes in this
diagram are acyclic except in degree $d$, so that, up to translation by $-d$, the
diagram is actually a diagram of morphisms of $\sO_{P'}$-modules. It follows that its
commutativity can be checked locally on $P'$. Thus we may assume that  
$P$ is affine, and that the ideal $\sI$ of $Y$ in $P$ is
generated by a regular sequence $t_1,\ldots,t_d$. Then the ideal $\sI'$ of $Y'$ in $P'$
is generated by the images $t'_1,\ldots,t'_d$ of $t_1,\ldots,t_d$ in $\sO_{P'}$, which
form a regular sequence. Let $\fV = (V_1,\ldots,V_d)$ be the open covering of
$P\setminus Y$ defined by the sequence $(t_1,\ldots,t_d)$. For any section $a \in 
\Gamma(P, \sO_P)$, let us still denote by $a/t_1\cdots t_d$ the image of 
$a/t_1\cdots t_d \in \Gamma(V_1\cap\ldots\cap V_d, \sO_P)$ under the 
canonical homomorphisms
\eqn{\Gamma(V_1\cap\ldots\cap V_d, \sO_P) \to H^{d-1}(P\setminus 
Y,\sO_P) \to H^d_Y(P, \sO_P) = \Gamma(P, \sH^d_Y(\sO_P)). }
Then the canonical morphism
\eqn{\omega_{Y/P} \riso \sExt^d_{\sO_P}(\sO_Y,\sO_P) \to 
\sH^d_Y(\sO_P)}
maps $(\bt_1^\vee\wedge\ldots\wedge\bt_d^\vee)\otimes a$ to $\varepsilon(d)a/t_1\cdots
t_d$, where $\varepsilon(d) \in \{\pm 1\}$ only depends upon $d$ (see
\cite[p.~252-254]{Co00}). The commutativity of \eqref{bctri2} follows.
\end{proof}

\begin{proof}[Proof of Case b)]
When $u$ is residually stable, the diagram analogous to \eqref{bctrace} commutes,
thanks to \cite[3.4.3, (TRA4)]{Co00}. Moreover, the isomorphisms $e_\pi$ and $d_i$
entering in the local definition of $\lambda_f$ in \ref{deflambda} c) also commute with
base change by $u$, thanks to \cite[p.~139, (VAR6)]{Co00}. Then it suffice to observe
that $\eta_i$ commutes with flat base change, which is clear.
\end{proof}

\begin{proof}[Proof of Case c)]
When $f$ is flat, $f$ is a CM map, and the results of \cite[3.5 - 3.6]{Co00} can be
applied. Then the commutativity of \eqref{bctrace} follows from \cite[Theorem
3.6.5]{Co00}, provided one checks that $\lambda_f$ identifies the base change
isomorphism \eqref{pbomega} for $\omega_{Y/X}$ with the more subtle base change
isomorphism $\beta_{f,u}$ for $\omega_f$ defined in \cite[Theorem 3.6.1]{Co00}. As we
will not use Case c) in this article, we leave the details to the reader.
\end{proof}

\subsection{}\label{Deftau}
Let $X$ be a noetherian scheme with a dualizing complex, and $f : Y \to X$ a proper
complete intersection morphism of virtual relative dimension $0$. One can define in
$\Dbcoh(X)$ a ``trace morphism''
\eq{tauf}{\tau_f : \RR f_*(\sO_Y) \lra \sO_X}
as follows. Thanks to the relative duality theorem (see \cite[VII, 3.4]{Ha66} or
\cite[Th.\ 3.4.4]{Co00}), defining $\tau_f$ is equivalent to defining a morphism $\sO_Y
\to f^!\sO_X$. Using the isomorphism $\lambda_f$, this is also equivalent to
defining a morphism
\eq{phif}{\varphi_f:\sO_Y \lra \omega_{Y/X},}
\ie a section of the invertible sheaf $\omega_{Y/X}$. We define $\varphi_f$ as being 
the morphism which maps $1$ to the canonical section $\delta_f$ of $\omega_{Y/X}$, 
defined in \ref{defdelta}. 

From this construction, it follows that the morphism $\tau_f$ can be described
equivalently either as the composition
\eq{phitau1}{\tau_f : \RR f_*(\sO_Y) \xrightarrow{\RR f_*(\lambda_f\circ\varphi_f)} \RR 
f_*(f^!\sO_X) \xrightarrow{\Tr_{f}} \sO_X,}
or as the composition
\eq{phitau2}{\tau_f : \RR f_*(\sO_Y) \xrightarrow{\RR f_*(\varphi_f)} \RR 
f_*(\omega_{Y/X}) \xrightarrow{\Trs_{f}} \sO_X,}
where $\Trs_f$ is the trace map defined in \eqref{trsharp}.

Before proving Theorem \ref{Thtau}, we relate $\tau_f$ to the residue symbol defined 
in \cite[(A.1.4)]{Co00} (which differs by a sign from Hartshorne's definition in 
\cite{Ha66}).

\begin{prop}\label{Taures}
With the hypotheses of \ref{Deftau}, assume in addition that $f$ is finite and flat,
and that $f = \pi \circ i$, where $\pi$ is smooth of relative dimension $d$, and $i$ is
a closed immersion, globally defined by a regular sequence $(t_1,\ldots,t_d)$ of
sections of $\sO_P$. Then, for any section $a$ of $\sO_P$, with reduction $\bar{a}$ on
$Y$, we have
\eq{taures}{ \tau_f(\bar{a}) = \Res_{P/X} \left[\begin{array}{c} 
a\,dt_1\wedge \ldots \wedge dt_d \\ t_1,\ldots,t_d 
\end{array}\right]. }
\end{prop}

\begin{proof}
Let $\omega = a\,dt_1\wedge\ldots\ldots\wedge dt_d$. By construction, the residue 
symbol is given by
\eq{ressymb}{ \Res_{P/X} \left[\begin{array}{c} 
\omega \\ t_1,\ldots,t_d 
\end{array}\right] = (-1)^{d(d-1)/2}\;\varphi_{\omega}(1), }
where $\varphi_{\omega} : f_*\sO_Y \to \sO_X$ is the image of 
$(t_1^{\vee}\wedge\ldots\wedge t_d^{\vee})\otimes i^*(\omega)$ by the 
isomorphism of complexes concentrated in degree $0$ \cite[(A.1.3)]{Co00}
\eq{resiso}{ \omega_{Y/P}[-d]\otimesL_{\sO_Y} \LL i^*(\omega_{P/X}[d]) 
\xrightarrow[\sim]{\eta_i^{-1}} i^{\flat}\pi^{\sharp}\sO_X 
\xrightarrow[\sim]{\psi_{i,\pi}^{-1}} f^\flat\sO_X; }
here $f^\flat\sO_X = \sHom_{\sO_X}(f_*\sO_Y,\sO_X)$ viewed as a $\sO_Y$-module, and
$\psi_{i,\pi}$ is the canonical isomorphism of functors $f^\flat \riso
i^\flat\pi^\sharp$. Since $\Trf_{f}$ is the morphism
\linebreak$f_*\sHom_{\sO_X}(f_*\sO_Y,\sO_X) \to \sO_X$ given by evaluation at $1$, we
can use the isomorphism $d_f : f^\flat \riso f^!$ and the equality $\Tr_f\circ f_*(d_f)
= \Trf_f$ \cite[3.4.3, (TRA2)]{Co00} to write
\eq{res}{ \Res_{P/X} \left[\begin{array}{c} 
\omega \\ t_1,\ldots,t_d \end{array}\right] = 
(-1)^{d(d-1)/2}\;\Tr_f(f_*(d_f\circ\psi_{i,\pi}^{-1}\circ\eta_i^{-1})
(t_1^{\vee}\wedge\ldots\wedge t_d^{\vee}\otimes i^*(\omega))). 
}

On the other hand, we have by definition
\eqn{\zeta'_{i,\pi}(\delta_f) = (-1)^{d(d-1)/2}\;
t_1^{\vee}\wedge\ldots\wedge t_d^{\vee} \otimes 
i^*(dt_1\wedge\ldots\ldots\wedge dt_d),}
so we deduce from \eqref{phitau1} the equality
\begin{eqnarray*} 
\tau_f(\bar{a}) & = & \Tr_f(f_*(\lambda_f\circ\varphi_f)(\bar{a})) \\
& = & (-1)^{d(d-1)/2}\;\Tr_f(f_*(c_{i,\pi}^{-1}\circ i^!(e_{\pi})\circ d_i \circ 
\eta_i^{-1})(t_1^{\vee}\wedge\ldots\wedge t_d^{\vee}\otimes i^*(\omega))).
\end{eqnarray*}
Therefore, it suffices to check that
\eqn{d_f\circ\psi_{i,\pi}^{-1} = c_{i,\pi}^{-1}\circ i^!(e_{\pi})\circ d_i,}
and this results from (VAR5) \cite[(3.3.26)]{Co00}.
\end{proof}

\subsection{}\label{Prooftau}
\textit{Proof of Theorem \ref{Thtau}.}

\romain The transitivity formula \eqref{transtau} is the equality of the exterior 
composites in the diagram
\eqn{ \xymatrix@C=58pt{
\RR f_* \RR g_* \sO_Z \ar[r]^-{\RR f_* \RR g_*(\varphi_g)} 
\ar[d]!<0ex,2ex>^{\RR f_* \RR g_*(\varphi_{fg})}
& \RR f_* \RR g_* \omega_{Z/Y} \ar@/^2.5pc/[rddd]^-{\RR f_*(\Trs_g)}
\ar[d]!<0ex,3.5ex>_(.65){\RR f_* \RR g_*(\Id\otimes\LL g^*(\varphi_{f}))}\\
\entry{\RR f_* \RR g_* \omega_{Z/X}} \ar[r]^-{\RR f_* \RR g_*(\zeta'_{g,f})}
\ar[dd]^-{\Trs_{fg}}
& \entry{\RR f_* \RR g_*(\omega_{Z/Y} \otimesL_{\sO_Z} \LL g^*\omega_{Y/X})}
\ar[d]_-{\sim}\\
 & \RR f_*((\RR g_*\omega_{Z/Y}) \otimesL_{\sO_Y} \omega_{Y/X})
\ar[d]_-{\RR f_*(\Trs_g\otimes\Id)} \\
 \sO_X & \RR f_*\omega_{Y/X} \ar[l]_-{\Trs_f} & 
\RR f_*\sO_Y, \ar[l]_-{\RR f_*(\varphi_f)}
}} 
where the upper left square commutes thanks to \eqref{transdelta}, the lower left 
square is the commutative square \eqref{transTrs}, and the right triangle commutes by 
functoriality.

\romain Thanks to Proposition \ref{basechangeTrs} and to the description \eqref{phitau2}
of $\tau_f$, the assertion follows from the compatibility of the canonical section
$\delta_f$ with Tor-independent pull-backs (proved in Proposition \ref{propdelta} (ii))
and the functoriality of the base change morphism.

\romain To prove \eqref{tautrace}, it suffices to prove that the equality holds in the
henselization $\sO_{X,x}^{\mathrm{h}}$ of the local ring of $X$ at each point $x$. As
the morphism $\Spec\sO_{X,x}^{\mathrm{h}} \to X$ is residually stable
\cite[p.~132]{Co00}, Proposition \ref{basechangeTrs} and the commutation with base
change of the classical trace map for the finite locally free algebra $f_*\sO_Y$ allow
to assume that $X = \Spec A$, where $A$ is a henselian noetherian local ring. Then $Y$
is a disjoint union of open subschemes $Y_i = \Spec B_i$, where $B_i$ is a finite local
algebra over $A$. Each of the morphisms $Y_i \to X$ is a complete intersection morphism
of virtual relative dimension $0$ (since this is a local condition on $Y$), and the
additivity of the trace (valid both for $\Tr_f$, hence for $\tau_f$, and for
$\trace_{f_*\sO_Y/\sO_X}$) shows that it suffices to prove \eqref{tautrace} for each
morphism $Y_i \to X$. So we may assume that $B$ is local. We can choose a presentation
$B \cong C/I$, where $C$ is a smooth $A$-algebra, and $I$ is an ideal in $C$. Let $P =
\Spec C$, $\sI = I\sO_P$, and let $y \in Y \subset P$ be the closed
point. Then $\sI_y$ is generated by a regular sequence $(t_1,\ldots,t_d)$.
Shrinking $P$ if necessary, we may assume that $t_1,\ldots,t_d$ generate $\sI$ globally
on $P$, so that the hypotheses of \ref{Taures} are satisfied. Then \eqref{tautrace}
follows from \eqref{taures} and from property (R6) of the residue symbol
\cite[p.~240]{Co00}.
\hfill$\Box$

\end{document}